\documentclass[12pt,a4paper,twoside,reqno]{amsart}

\usepackage{amsmath, amssymb, amsthm, enumerate, enumitem}
\usepackage[normalem]{ulem}
\usepackage{xr}
\makeatletter
\def\l@section{\@tocline{1}{10pt}{1pc}{}{}}
\def\l@subsection{\@tocline{2}{0pt}{1pc}{4.6em}{}}
\def\l@subsubsection{\@tocline{3}{0pt}{1pc}{7.6em}{}}
\renewcommand{\tocsection}[3]{%
  \indentlabel{\@ifnotempty{#2}{\makebox[2.3em][l]{%
    \ignorespaces#1 #2.\hfill}}}\textbf{#3}}
\renewcommand{\tocsubsection}[3]{%
  \indentlabel{\@ifnotempty{#2}{\hspace*{2.3em}\makebox[2.3em][l]{%
    \ignorespaces#1 #2.\hfill}}}#3}
\renewcommand{\tocsubsubsection}[3]{%
  \indentlabel{\@ifnotempty{#2}{\hspace*{4.6em}\makebox[3em][l]{%
    \ignorespaces#1 #2.\hfill}}}#3}
\makeatother

\newcommand{\MM}{\mathcal{M}}
\newcommand{\CC}{\mathcal{C}}
\newcommand{\IR}{\mathbb{R}}
\newcommand{\IB}{\mathbb{B}}

\newcommand{\LL}{\mathcal{L}}

\newcommand{\eps}{\varepsilon}
\newcommand{\ov}[1]{\overline{#1}}
\newcommand{\un}[1]{\underline{#1}}
\newcommand{\td}[1]{\widetilde{#1}}

\DeclareMathOperator{\thick}{thick}
\DeclareMathOperator{\thin}{thin}

\DeclareMathOperator{\stan}{stan}
\DeclareMathOperator{\scal}{scal}

\DeclareMathOperator{\hyp}{hyp}
\DeclareMathOperator{\Ric}{Ric}

\DeclareMathOperator{\dist}{dist}

\DeclareMathOperator{\diam}{diam}

\DeclareMathOperator{\vol}{vol}
\DeclareMathOperator{\const}{const}

\DeclareMathOperator{\Rm}{Rm}

\newcommand{\EMPTY}[1]{}

\newtheorem{Theorem}{Theorem}[section]
\newtheorem{Lemma}[Theorem]{Lemma}
\newtheorem{Corollary}[Theorem]{Corollary}
\newtheorem{Proposition}[Theorem]{Proposition}
\newtheorem{Definition}[Theorem]{Definition}

\numberwithin{equation}{section}
\newtheorem*{Claim}{Claim}

\newtheorem*{Claim1}{Claim 1}
\newtheorem*{Claim2}{Claim 2}

\theoremstyle{definition}
\newcommand{\blocktext}{aaa}
\newtheorem*{blocktheorem}{\blocktext}
\newenvironment{block}[1]{\renewcommand{\blocktext}{#1} \begin{blocktheorem}}{\end{blocktheorem}}

\title[Long-time behavior of 3d Ricci flow --- A]{Long-time behavior of 3 dimensional Ricci flow\\A: Generalizations of Perelman's long-time estimates}
\author{Richard H Bamler}
\address{UC Berkeley, Department of Mathematics, 970 Evans Hall, Berkeley, CA 94720, USA}
\email{rbamler@math.berkeley.edu}
\date{\today}

\begin{document}
\begin{abstract}
This is the first of a series of papers on the long-time behavior of 3 dimensional Ricci flows with surgery.
In this paper we first fix a notion of Ricci flows with surgery, which will be used in this and the following three papers.
Then we review Perelman's long-time estimates and generalize them to the case in which the underlying manifold is allowed to have a boundary.
Eventually, making use of Perelman's techniques, we prove new long-time estimates, which hold whenever the metric is sufficiently collapsed.
\end{abstract}

\maketitle
\tableofcontents

\section{Introduction}
In this paper, we first introduce a notion of Ricci flows with surgery, which will be used throughout the whole series of papers (see section \ref{sec:IntroRFsurg}).
We will also mention known existence and extension results for such Ricci flows with surgery.
Then we review the long-time estimates of Perelman (cf \cite{PerelmanII}) using our notion of Ricci flows with surgery (see section \ref{sec:Perelman}).
For future purposes we include in this discussion the case in which the underlying manifold is allowed to be non-compact or has a boundary.
The estimates in this more general case are of independent interest.
Eventually, we derive new long-time estimates using Perelman's techniques, which hold under certain collapsing conditions (see section \ref{sec:maintools}).
Those estimates will be used in \cite{Bamler-LT-main}.

In the following we will outline the results of this paper.
For a shorter summary we refer to subsection \ref{subsec:outlineofintroduction} of \cite{Bamler-LT-Introduction}, where these results are also explained within the context of the whole series of papers.
All results of this paper will be used to describe (3 dimensional) Ricci flows with surgery $\MM$ at large times $t$.
For a precise definition of Ricci flows with surgery, we refer to the subsequent section \ref{sec:IntroRFsurg}.
In this introduction we assume for simplicity that $\MM$ is non-singular, i.e. that it is given by a smooth family of Riemannian metrics $(g_t)_{t \in [0, \infty)}$ on a $3$-manifold $M$ that satisfies the Ricci flow equation $\partial_t g_t = - 2 \Ric_{g_t}$.

Our first two results are generalizations of results of Perelman for the case in which the underlying manifold $M$ is non-compact and/or has a boundary.

\begin{block}{Result I: Non-collapsedness controls curvature (Proposition \ref{Prop:genPerelman}, subsection \ref{subsec:Perelmanlongtime})}
This result is a generalization of a celebrated theorem of Perelman.
It roughly states the following:

 \textit{For every $w > 0$ there are constants $\ov{r} = \ov{r}(w) >0$ and $K = K(w) < \infty$ such that if $(x_0,t_0) \in M \times [0, \infty)$ and $0 < r_0 < \ov{r} \sqrt{t_0}$, then the the following holds: If the time-$t_0$ volume of the time-$t_0$ ball $B(x_0, t_0, r_0)$ is greater than $w r_0^3$ and the time-$t_0$ sectional curvatures are bounded from below by $- r_0^{-2}$ on this ball, then $|{\Rm_{t_0}}| < K r_0^{-2}$ on $B(x_0,t_0, r_0)$.}
 
We will generalize this result to the case in which $M$ is non-compact and/or has a boundary.
It will turn out that if the boundary of $M$ stays sufficiently far away from $x_0$ on a time-interval of the form $[t_0- \frac1{10} r_0^2, t_0]$, then the same estimate holds.
\end{block}

\begin{block}{Result II: Bounded curvature at bounded distance from non-collapsed regions (Lemma \ref{Lem:6.3bc}, subsection \ref{subsec:Perelmanlongtime})}
This result is a generalization of another result of Perelman, which can be summarized as follows:

\textit{For every $A < \infty$ there are constants $\ov{r} = \ov{r}(A) > 0$ and $K = K(A) < \infty$ such that if $(x_0,t_0) \in M \times [0, \infty)$ and $0 < r_0 < \ov{r} \sqrt{t_0}$, then we can make the following conclusion: If $|{\Rm_{t_0}}| \leq r_0^{-2}$ on the parabolic neighborhood $P(x_0, t_0, r_0, - r_0^2) = B(x_0, t_0, r_0) \times [t_0 - r_0^2, t_0]$, and $\vol_{t_0} B(x_0, t_0, r_0) > A^{-1} r_0^3$, then $|{\Rm_t}| < K r_0^{-2}$ on $B(x_0, t_0, A r_0)$.}

This result is an ingredient for the proof of Result I.
\end{block}

The next 5 results characterize the Ricci flow in regions that are collapsed, but that become non-collapsed when we pass to the universal, or a local cover of $M$.
By this we mean the following:
Consider the universal cover $\pi : \td{M} \to M$ of $M$ and the pull-backs $\td{g}_t = \pi^* g_t$ of the Riemannian metrics $g_t$ on $\td{M}$.
Then $(\td{g}_t)_{t \in [0, \infty)}$ is a solution to the Ricci flow on $\td{M}$.
Let $x \in M$ and consider a lift $\td{x} \in \td{M}$ of $x$.
Then the volume of a ball $B^{\td{M}} (\td{x}, t, r)$ around $\td{x}$ in $(\td{M}, \td{g}_t)$ is not smaller than the volume of the corresponding ball $B(x,t,r)$ in $(M, g_t)$.
In fact, the restriction
\[ \pi |_{B^{\td{M}} (\td{x}, t, r)} : B^{\td{M}} (\td{x}, t, r) \to B(x,t,r) \]
is a surjective local diffeomorphism.
Note that the volume of $B^{\td{M}} (\td{x}, t, r)$ can be much larger than that of $B(x,t,r)$, for example if $B(x,t,r)$ collapses along \emph{incompressible} (i.e. $\pi_1$-injective) $S^1$ or $T^2$-fibers.
On the other hand, if this collapse occurs along $S^1$ or $T^2$-fibers that are compressible in $M$, but incompressible in some subset $U \subset M$, then the volume of $B^{\td{M}} (\td{x}, t, r)$ might be comparable to that of $B(x,t,r)$, but the volume of the ball $B^{\td{U}} (\td{x}', t, r)$ around a lift $\td{x}'$ of $x$ in the universal cover $\td{U} \to U$ of $U$ will be much larger.
It can be seen that the volume of $B^{\td{U}} (\td{x}', t, r)$ is monotone in $U$ (with respect to inclusion), the largest volume being achieved if we choose $U = B(x,t,r)$.
We refer to subsection \ref{subsec:outlineofintroduction} of \cite{Bamler-LT-Introduction} for a further discussion of collapsing behaviors within the context of this series of papers.

We will now use the following terminology; for more details see Definition \ref{Def:goodness}:
For every point $x \in M$ and time $t \geq 0$, we first fix a local scale $\rho(x,t) > 0$, which roughly measures how large the negative sectional curvatures are in a neighborhood around $x$ (for more details see Definition \ref{Def:rhoscale}).
Then we call a point $x \in M$ \emph{good} if it is non-collapsed in the universal cover $\td{M}$ of $M$, i.e. if the volume of $B^{\td{M}}(\td{x},t,\rho(x,t))$ is larger than $c \rho^3(x,t)$ for some uniform $c > 0$.
If the volume of the ball $B^{\td{U}}(\td{x}',t,\rho(x,t))$ is larger than $c \rho^3(x,t)$ for some subset $U \subset M$, then we say that $x$ is \emph{good relative to $U$}.
Finally, if $x$ is good relative to $U = B(x,t,\rho(x,t))$, then we call $x$ \emph{locally good}.

We can now state the next 5 results:

\begin{block}{Result III: Bounded curvature around good points (Proposition \ref{Prop:curvcontrolgood}, subsection \ref{subsec:boundedcurvaroundgoodpts})}
This result can be summarized as follows:

\textit{For every $w > 0$ there are $\ov{r} = \ov{r} (w) > 0$ and $K = K(w) < \infty$ such that:
For every $(x_0, t_0) \in M \times [0,\infty)$ and $0 < r_0 < \sqrt{t_0}$ we have:
If $\vol_{t_0} B^{\td{M}} (\td{x}_0, t_0, r_0) > w r_0^3$ and if the time-$t_0$ sectional curvatures are bounded from below by $-r_0^{-2}$ on $B(x_0, t_0, r_0)$, then $|{\Rm_{t_0}}| < K r_0^{-2}$ on that ball.}

This result will follow from Result I applied to the universal covering flow $(\td{M}, \td{g}_t)$.
\end{block}

\begin{block}{Result IV: Bounded curvature at bounded distance from sufficiently collapsed and good regions (Proposition \ref{Prop:curvcontrolincompressiblecollapse}, subsection \ref{subsec:boundedcurvboundeddistgoodregions})}
This result can be interpreted as a variation of Result II in the collapsed case.
It reads:

\textit{For every $A < \infty$ there are constants $\ov{w} = \ov{w}(A) > 0$ and $K = K(A) < \infty$ such that for every $(x_0,t_0) \in M \times [0, \infty)$ and $0 < r_0 < \sqrt{t_0}$ we have: If $|{\Rm_{t_0}}| \leq r_0^{-2}$ on the parabolic neighborhood $P(x_0, t_0, r_0, - r_0^2) = B(x_0, t_0, r_0) \times [t_0 - r_0^2, t_0]$, $\vol_{t_0} B^{\td{M}} (\td{x}_0, t_0, r_0) > A^{-1} r_0^3$ and $\vol_{t_0} B(x_0, t_0, r_0) < \ov{w} r_0^3$, then $|{\Rm_{t_0}}| < K r_0^{-2}$ on $B(x_0, t_0, A r_0)$.}

Note that we did not need to assume that $r_0 < \ov{r} \sqrt{t_0}$ for some $\ov{r} = \ov{r}(A) > 0$ as in Result II.
This difference will be essential for us.
So Result IV is not strictly a generalization of Result II and does not directly follow from Result II by passing to the universal cover.
Instead, the proof of this result makes use of the fact that $B(x_0, t_0, r_0)$ is sufficiently collapsed.
\end{block}

\begin{block}{Result V: Curvature control at points that are good relative to regions whose boundary is geometrically controlled (Proposition \ref{Prop:curvboundinbetween}, subsection \ref{subsec:curvboundinbetween})}
We next consider a subset $U \subset M$ and a point $x_0 \in U$ that is good relative to $U$.
We then obtain a generalization of Result III:

\textit{For every $w > 0$ there are constants $\ov{r} = \ov{r} (w) > 0$ and $K = K(w) < \infty$ such that:
For every $(x_0, t_0) \in U \times [0, \infty)$ and $0 < r_0 < \ov{r}(w) \sqrt{t_0}$ we have:
If $\vol_{t_0} B^{\td{U}}(\td{x}'_0, t_0, r_0) > w r_0^3$ and the sectional curvatures are bounded from below by $- r_0^{-2}$ on $B(x_0, t_0, r_0)$ and if $|{\Rm}| < r_0^{-2}$ on $P(x,t_0, r_0, - r_0^2)$ for all $x \in \partial U$, then $|{\Rm_{t_0}}| < K r_0^{-2}$ on $B(x_0, t_0, r_0)$.}

The idea of the proof will be that under these assumptions the boundary of $U$ stays far enough away from $x_0$ for all times of $[t_0 - r_0^2, t_0]$ if it is far enough away at time $t_0$.
This fact will enable us to localize the arguments in the proof of Result III.
\end{block}

\begin{block}{Result VI: Controlled diameter growth of regions whose boundary is sufficiently collapsed and good (Proposition \ref{Prop:slowdiamgrowth}, subsection \ref{subsec:controlleddiamgrowth})}
We will next control the diameter growth of a subset $U \subset M$ under the Ricci flow, only based on geometric control around its boundary and a diameter bound at early times.
In rough terms, our statement will be:

\textit{For every $A < \infty$ there are $\ov{w} = \ov{w} (A) > 0$ and $A' = A'(A), K = K(A) < \infty$ such that:
Assume that $0 < r_0 < \sqrt{t_0}$ and $x_0 \in M$.
Assume that at time $t_0 - r_0^2$ the subset $U$ has bounded diameter: $U \subset B(x_0, t_0 - r_0^2, A r_0)$ and assume that the boundary of $U$ stays within controlled distance to $x_0$ for some time, so that $\partial U \subset B(x_0, t, A r_0)$ for all $t \in [t_0 - r_0^2, t_0]$.
Then if $\vol_{t_0} B^{\td{M}} (\td{x}_0, t_0, r_0) > A^{-1} r_0^3$, $\vol_{t_0} B(x_0, t_0, r_0) < \ov{w} r_0^3$ and the sectional curvatures on $B(x_0, t_0, r_0)$ are bounded from below by $- r_0^{-2}$, we have $U \subset B(x_0, t, A' r_0)$ for all $t \in [t_0 - r_0^2, t_0]$.
Moreover, $|{\Rm}| < K r_0^{-2}$ on $U \times [t_0 - r_0^2, t_0]$.}
\end{block}

\begin{block}{Result VII: Curvature control in large regions that are locally good everywhere (Proposition \ref{Prop:curvboundnotnullinarea}, subsection \ref{subsec:curvcontrollocgoodeverywhere})}
In the last result we derive a curvature bound assuming only local goodness.
In order to achieve this bound, we must however assume that the local goodness holds in a sufficiently large region and also at earlier times:

\textit{For all $w > 0$ there is a constant $K = K(w) < \infty$ such that the following holds:
Let $x_0 \in U \subset M$ and $0 < r_0 < \sqrt{t_0}$ and $b > 0$ and assume that $|{\Rm}| < r_0^{-2}$ on $P(x,t_0, r_0, - r_0^2)$ for all $x \in \partial U$.
Assume moreover that for every $t \in [t_0 - r_0^2, t_0]$, $x \in B(x_0, t, b) \cap U$ and every $0 < r < r_0$ for which $B(x,t,r) \subset U$ and for which the sectional curvatures on $B(x, t, r)$ are bounded from below by $-r^{-2}$ we have
\[ \vol_t B^{\td{B(x,t,r)}} (\td{x}', t, r) > w r^{3}, \]
where $\td{B(x,t,r)}$ is the universal cover of $B(x,t,r)$ and $\td{x}' \in \td{B(x,t,r)}$ is a lift of $x \in B(x,t,r)$. 
Then $|{\Rm_{t_0}}| < K r_0^{-2}$ on $U \cap B(x_0, t_0, b - r_0)$.}
\end{block}

We refer to \cite{Bamler-LT-Introduction} for historical remarks and acknowledgements.

Note that in the following all manifolds are always assumed to be orientable and 3 dimensional, unless stated otherwise.

\section{Introduction to Ricci flows with surgery} \label{sec:IntroRFsurg}
\subsection{Definition of Ricci flows with surgery} \label{sec:DefRFsurg}
In this section, we give a precise definition of the Ricci flows with surgery that we are going to analyze subsequently.
We will mainly use the language developed in \cite{Bamler-diploma} here.
In a first step, we define Ricci flows with surgery in a very broad sense.
After explaining some useful notions, we will make precise how we assume that the surgeries are performed.
This characterization can be found in Definition \ref{Def:precisecutoff}.
We have chosen a phrasing that unifies most of the common constructions of Ricci flows with surgery, such as those presented in \cite{PerelmanII}, \cite{KLnotes}, \cite{MTRicciflow}, \cite{BBBMP} and \cite{Bamler-diploma}.
Hence the main Theorems \ref{Thm:MainTheorem-III} and \ref{Thm:geombehavior} of \cite{Bamler-LT-main} can be applied to the Ricci flows with surgery that were constructed in each of these publications.

\begin{Definition}[Ricci flow with surgery] \label{Def:RFsurg}
Consider a time-interval $I \subset \IR$.
Let $T^1 < T^2 < \ldots$ be times of the interior of $I$ that form a possibly infinite, but discrete subset of $\IR$ and divide $I$ into the intervals
\[ I^1 = I \cap (-\infty, T^1), \quad I^2 = [T^1, T^2), \quad I^3 = [T^2, T^3), \quad \ldots \]
and $I^{k+1} = I \cap [T^k,\infty)$ if there are only finitely many $T^i$'s and $T^k$ is the last such time and $I^1 = I$ if there are no such times.
Consider Ricci flows $(M^1 \times I^1, g^1_t), (M^2 \times I^2, g^2_t), \ldots$ on manifolds $M^1, M^2, \ldots$, which may have a boundary, and time-intervals $I^1, I^2, \ldots$ (i.e. $\partial_t g_t^i = - 2 \Ric_{g^i_t}$ for each $i$).
Let $\Omega^i \subset M^i$ be open sets on which the metric $g^i_t$ converges smoothly as $t \nearrow T^i$ to some Riemannian metric $g^i_{T^i}$ on $\Omega_i$ and let 
\[ U^i_- \subset \Omega^i \qquad \text{and} \qquad U^i_+ \subset M^{i+1} \]
be open subsets such that there are isometries
\[ \Phi^i : (U^i_-, g_{T^i}^i) \longrightarrow (U^i_+, g_{T^i}^{i+1}), \qquad (\Phi^i)^* g_{T^i}^{i+1} |_{U_+^i} = g^i_{T^i} |_{U^i_-}. \]
We assume moreover that we never have $U_-^i = \Omega^i = M^i$ and $U_+^i = M^{i+1}$ and that every component of $M^{i+1}$ contains a point of $U^i_+$.
Then, we call $\MM = ((T^i), (M^i \times I^i, g_t^i), (\Omega^i), (U^i_{\pm}), (\Phi^i))$ a \emph{Ricci flow with surgery on the time-interval $I$} and the times $T^1, T^2, \ldots$ \emph{surgery times}.

If $t \in I^i$, then $(\MM(t), g(t)) = (M^i \times \{ t \}, g_t^i)$ is called the \emph{time-$t$ slice of $\MM$}.
The points in $\MM(T^i) \setminus U^i_+ \times \{T^i \}$ are called \emph{surgery points}.
For $t = T^i$, we define the \emph{(presurgery) time $T^{i-}$-slice} to be $(\MM(T^{i-}), g(T^{i-})) = (\Omega^i \times \{ T^i \}, g^i_{T^i})$.
The points $\Omega^i \times \{ T^i \} \setminus U^i_- \times \{ T^i \}$ are called \emph{presurgery points}.

If $\MM$ has no surgery points, then we call $\MM$ \emph{non-singular} and write $\MM = M \times I$.
\end{Definition}

We will often view $\MM$ in the \emph{space-time picture}, i.e. we imagine $\MM$ as a topological space $\bigcup_{t \in I} \MM (t) = \bigcup_i M^i \times I^i$ where the components in the latter union are glued together via the diffeomorphisms $\Phi^i$.

The following vocabulary will prove to be useful when dealing with Ricci flows with surgery:

\begin{Definition}[Ricci flow with surgery, space-time curve]
Consider a sub-interval $I' \subset I$.
A map $\gamma : I' \to \bigcup_{t \in I'} \MM(t)$ (also denoted by $\gamma : I' \to \MM$) is called a \emph{space-time curve} if $\gamma(t) \in \MM(t)$ for all $t \in I'$, if $\gamma$ restricted to each sub-time-interval $I^i$ is continuous and if $\lim_{t \nearrow T^i} \gamma(t) \in U^i_-$ and $\gamma(T^i) = \Phi^i(\lim_{t \nearrow T^i} \gamma(t))$ for all $i$.
\end{Definition}

So a space-time curve is a continuous curve in $\MM$ in the space-time picture that is parameterized by the time function.

\begin{Definition}[Ricci flow with surgery, points in time] \label{Def:pointsurvives}
For $(x,t) \in \MM$, consider a spatially constant space-time curve $\gamma$ in $\MM$ that starts or ends in $(x,t)$ and exists forwards or backwards in time for some duration $\Delta t \in \IR$ and that doesn't hit any (pre-)surgery points except possibly at its endpoints.
Then we say that the point $(x,t)$ \emph{survives until time} $t + \Delta t$ and we denote the other endpoint by $(x, t + \Delta t)$.

Observe that this notion also makes sense, if $(x, t^-) \in \MM$ is a presurgery point and $\Delta t \leq 0$.
\end{Definition}

Note that the point $(x, t+\Delta t)$ is only defined if $(x, t)$ survives until time $t + \Delta t$, which entails that $\MM$ is defined at time $t + \Delta t$.
Using the previous definition, we can define parabolic neighborhoods in $\MM$.

\begin{Definition}[Ricci flow with surgery, parabolic neighborhoods] \label{Def:parabnbhd}
Let $(x,t) \in \MM$, $r \geq 0$ and $\Delta t \in \IR$.
Consider the ball $B = B(x,t,r) \subset \MM(t)$.
For each $(x',t) \in B$ consider the union $I^{\Delta t}_{x',t}$ of all points $(x',t+t') \in \MM$ that are well-defined in the sense of Definition \ref{Def:pointsurvives} for $t' \in [0, \Delta t]$ resp. $t' \in [\Delta t, 0]$.
Define the \emph{parabolic neighborhood} $P(x,t,r,\Delta t) = \bigcup_{x' \in B} I^{\Delta t}_{x',t}$.
We call $P(x,t,r,\Delta t)$ \emph{non-singular} if all points in $B(x,t,r)$ survive until time $t+ \Delta t$.
\end{Definition}

The following notion will be used in section \ref{sec:maintools} and in \cite{Bamler-LT-main}.
\begin{Definition}[sub-Ricci flow with surgery] \label{Def:subRF}
Consider a Ricci flow with surgery $\MM = ((T^i), (M^i \times I^i, g_t^i), (\Omega^i), (U^i_{\pm}), (\Phi^i))$ on the time-interval $I$.
Let $I' \subset I$ be a sub-interval and consider the indices $i$ for which the intervals ${I'}^i = I^i \cap I'$ are non-empty.
For each such $i$ consider a submanifold ${M'}^i \subset M^i$ of the same dimension and possibly with boundary.
Let ${g'_t}^i$ be the restriction of $g^i_t$ onto ${M'}^i \times {I'}^i$ and set ${\Omega'}^i = \Omega^i \cap {M'}^i$ and ${U'_-}^i = U^i_- \cap {M'}^i$ as well as ${U'_+}^i = U^i_+ \cap {M'}^{i+1}$.
Assume that for each $i$ for which ${I'}^i$ and ${I'}^{i+1}$ are non-empty, we have $\Phi^i ({U'_-}^i) = {U'_+}^i$ and let ${\Phi'}^i$ be the restriction of $\Phi^i$ to ${U'_-}^i$.

In the case in which ${U'_-}^i = {\Omega'}^i = {M'}^i$ and ${U'_+}^i = {M'}^{i+1}$ for some $i$, we can combine the Ricci flows ${g'_t}^i$ and ${g'_t}^{i+1}$ on ${M'}^i \times {I'}^i$ and ${M'}^{i+1} \times {I'}^{i+1}$ to a Ricci flow on the time-interval ${I'}^i \cup {I'}^{i+1}$ and hence remove $i$ from the list of indices.

Then $\MM' = (({T'}^i), ({M'}^i \times {I'}^i, {g'_t}^i), ({\Omega'}^i), ({U'_{\pm}}^i), ({\Phi'}^i))$ is a Ricci flow with surgery in the sense of Definition \ref{Def:RFsurg}.

Assume that for all $t \in I'$ the boundary points $\partial \MM' (t) \subset \MM(t)$ (by this we mean all points in $\MM(t)$ that don't lie in the interior of $\MM'(t)$ or $\MM(t) \setminus \MM'(t))$ survive until any other time of $I'$ and that $\partial \MM' (t)$ is constant in $t$.
Then we call $\MM'$ a \emph{sub-Ricci flow with surgery} and we write $\MM' \subset \MM$.
\end{Definition}

We will now characterize three important local, approximate geometries, which we will frequently be dealing with: $\varepsilon$-necks, strong $\varepsilon$-necks and $(\varepsilon, E)$-caps.
The notions below also make sense for presurgery time-slices.

\begin{Definition}[Ricci flow with surgery, $\varepsilon$-necks] \label{Def:epsneck}
Let $\varepsilon > 0$ and consider a Riemannian manifold $(M,g)$.
We call an open subset $U \subset M$ an $\varepsilon$-neck, if there is a diffeomorphism $\Phi : S^2 \times (-\frac1{\varepsilon}, \frac1{\varepsilon}) \to U$ such that there is a $\lambda > 0$ with $\Vert \lambda^{-2} \Phi^* g - g_{S^2 \times \IR} \Vert_{C^{[\varepsilon^{-1}]}} < \varepsilon$, where $g_{S^2 \times \IR}$ is the standard round metric on $S^2 \times (-\frac1{\varepsilon}, \frac1{\varepsilon})$ of constant scalar curvature $2$.

We say that $x \in U$ is a \emph{center} of $U$ if $x \in \Phi(S^2 \times \{0\})$ for such a $\Phi$.

If $\MM$ is a Ricci flow with surgery and $(x,t) \in \MM$, then we say that \emph{$(x,t)$ is a center of an $\varepsilon$-neck} if $(x,t)$ is a center of an $\varepsilon$-neck in $\MM(t)$.
\end{Definition}

\begin{Definition}[Ricci flow with surgery, strong $\varepsilon$-necks]
Let $\varepsilon > 0$ and consider a Ricci flow with surgery $\MM$ and a time $t_2$.
Consider a subset $U \subset \MM(t_2)$ and assume that all points of $U$ survive until some time $t_1 <  t_2$.
Then the subset $U \times [t_1,t_2] \subset \MM$ is called a \emph{strong $\varepsilon$-neck} if there is a factor $\lambda > 0$ such that after parabolically rescaling by $\lambda^{-1}$, the flow on $U \times [t_1,t_2]$ is $\varepsilon$-close to the standard flow on $[-a,0]$ for $a \geq 1$.
By this we mean $a = \lambda^{-2} (t_2 - t_1) \geq 1$ and there is a diffeomorphism $\Phi : S^2 \times (- \frac1\varepsilon, \frac1\varepsilon ) \to U$ such that
\[ \Vert \lambda^{-2} \Phi^* g( \lambda^2 t + t_2) - g_{S^2 \times \IR} (t) \Vert_{C^{[\varepsilon^{-1}]}( S^2 \times ( - \frac1\varepsilon, \frac1\varepsilon ) \times [-a, 0])} < \varepsilon. \]
Here $(g_{S^2 \times \IR}(t))_{t \in (-\infty,0]}$ is the standard Ricci flow on $S^2 \times \IR$ that has constant scalar curvature $2$ at time $0$ and $\lambda^{-2} \Phi^*  g (\lambda^2 t + t_2)$ denotes the pull-back of the parabolically rescaled flow on $U \times [t_1, t_2]$.

A point $(x, t_2) \in U \times \{ t_2 \}$ is called a \emph{center of $U \times [t_1, t_2]$} if $(x, t_2) \in \Phi ( S^2 \times \{ 0 \} \times \{ t_2 \} )$ for such a $\Phi$.
\end{Definition}

\begin{Definition}[Ricci flow with surgery, $(\varepsilon, E)$-caps] \label{Def:epscap}
Let $\varepsilon, E > 0$ and consider a Riemannian manifold $(M, g)$ and an open subset $U \subset M$.
Suppose that $(\diam U)^2 |{\Rm}|(y) < E^2$ for any $y \in U$ and $E^{-2} |{\Rm}|(y_1) \leq |{\Rm}|(y_2) \leq E^2 |{\Rm}|(y_1)$ for any $y_1, y_2 \in U$.
Furthermore, assume that $U$ is either diffeomorphic to $\IB^3$ or $\IR P^3 \setminus \ov{\IB}^3$ and that there is a compact set $K \subset U$ such that $U \setminus K$ is an $\varepsilon$-neck.

Then $U$ is called an \emph{$(\varepsilon, E)$-cap}.
If $x \in K$ for such a $K$, then we say that $x$ is a \emph{center of $U$}.

Analogously as in Definition \ref{Def:epsneck}, we define $(\varepsilon, E)$-caps in Ricci flows with surgery.
\end{Definition}

With these concepts at hand we can soon give an exact description of the surgery process that will be assumed to be carried out at each surgery time.
To do this, we first fix a geometry that models the metric with which we will endow the filling $3$-balls after each surgery.

\begin{Definition}[surgery model]
Consider $M_{\stan} = \IR^3$ with its natural $SO(3)$-action and let $g_{\stan}$ be a complete metric on $M_{\stan}$ such that
\begin{enumerate}
\item $g_{\stan}$ is $SO(3)$-invariant,
\item $g_{\stan}$ has non-negative sectional curvature,
\item $(M_{\stan}, g_{\stan})$ is isometric to the standard round $S^2 \times (0, \infty)$ of scalar curvature $2$, outside of some compact subset.
\end{enumerate}
For every $r > 0$, we denote the $r$-ball around $0$ by $M_{\stan}(r)$.

Let $D_{\stan} > 0$ be a positive number such that the compact subset in (3) is contained in $M_{\stan} (D_{\stan})$.
Then we call $(M_{\stan}, g_{\stan}, D_{\stan})$ a \emph{surgery model}.
\end{Definition}

\begin{Definition}[$\varphi$-positive curvature] \label{Def:phipositivecurvature}
We say that a Riemannian metric $g$ on a manifold $M$ has \emph{$\varphi$-positive curvature} for $\varphi > 0$ if for every point $x \in M$ there is an $X > 0$ such that $\sec_x \geq - X$ and
\[ \scal_x \geq - \tfrac32 \varphi \qquad \text{and} \qquad \scal_x \geq 2 X (\log (2 X) - \log \varphi - 3). \]
\end{Definition}
Observe that by \cite[Theorem 4.1]{Ham} this condition is improved by the Ricci flow in the following sense: If $(M, (g_t)_{t \in [t_0, t_1]})$ is a Ricci flow on a compact $3$-manifold with $t_0 > 0$ and $g_{t_0}$ is $t_0^{-1}$-positive, then the curvature of $g_t$ is $t^{-1}$-positive for all $t \in [t_0, t_1]$.

\begin{Definition}[Ricci flow with surgery, $\delta(t)$-precise cutoff] \label{Def:precisecutoff}
Let $\MM$ be a Ricci flow with surgery defined on some time-interval $I \subset [0,\infty)$, let $(M_{\stan},g_{\stan},D_{\stan})$ be a surgery model and let $\delta : I  \to (0, \infty)$ be a function.
We say that $\MM$ is \emph{performed by $\delta(t)$-precise cutoff (using the surgery model $(M_{\stan},g_{\stan}, D_{\stan})$)} if
\begin{enumerate}
\item For all $t > 0$ the metric $g(t)$ is complete and has $t^{-1}$-positive curvature.
\item For every surgery time $T^i$, the subset $\MM(T^i) \setminus U^i_+$ is a disjoint union $D^i_1 \cup D^i_2 \cup \ldots$ of finitely or countably infinitely many smoothly embedded $3$-disks.
\item For every such $D^i_j$ there is an embedding 
\[ \Phi^i_j : M_{\stan}(\delta^{-1}(T^i)) \longrightarrow \MM(T^i) \]
such that $D^i_j \subset \Phi^i_j (M_{\stan}(D_{\stan}))$ and such that the images $\Phi^i_j (M_{\stan} \linebreak[1] (\delta^{-1} \linebreak[1] (T^i)))$ are pairwise disjoint and there are constants $0 <\lambda^i_j \leq \delta(T^i) \sqrt{T^i}$ such that 
\[ \big\Vert g_{\stan} - (\lambda^i_j)^{-2} (\Phi^i_j)^* g(T^i) \big\Vert_{C^{[\delta^{-1}(T^i)]}(M_{\stan}(\delta^{-1}(T^i)))} < \delta(T^i). \]
\item For every such $D^i_j$, the points on the boundary of $U^i_-$ in $\MM(T^{i-})$ corresponding to $\partial D^i_j$ are centers of strong $\delta(T^i)$-necks.
\item For every $D^i_j$ for which the boundary component of $\partial U^i_-$ corresponding to the sphere $\partial D^i_j$ bounds a $3$-disk component $(D')^i_j$ of $M^i \setminus U^i_-$ (i.e. a ``trivial surgery'', see below), the following holds:
For every $\chi > 0$, there is some $t_\chi < T^i$ such that for all $t \in (t_\chi,T^i)$ there is a $(1+\chi)$-Lipschitz map $\xi : (D')^i_j \to D^i_j$ that corresponds to the identity on the boundary.
\item For every surgery time $T^i$, the components of $\MM(T^{i-}) \setminus U^i_-$ are diffeomorphic to one of the following manifolds: $S^2 \times I$, $D^3$, $\IR P^3 \setminus B^3$, a spherical space form, $S^1 \times S^2$, $\IR P^3 \# \IR P^3$ and (in the non-compact case) $S^2 \times [0,\infty)$, $S^2 \times \IR$, $\IR P^3 \setminus \ov{B}^3$.
\end{enumerate}
We will speak of each $D^i_j$ as \emph{a surgery} and if $D^i_j$ satisfies the property described in (5), we call it a \emph{trivial surgery}.

If $\delta > 0$ is a number, we say that $\MM$ is \emph{performed by $\delta$-precise cutoff} if this is true for the constant function $\delta(t) = \delta$.
If $\MM$ is performed by $\delta(t)$-precise cutoff for some function $\delta : I \to (0, \infty)$ and if this function is not of essence, then we sometimes also say that $\MM$ is \emph{performed by precise cutoff}.
\end{Definition}

Note that a Ricci flow with surgery that is performed by precise cutoff may have time-slices that are non-compact or have a boundary.
We need to allow for this possibility, because we later want to analyze universal covers or subsets of Ricci flows with precise cutoff; such settings will, however, only be considered in this paper.
For this reason, we also need to allow the possibility of countably infinitely many surgeries in Definition \ref{Def:precisecutoff}(2).
If, however, a Ricci flow with surgery $\MM$ that is performed by precise cutoff has a time-slice $\MM (t_0)$ that is closed (i.e. compact and no boundary), then all later time-slices $\MM(t)$, $t \geq t_0$ are closed as well.
In particular, if the initial time-slice of $\MM$ is closed, as it will be assumed for normalized initial conditions (see Definition \ref{Def:normalized}), then so are all time-slices.

Observe furthermore that we have phrased the Definition in such a way that if $\MM$ is a Ricci flow with surgery that is performed by $\delta(t)$-precise cutoff, then it is also performed by $\delta'(t)$-precise cutoff whenever $\delta'(t) \geq \delta(t)$ for all $t$.
Note also that trivial surgeries don't change the topology of the component on which they are performed.

We remark that our notion of ``$\delta(t)$-precise cutoff'' differs slightly from Perelman's notion of ``$\delta(t)$-cutoff'' (cf \cite{PerelmanII}).
For example, in our picture the surgeries have size $\lesssim \delta(t) \sqrt{t}$, while in Perelman's construction the size is $\approx h( \delta(t), \linebreak[1] \delta^2 (t) r(t)) \linebreak[1] < \delta(t) r(t)$, where $r(t)$ is similar to the function $\un{r}_\varepsilon (t)$ introduced in Proposition \ref{Prop:CNThm-mostgeneral} below.
This difference will not be essential.
In fact, every ``Ricci flow with $\delta(t)$-cutoff'', as constructed by Perelman, is a ``Ricci flow with surgery that is performed by $\delta'(t)$-precise cutoff'', in the sense of Definition \ref{Def:precisecutoff}, for some suitable function $\delta'(t)$.

\subsection{Existence of Ricci flows with surgery} \label{sec:ExRFsurg}
Ricci flows with surgery and precise cutoff as introduced in Definition \ref{Def:precisecutoff} can indeed be constructed from any given initial metric.
We will make this fact more precise in this subsection.
To simplify things, we restrict the geometries that we want to consider as initial conditions.

\begin{Definition}[Normalized initial conditions] \label{Def:normalized}
We say that a Riemannian $3$-manifold $(M,g)$ is \emph{normalized} if 
\begin{enumerate}
\item $M$ is compact, orientable and has no boundary,
\item $|{\Rm}| < 1$ everywhere and
\item $\vol B(x,1) > \frac{\omega_3}2$ for all $x \in M$, where $\omega_3$ is the volume of a standard Euclidean $3$-ball.
\end{enumerate}
We say that a Ricci flow with surgery $\MM$ has \emph{normalized initial conditions}, if $\MM(0)$ is normalized.
\end{Definition}
Any Riemannian metric on a compact and orientable $3$-manifold can be rescaled to be normalized.
Next, recall
\begin{Definition}[$\kappa$-noncollapsedness]
Let $\MM$ be a Ricci flow with surgery, $(x,t) \in \MM$ and $\kappa, r_0 > 0$.
We say that $\MM$ is \emph{$\kappa$-noncollapsed in $(x,t)$ on scales less than $r_0 > 0$} if $\vol_t B(x,t,r) \geq \kappa r^3$ for all $0 < r < r_0$ for which
\begin{enumerate}
\item the ball $B(x,t,r)$ is relatively compact in $\MM(t)$ and does not intersect the boundary $\partial \MM(t)$,
\item the parabolic neighborhood $P(x,t,r, -r^2)$ is non-singular and
\item $|{\Rm}| < r^{-2}$ on $P(x,t,r,-r^2)$.
\end{enumerate}
\end{Definition}

We now introduce a notion of canonical neighborhood assumptions, which slightly differs from the notions that can be found in other sources, but which suits better our purposes.
\begin{Definition}[canonical neighborhood assumptions] \label{Def:CNA}
Let $\MM$ be a Ricci flow with surgery, $(x, t) \in \MM$ and $r, \varepsilon, \eta > 0$, $E < \infty$ be constants.
We say that $(x,t)$ satisfies the \emph{canonical neighborhood assumptions $CNA(r, \varepsilon, E, \eta)$} if $|{\Rm}|(x,t) < r^{-2}$ or if the following three properties hold:
\begin{enumerate}[label=(\textit{\arabic*})]
\item $(x,t)$ is a center of a strong $\varepsilon$-neck or an $(\varepsilon, E)$-cap $U \subset \MM (t)$, \\
If $U \approx \IR P^3 \setminus \ov{B}^3$, then there is a time $t_1 < t$ such that all points on $U$ survive until time $t_1$ and such that flow on $U \times [t_1, t]$ lifted to its double cover contains strong $\varepsilon$-necks and both lifts of $(x,t)$ are centers of such strong $\varepsilon$-necks,
\item $|{\nabla |{\Rm}|^{-1/2}}| (x,t) < \eta^{-1}$ and $| \partial_t |{\Rm}|^{-1} | (x,t)  < \eta^{-1}$,
\item $\vol_t B( x, t, r' ) > \eta (r')^3$ for all $0 < r' \leq |{\Rm}|^{-1/2} (x, t)$,
\end{enumerate}
or property (2) holds and the component of $\MM(t)$ in which $x$ lies, is closed and the sectional curvatures are positive and $E^2$-pinched on this component, i.e. they are contained in an interval of the form $(\lambda, E^2 \lambda)$ for some $\lambda > 0$ (and hence that component is diffeomorphic to a spherical space form).
\end{Definition}

Note that we have added an additional assumption in the case in which $U \approx \IR P^3 \setminus \ov{B}^3$ to ensure that the canonical neighborhood assumptions are stable when taking covers of Ricci flows with surgery (compare with Lemma \ref{Lem:tdMM}).
We remark that every manifold that contains a set diffeomorphic to $\IR P^3 \setminus \ov{B}^3$, admits a double cover in which this set lifts to a set diffeomorphic to $S^2 \times (0,1)$.
So it is possible to verify this extra assumption if all the other canonical neighborhood assumptions hold in any double cover. 

The following proposition gives a characterization of regions of high curvature in a Ricci flow with surgery that is performed by precise cutoff.
The power of this proposition lies in the fact that none of the parameters depends on the number or the preciseness of the preceding surgeries.
Thus, it provides a tool to perform surgeries in a controlled way and hence it can be used to construct long-time existent Ricci flows with surgery as presented in Proposition \ref{Prop:RFwsurg-existence} below.
The following proposition also plays an important role in the long-time analysis of Ricci flows with surgery that are performed by precise cutoff and will, in particular, be used in sections \ref{sec:Perelman} and \ref{sec:maintools} of this paper.

\begin{Proposition}[Canonical Neighborhood Theorem, Ricci flows with surgery] \label{Prop:CNThm-mostgeneral}
For every surgery model $(M_{\stan}, \linebreak[1] g_{\stan}, \linebreak[1] D_{\stan})$ and every $\varepsilon > 0$ there are constants $\un\eta > 0$ and $\un{E}_\varepsilon < \infty$ and decreasing continuous positive functions $\un{r}_\varepsilon, \un\delta_\varepsilon, \un\kappa : [0,\infty) \to (0, \infty)$ such that the following holds:

Let $\MM$ be a Ricci flow with surgery on some time-interval $[0,T)$ that has normalized initial conditions and that is performed by $\un\delta_\varepsilon (t)$-precise cutoff.
Then for every $t \in [0,T)$
\begin{enumerate}[label=(\textit{\alph*})]
\item $\MM$ is $\un\kappa(t)$-noncollapsed at scales less than $\sqrt{t}$ at all points of $\MM(t)$.
\item All points of $\MM (t)$ satisfy the canonical neighborhood assumptions $CNA \linebreak[1] (\un{r}_\varepsilon (t) \sqrt{t}, \linebreak[1] \varepsilon, \linebreak[1] \un{E}_\varepsilon,  \linebreak[1] \un\eta)$.
\end{enumerate}
\end{Proposition}
For a proof of this proposition and of Proposition \ref{Prop:RFwsurg-existence} see \cite[sec 5]{PerelmanII}, \cite[sec 77ff]{KLnotes}, \cite[Proposition 17.1, Theorem 15.9]{MTRicciflow}, \cite[Proposition B, Theorem 5.3.1]{BBBMP}, \cite[Theorem 7.5.1]{Bamler-diploma}.
The following proposition provides us an existence result for Ricci flows with surgery.
\begin{Proposition} \label{Prop:RFwsurg-existence}
Given a surgery model $(M_{\stan}, g_{\stan}, D_{\stan})$, there is a continuous function $\un\delta : [0, \infty) \to (0, \infty)$ such that if $\delta' : [0, \infty) \to (0, \infty)$ is a continuous function with $\delta'(t) \leq \un\delta(t)$ for all $t \in [0,\infty)$ and $(M,g)$ is a normalized Riemannian manifold, then there is a Ricci flow with surgery $\MM$ defined for times $[0, \infty)$ with $\MM(0) = (M,g)$ and that is performed by $\delta'(t)$-precise cutoff. (Observe that we can possibly have $\MM(t) = \emptyset$ for large $t$.)

Moreover, if $\MM$ is a Ricci flow with surgery on some time-interval $[0,T)$ that has normalized initial conditions and that is performed by $\un\delta(t)$-precise cutoff, then $\MM$ can be extended to a Ricci flow on the time-interval $[0, \infty)$ that is performed by $\delta'(t)$-precise cutoff on the time-interval $[T, \infty)$.
\end{Proposition}

We point out that the functions $\un\delta_\epsilon (t), \un{r}_\varepsilon(t), \un\kappa (t)$ and the constants $\un{\eta}, \un{E}_\varepsilon$ in Proposition \ref{Prop:CNThm-mostgeneral} as well as the function $\un\delta(t)$ in Proposition \ref{Prop:RFwsurg-existence} depend on the choice of the surgery model.
\begin{quote}
\textit{\textbf{From now on we will fix a surgery model $(M_{\stan}, g_{\stan}, D_{\stan})$ for the rest of this and the following three papers and we will not mention this dependence anymore.}}
\end{quote}

\section{Perelman's long-time analysis results and certain generalizations} \label{sec:Perelman}
\subsection{Perelman's long-time curvature estimates} \label{subsec:Perelmanlongtime}
In this subsection, we will review some of Perelman's long-time analysis results (see \cite{PerelmanII}).
We will generalize these results to the boundary case and go through most of their proofs.
The most important result of this section will be Proposition \ref{Prop:genPerelman} below.
It will be used in section \ref{sec:maintools} of this paper.
In addition, many of the Lemmas leading to this Proposition will also be used in that section.
The boundary case will be important for us, because we want to analyze Ricci flows in local covers.

The following notation will be used throughout the whole paper.
\begin{Definition} \label{Def:rhoscale}
Let $(M,g)$ be a Riemannian manifold and $x \in M$ a point.
We define
\[ \rho(x) = \sup \{ r \;\; : \;\; \sec \geq - r^{-2} \;\; \text{on} \;\; B(x,r) \}. \]
For $r _0 > 0$ we set furthermore $\rho_{r_0}(x) = \min \{ \rho(x), r_0 \}$.
If $(M, g) = \MM(t)$ is the time-slice of a Ricci flow (with surgery) $\MM$, then we often use the notation $\rho(x,t)$ and $\rho_{r_0} (x,t)$.
\end{Definition}

We also need to use the $\LL$-functional as defined in \cite[sec 7]{PerelmanI}:
For any smooth space-time curve $\gamma : [t_1, t_2] \to \MM$ ($t_1 < t_2 \leq t_0$) in a Ricci flow with surgery $\MM$ set
\begin{equation} \label{eq:defLLlength}
 \LL(\gamma) = \int_{t_1}^{t_2} \sqrt{t_0 - t'} \big(|\gamma'|^2(t') + \scal (\gamma(t'), t') \big) dt'.
\end{equation}
We say that $\LL$ is \emph{based in $t_0$} and call $\LL(\gamma)$ the \emph{$\LL$-length} of $\gamma$.
A curve $\gamma$ that is a critical point of $\LL$ with respect to variations that fix the endpoints is called \emph{$\LL$-geodesic}. 

We now present the main result of this section.
Before we do that we introduce the following convention that we will assume from now on:
We will often be dealing with Ricci flows with surgery $\MM$ defined on a time-interval of the form $[ t_0 - r_0^2, t_0]$ and most results require certain canonical neighborhood assumptions to hold on $\MM$.
For times that are very close to $t_0 - r_0^2$ this may be problematic since strong $\varepsilon$-necks might stick out of the time-interval.
So we will assume from now on that in such a setting $\MM$ can be extended backwards to a Ricci flow with surgery $\MM' \supset \MM$ in which the required canonical neighborhood assumptions hold on the time-interval $[t_0 - r_0^2, t_0]$.
In fact, in our applications $\MM$ will always arise as such a restriction.
Note that we could also resolve this issue by requiring in the assumptions of each of the following results that the canonical neighborhood assumptions hold on a slightly smaller time-interval $[t_0 - 0.99 r_0^2, t_0]$.

\begin{Proposition}[\hbox{\cite[6.8]{PerelmanII}} in the non-compact case] \label{Prop:genPerelman}
There is a constant $\varepsilon_0 > 0$ such that for all $w, r, \eta > 0$ and $E < \infty$ and $1 \leq A < \infty$ and $m \geq 0$ there are $\tau = \tau(w, A, E, \eta), \ov{r} = \ov{r}(w, A, E, \eta), \widehat{r} = \widehat{r}(w, E, \eta), \delta = \delta(r, w, A, E, \eta, m) > 0$ and $K_m = K_m (w, A, E, \eta), C_1 = C_1(w, A, E, \eta), Z = Z(w, A, \linebreak[1] E, \linebreak[1] \eta) \linebreak[1] < \infty$ such that: \\
Let $r_0^2 \leq t_0/2$ and let $\MM$ be a Ricci flow with surgery (whose time-slices are allowed to have boundary) on the time-interval $[t_0 - r_0^2, t_0]$ that is performed by $\delta$-precise cutoff and consider a point $x_0 \in \MM(t_0)$.
Assume that the canonical neighborhood assumptions $CNA (r \sqrt{t_0}, \varepsilon_0, E, \eta)$, as described in Definition \ref{Def:CNA}, are satisfied on $\MM$.
We also assume that the curvature on $\MM$ is uniformly bounded on compact time-intervals which don't contain surgery times and that all time slices of $\MM$ are complete.

In the case in which some time-slices of $\MM$ have non-empty boundary, we assume that
\begin{enumerate}[label=(\roman*)]
\item For all $t_1, t_2 \in [t_0 - \frac1{10} r_0^2, t_0]$, $t_1 < t_2$ we have:
if some $x \in B(x_0, t_0, r_0)$ survives until time $t_2$ and $\gamma : [t_1, t_2] \to \MM$ is a space-time curve with endpoint $\gamma(t_2) \in B(x, t_2, (A+3) r_0)$ that meets the boundary $\partial\MM$ somewhere, then it has $\LL$-length $\LL(\gamma) > Z r_0$ ($\LL$ being based in $t_2$, see (\ref{eq:defLLlength})).
\item For all $t \in [t_0 - \frac1{10} r_0^2, t_0]$ we have: if some $x \in B(x_0, t_0, r_0)$ survives until time $t$, then $B(x, t, 2(A+3) r_0 + r \sqrt{t_0})$ does not meet the boundary $\partial\MM(t)$.
\end{enumerate}
Now assume that
\begin{enumerate}[label=(\roman*), start=3]
\item $0 < r_0 \leq \ov{r} \sqrt{t_0}$,
\item $\sec_{t_0} \geq - r_0^{-2}$ on $B(x_0, t_0, r_0)$ and
\item $\vol_{t_0} B(x_0, t_0, r_0) \geq w r_0^3$.
\end{enumerate}
Then $|{\nabla^k \Rm}| < K_m r_0^{-2 - k}$ on $B(x_0, t_0, A r_0)$ for all $k \leq m$.
In particular, if $r_0 = \rho(x_0, t_0)$, then $r_0 > \widehat{r} \sqrt{t_0}$.

If moreover the surgeries on $\MM$ are performed by $\delta'$-cutoff for some $0 < \delta' \leq \delta$ with $C_1 \delta' \sqrt{t_0} \leq r_0$, then the parabolic neighborhood $P(x_0, t_0, A r_0, -\tau r_0^2)$ is non-singular and we have $|{\nabla^k \Rm}| < K_k r_0^{-2-k}$ on $P(x_0, t_0, A r_0, -\tau r_0^2)$ for all $k \geq 0$.
\end{Proposition}

The following Corollary is a consequence of Propositions \ref{Prop:genPerelman} and \ref{Prop:CNThm-mostgeneral}.

\begin{Corollary}[cf \hbox{\cite[6.8, 7.3]{PerelmanII}}] \label{Cor:Perelman68}
There is a continuous positive function $\delta : [0, \infty) \to (0, \infty)$ such that for every $w > 0$, $1 \leq A < \infty$ and $m \geq 0$ there are constants $\tau = \tau(w, A), \ov{\rho} = \ov{\rho} (w), \ov{r} = \ov{r} (w, A), c_1 = c_1(w, A) > 0$ and $T = T(w, A, m), K_m = K_m(w, A) < \infty$ such that:

Let $\MM$ be a Ricci flow with surgery on the time-interval $[0, \infty)$ with normalized initial conditions (whose time-slices are all closed) that is performed by $\delta(t)$-precise cutoff.
Let $t > T$ and $x \in \MM (t)$.
\begin{enumerate}[label=(\textit{\alph*})]
\item If $0 < r \leq \min \{ \rho(x,t), \ov{r} \sqrt{t} \}$ and $\vol_t B(x,t, r) \geq w r^3$, then $|{\nabla^k \Rm}| < K_m r^{-2-k}$ on $B(x,t, Ar)$ for all $k \leq m$.
Moreover, if all surgeries on the time-interval $[t - r^2, t]$ are performed by $c_1 r t^{-1/2}$-precise cutoff, then the parabolic neighborhood $P(x, t, Ar, - \tau r^2)$ is non-singular and we have $|{\nabla^k \Rm}| < K_k r^{-2 - k}$ on $P(x, t, Ar, - \tau r^2)$ for all $k \geq 0$.
\item If $\vol_t B(x,t,\rho(x,t)) \geq w \rho^3 (x,t)$, then $\rho(x,t) > \ov{\rho} \sqrt{t}$ and the parabolic neighborhood $P(x,t, A \sqrt{t}, - \tau t)$ is non-singular and we have $|{\nabla^k \Rm}| < K_k t^{-1-k/2}$ on $P(x,t, A \sqrt{t}, - \tau t)$ for all $k \geq 0$.
\end{enumerate}
\end{Corollary}

In the case $A = 1$, this Corollary implies \cite[6.8]{PerelmanII} and parts of \cite[7.3]{PerelmanII}.

In the following, we will present proofs of Proposition \ref{Prop:genPerelman} and Corollary \ref{Cor:Perelman68}.
They require a few rather complicated Lemmas, which we will establish first.
The proofs of Proposition \ref{Prop:genPerelman} and Corollary \ref{Cor:Perelman68} can be found at the end of this subsection.
Note that the following arguments will be very similar to those presented in \cite{PerelmanII} and \cite{KLnotes} with small modifications according to the author's taste.
Occasionally, we will omit shorter arguments and refer to \cite{KLnotes}.
The main objective in the proofs will be the discussion of the influence of the boundary.
Upon the first reading, it is recommended to skip the remainder of this subsection.
The boundary case of Proposition \ref{Prop:genPerelman}, which is the new result of this subsection, will only be used in subsection \ref{subsec:curvboundinbetween}.

The following distance distortion estimates will be used frequently throughout this paper.
\begin{Lemma}[distance distortion estimates] \label{Lem:distdistortion}
Let $(M, (g_t)_{t \in [t_1, t_2]})$ be a Ricci flow whose time-slices are complete and let $x_1, x_2 \in M$.
Then
\begin{enumerate}[label=(\alph*)]
\item If $\Ric_t \leq K$ along some minimizing geodesic between $x_1$ and $x_2$ in $(M, g_t)$, then at time $t$ we have $\frac{d}{d t^-} \dist_t (x_1, x_2) \geq - K \dist_t (x_1, x_2)$ in the barrier sense.
Likewise, if $\Ric_t \geq - K$ along such a minimizing geodesic, then $\frac{d}{d t^+} \dist_t (x_1, x_2) \leq K \dist_t (x_1, x_2)$ in the barrier sense.
\item If at some time $t$ we have $\dist_t (x_1, x_2) \geq 2 r$ and $\Ric_t \leq r^{-2}$ on $B(x_1, r) \cup B(x_2, r)$ for some $r > 0$, then $\frac{d}{dt^-} \dist_t (x_1, x_2) \geq - \frac{16}3 r^{-1}$ in the barrier sense.
\end{enumerate}
Both statements are also true in a Ricci flow with surgery if we can guarantee that some minimizing geodesic between $x_1$ and $x_2$ doesn't intersect surgery points.
For example, this condition is satisfied if at time $t$ the surgeries are performed by $\delta$-precise cutoff for some sufficiently small $\delta$ and if $|{\Rm}| (x_1, t), |{\Rm}| (x_2, t) < c \delta^{-2} t^{-1}$ for a certain universal $c > 0$, which depends on the chosen surgery model.
\end{Lemma}

\begin{proof}
See \cite[sec 27]{KLnotes}, \cite[8.3]{PerelmanI}, \cite[sec 2.3]{Bamler-diploma}.
The very last statement follows from Definition \ref{Def:precisecutoff}(3).
\end{proof}

We will also need
\begin{Lemma} \label{Lem:shortrangebounds}
Let $\MM$ be a Ricci flow with surgery that satisfies the canonical neighborhood assumptions $CNA (r, \varepsilon, E, \eta)$ for some $r, \varepsilon, E, \eta > 0$, let $(x, t) \in \MM$ and set $Q = |{\Rm}|(x,t)$.
\begin{enumerate}[label=(\alph*)]
\item If $Q \leq r^{-2}$, then $|{\Rm}| < 2 r^{-2}$ on $P(x, t, \frac{\eta}{10} r, - \frac{\eta}{10} r^2)$.
\item If $Q \geq r^{-2}$, then $|{\Rm}| < 2 Q$ on $P(x, t, \frac{\eta}{10} Q^{-1/2}, - \frac{\eta}{10} Q^{-1})$.
\end{enumerate}
\end{Lemma}

\begin{proof}
See \cite[4.2]{PerelmanII}, \cite[Lemma 70.1]{KLnotes}, \cite[sec 6.2]{Bamler-diploma}.
\end{proof}

We now present the first main Lemma.

\begin{Lemma}[cf \hbox{\cite[6.3(a)]{PerelmanII}}] \label{Lem:6.3a}
For any $1 \leq A < \infty$ and $w, r, \eta > 0$ there are $\kappa = \kappa(w, A, \eta), \delta = \delta( A, r, \eta) > 0$, $Z = Z(A) < \infty$ such that: \\
Let $r_0^2 < t_0/2$ and let $\MM$ be a Ricci flow with surgery (whose time-slices are allowed to have boundary) on the time-interval $[ t_0 - r_0^2, t_0]$ that is performed by $\delta$-precise cutoff and consider a point $x_0 \in \MM(t_0)$.
Assume that the canonical neighborhood assumptions $CNA (r \sqrt{t_0}, \varepsilon, E, \eta)$ are satisfied on $\MM$ for some $\varepsilon, \eta > 0$.
We also assume that the curvature on $\MM$ is uniformly bounded on compact time-intervals which don't contain surgery times and that all time-slices of $\MM$ are complete.

Assume that the parabolic neighborhood $P(x_0, t_0, r_0, - r_0^2)$ is non-singular, that $|{\Rm}| \leq r_0^{-2}$ on $P(x_0, t_0, r_0, - r_0^2)$ and $\vol_{t_0} B(x_0, t_0, r_0) \geq w r_0^3$.

In the case in which some time-slices of $\MM$ have non-empty boundary, we assume that
\begin{enumerate}[label=(\roman*)]
\item every space-time curve $\gamma : [t, t_0] \to \MM$ with $t \in [t_0 - r_0^2, t_0)$ that ends in $\gamma(t_0) \in B(x_0, t_0, A r_0)$ and that meets the boundary $\partial \MM (t')$ at some time $t' \in [t, t_0]$, has $\LL$-length $\LL(\gamma) > Z r_0$ (based in $t_0$),
\item the ball $B (x_0, t_0, (2A+1) r_0 + r \sqrt{t_0})$ does not hit the boundary $\partial \MM(t_0)$ and for every $t \in [t_0 - \frac12 r_0^2, t_0]$ the ball $B(x_0, t, A (1 - 2 (t_0 - t) r_0^{-2} ) r_0 + \frac1{10} r_0)$ does not hit the boundary $\partial \MM(t)$.
\end{enumerate}

Then $\MM$ is $\kappa$-noncollapsed on scales less than $r_0$ at all points in the ball $B(x_0, t_0, A r_0)$.
\end{Lemma}

\begin{proof}
We follow the lines of \cite[6.3(a)]{PerelmanII}.

We first consider the case in which for some $t \in [t_0 - r_0^2, t_0]$ the component of $\MM(t)$ that contains $x_0$ is closed and has positive sectional curvature.
Then the same is true for the corresponding component of $\MM (t_0)$ and hence we are done by volume comparison.
So in the following, we exclude this case and hence the last option in Definition \ref{Def:CNA} of the canonical neighborhood assumptions will not occur.

Let $x_1 \in B(x_0, t_0, A r_0)$ and $0 < r_1 < r_0$ be such that $P(x_1, t_0, r_1, - r_1^2)$ is non-singular and such that $|{\Rm}| < r_1^{-2}$ on $P(x_1, t_0, r_1, - r_1^2)$.
Note that by condition (ii) the ball $B(x_1, t_0, r_1)$ does not hit the boundary $\partial \MM (t_0)$.

\begin{Claim1}
There is a universal constant $\delta_0 > 0$ such that if $\delta < \delta_0$, then we can restrict ourselves to the case $r_1 > \frac12 r \sqrt{t_0}$.
By this we mean that if the Lemma holds under the additional restriction that $r_1 > \frac12 r \sqrt{t_0}$ for some $\kappa' = \kappa' (w, A, \eta) > 0$, then it also holds whenever $r_1 \leq \frac12 r \sqrt{t_0}$ for some $\kappa = \kappa (w, A, \eta) > 0$.
\end{Claim1}

\begin{proof}
Assume that the Lemma holds whenever $r_1 > \frac12 r \sqrt{t_0}$, but assume that $r_1 \leq \frac12 r \sqrt{t_0}$.
Let $s > 0$ be the supremum over all $r'_1>0$ such that the properties above still hold for\footnote{In this and the following papers we use the notation ``$a \leftarrow b$'' for ``$a$ is replaced by $b$'' or ``$b$ is assigned to $a$''.} $r_1 \leftarrow r'_1$, that is $0 < r'_1 < r_0$, $P(x_1, t_0, r'_1, - r_1^{\prime 2})$ is non-singular and $|{\Rm}| < r_1^{\prime - 2}$ on $P(x_1, t_0, r'_1, - r_1^{\prime 2})$.
If $s \leq \frac12 r \sqrt{t_0}$, then there are several cases:
\begin{enumerate}[label=(\arabic*)]
\item The closure of $P(x_1, t_0, s, -s^2)$ hits a surgery point $(x', t')$.
By Definition \ref{Def:precisecutoff}(3), there is a neighborhood $U \subset \MM (t')$ of $(x',t')$ whose geometry is modeled on the surgery model on a scale of at least $c_1 s$ for some universal $c_1 > 0$.
Note that, again by Definition \ref{Def:precisecutoff}(3) and the fact that $(x', t')$ is a surgery point, we have $B(x', t', c_2 \delta^{-1} s) \subset U$ for some universal $c_2 > 0$.
Since by distance distortion estimates $\dist_{t'} (x_1, x') \leq 10 \dist_{t_0} (x_1, x') \leq 10 s$, we find that for $\delta < \frac1{20} c_2$ we have $B(x_1, t', \frac1{10} s) \subset U$.
Since the standard solution is uniformly noncollapsed, we have $\vol_{t'} B(x_1, t', \frac1{10} s) > \kappa' s^3$ for some universal $\kappa' > 0$.
Again, by distance distortion estimates, we have $B(x_1, t', \frac1{10} s) \subset B(x_1, t_0, s)$.
So together with a volume distortion estimate, we conclude $\vol_{t_0} B(x_1, t_0, s) > \kappa'' s^3$ for some universal $\kappa'' > 0$.
By volume comparison, this implies $\vol_{t_0} B(x_1, t_0, r_1) > \kappa r_1^3$ for some universal $\kappa > 0$ (recall that by our assumptions $s \geq r_1$).
\item There is a point $(x', t')$ in the closure of $P(x_1, t_0, s, -s^2)$ with $|{\Rm}|(x',t') = s^{-2}$.
Then let $\gamma : [0,l] \to \MM(t_0)$ be a time-$t_0$ minimizing geodesic, parameterized by arclength, between $x_1$ and $x'$.
So the image of $\gamma$ lies in the closure of $B(x_1, t_0, s)$.
Let $x'' = \gamma ( l - \frac{\eta}{100} s ) \in B(x_1, t_0, s)$ if $l > \frac{\eta}{100} s$ and $x'' = x'$ otherwise.
Using distance-distortion estimates, we find that $\dist_{t'} (x'', x') \leq 10 \dist_{t_0} (x'', x') \leq \frac{\eta}{10} s$.
By Lemma \ref{Lem:shortrangebounds} and the canonical neighborhood assumptions $CNA ( r \sqrt{t_0}, \linebreak[1] \varepsilon, \linebreak[1] E, \linebreak[1] \eta)$, we conclude that $|{\Rm}|(x'', t') \geq \frac12 s^{-2} > r^{-2} t_0^{-1}$.

Next, we use distance-distortion estimates to show that $B(x'', t', \frac{\eta}{2000} s)  \subset B(x_1, t_0, s)$:
Assume that this inclusion was wrong and pick $t'' \in (t', t_0]$ minimal such that we have $B(x'', t'', \frac{\eta}{2000} s) \subset B(x_1, t_0, s)$ (note that the inclusion holds for $t'' = t_0$, because $B(x'', t_0, \frac{\eta}{100} s ) \subset B(x_1, t_0, s)$ and note that the set of all $t'' \in (t', t_0]$ for which the inclusion holds is closed).
We can then use distance-distortion estimates to show that $B(x'', t'', \frac{\eta}{2000} s) \subset B(x'', t_0, \frac{\eta}{200} s) \subset B(x_1, t_0, s - \frac{\eta}{200} s)$.
Since $t'' > t'$, this contradicts the minimality of $t''$.

We can now use the canonical neighborhood assumptions, to conclude that $\vol_{t'} B(x'', \linebreak[1] t', \linebreak[1]  \frac{\eta}{2000} s) > \eta ( \frac{\eta}{2000} s)^3$ and as in case (1) we obtain that $\vol_{t_0} B(x_1, t_0, r_1) > \kappa r_1^3$ for some universal $\kappa = \kappa(\eta) > 0$.
\item We have $s = r_0$.
So $r_0 \leq \frac12 r \sqrt{t_0}$.
In this case choose $0 < d \leq (A+1) r_0$ maximal with the property that $|{\Rm}| < r_0^{-2} = s^{-2}$ on $B(x_1, t_0, d)$.
So $d \geq r_0$.
If $d = (A+1) r_0$, then
\[ \vol_{t_0} B(x_1, t_0, d) \geq \vol_{t_0} B(x_0, t_0, r_0) \geq w r_0^{3} = \frac{w}{(A+1)^3} d^3. \]
So by volume comparison and assumption (ii) we obtain a lower volume bound on the normalized volume of $B(x_1, t_0, r_1)$ since $r_1 < r_0 < d$.
Assume now that $d < (A+1) r_0$.
Then $|{\Rm}|(x', t_0) = r_0^{-2} \geq 4 r^{-2} t_0^{-1}$ for some $(x', t_0)$ in the closure of $B(x_1, t_0, d)$.
As in case (2), we can find a point $(x'', t_0) \in B(x_1, t_0, d - \frac{\eta}{10} r_0)$ with $|{\Rm}|(x'', t_0) > \frac12 r_0^{-2} > r^{-2} t_0^{-1}$.
By the canonical neighborhood assumptions, we have $\vol_{t_0} B(x_1, t_0, d) \geq \vol_{t_0} B(x'', t_0, \frac{\eta}{10} r_0) > \eta (\frac{\eta}{10} r_0)^3$.
So again by volume comparison, we find that $\vol_{t_0} B(x_1, t_0, r_1) > \kappa r_1^3$ for some $\kappa = \kappa (\eta, A) > 0$.
\end{enumerate}
Lastly, if $s > \frac12 r \sqrt{t_0}$, then the conditions mentioned at the beginning of the proof hold for some $r_1' > \frac12 r \sqrt{t_0}$.
If the assertion of the Lemma holds for $r'_1$ and some $\kappa' = \kappa'(w, A, \eta) > 0$, then by volume comparison, it also holds for all $r_1 \leq r'_1$ and some $\kappa = \kappa(w, A, \eta) > 0$.
\end{proof}

So assume in the following that $r_1 > \frac12 r \sqrt{t_0}$.
We will now set up an $\LL$-geometry argument.
Define for any $t \in [t_0 - r_0^2, t_0]$ and $y \in \MM(t)$
\[
 L(y, t) = \inf \Big\{ \LL(\gamma) \;\; : \;\; \gamma : [t, t_0] \to \MM \; \text{smooth}, \; \gamma(t) = y, \; \gamma(t_0) = x_1 \Big\}.
\]
Moreover, set
\[ \ov{L}(y, t) = 2 \sqrt{t_0 - t} L(y, t) \qquad \text{and} \qquad \ell(y, t) = \frac1{2\sqrt{t_0 - t}} L(y, t). \]
Let
\begin{multline*}
 D_t = \big\{ y \in \MM (t) \;\; : \;\; \text{there is a minimizing $\LL$-geodesic $\gamma : [t, t_0] \to \MM \setminus \partial \MM$} \\ \text{with $\gamma(t) = y$ and $\gamma(t_0) = x_1$ that does not hit any surgery points} \big\}.
\end{multline*}
We can then define the reduced volume
\[ \widetilde{V}(t) = (t_0 - t)^{-n/2} \int_{D_t} e^{-\ell(\cdot, t)} d {\vol_t}. \]
It is shown in \cite[7.1]{PerelmanI} that $\widetilde{V}(t)$ is non-decreasing in $t$.

We will now prove that the quantity $\ell (\cdot, t_0 - r_0^2)$ is uniformly bounded from above on $B(x_0, t_0, r_0)$ by a constant that only depends on $A$ if $\delta$ is chosen small enough depending on $A$, $r$ and $\eta$.
To do this we will use a maximum principle argument on $D_t$.
The following claim will ensure hereby that extremal points of $L$ lie inside $D_t$.

\begin{Claim2}
For any $\Lambda < \infty$ there is a constant $\delta^* = \delta^* (\Lambda, r, \eta) > 0$ such that whenever $\delta \leq \delta^*$ and $Z \geq \Lambda$, then the following holds:
Assume that $r_1 > \frac12 r \sqrt{t_0}$.
If $t \in [t_0 - r_0^2, t_0]$, $y \in \MM(t)$ and $L(y,t) < \Lambda r_0$, then $y \in D_t$, which also implies that $(y,t)$ is not a surgery point.
\end{Claim2}

\begin{proof}
Assume that $y \in \MM (t) \setminus D_t$ or that $(y,t)$ is a surgery point.
Then there is a space-time curve $\gamma : [t, t_0] \to \MM$ with $\LL(\gamma) < \Lambda r_0$ that either touches $\partial \MM$ or a surgery point.
The first case is excluded by assumption (i), so $\gamma$ touches a surgery point.
We will now follow the lines of \cite[5.3]{PerelmanII}, \cite[Lemma 79.3]{KLnotes} or \cite[p 92]{Bamler-diploma}.

First, define
\[  \LL_+ (\gamma) = \int_{t}^{t_0} \sqrt{t_0 - t_*} \big(|\gamma'|^2(t_*) + \scal_+ (\gamma(t_*), t_*) \big) dt_*, \]
where $\scal_+ (\gamma(t_*), t_*)  = \max \{ \scal (\gamma(t_*), t_*), 0 \}$ denotes the non-negative part.
Using Definition \ref{Def:precisecutoff}(1) and Definition \ref{Def:phipositivecurvature}, we can estimate
\begin{multline*}
 \LL_+ (\gamma) \leq \LL (\gamma) + \int_{t}^{t_0} \sqrt{t_0 - t_*} \cdot \frac{3}{2 t_*}  dt_*
\leq \Lambda r_0 + \frac{3}2 \int_{t_0 - r_0^2}^{t_0} \frac{\sqrt{r_0^2}}{t_0 - r_0^2} dt_* \\
\leq \Lambda r_0 + \frac{3}2 \frac{\sqrt{r_0^2}}{t_0 - r_0^2} r_0^2
\leq (\Lambda + \tfrac32 ) r_0 .
\end{multline*}
Note that for any sub-interval $[t^*_1, t^*_2] \subset [t, t_0]$, we have
\begin{multline} \label{eq:LLplussmall}
 \LL_+ \big( \gamma|_{[t^*_1, t^*_2]} \big) = \int_{t^*_1}^{t^*_2} \sqrt{t_0 - t_*} \big(|\gamma'|^2(t_*) + \scal_+ (\gamma(t_*), t_*) \big) dt_* \\ \leq \LL_+ (\gamma) \leq (\Lambda + \tfrac32) r_0. 
\end{multline}

Assume that $\gamma$ touches a surgery point $(y', t')$.
Let $\theta > 0$, $D < \infty$ be constants, whose values will be fixed later in the proof, depending only on $\Lambda, r$.
Using \cite[4.5]{PerelmanII}, \cite[Lemma 74.1]{KLnotes} or \cite[Lemma 7.4.1]{Bamler-diploma}, and assuming $\delta$ to be small depending on $\theta, D, \eta$, we may find constants $\sigma \in (0, 1-\theta]$ and $\lambda > 0$ such that the parabolic neighborhood
\[ P = P(y', t', D \lambda, \sigma \lambda^2 ) \]
is non-singular and such that, after rescaling by $\lambda^{-2}$, the Ricci flow on $P$ is $\eps'$-close to a subset of a standard solution on the time-interval $[0, \sigma]$ for some suitably small $\eps' > 0$.
Here, the constant $\sigma$ can be chosen such that one of the following is true: $\sigma = 1- \theta$ or no point of $B(y', t', D \lambda)$ survives past time $t' + \sigma \lambda^2$.
Recall that a standard solution is a Ricci flow with initial metric $(M_{\textnormal{stan}}, g_{\textnormal{stan}})$, bounded curvature on compact time-intervals and complete time-slices.
Note also that the proofs for \cite[4.5]{PerelmanII}, \cite[Lemma 74.1]{KLnotes} or \cite[Lemma 7.4.1]{Bamler-diploma} still hold in the boundary case, since by Definition \ref{Def:precisecutoff}(3) we have the bound $\dist_{t'} (y', \partial \MM(t') ) > c'' \delta^{-1} \lambda$ for some universal constant $c'' > 0$.
So for small enough $\delta$, depending on $D$, we have $\dist_{t'} (y', \partial \MM(t') ) \gg D \lambda$.

The fact that $P$ is close to a subset of a standard solution implies that
\begin{equation} \label{eq:lowerscaltheta}
  \scal (x_* , t_*) >  \frac{c \lambda^{-2}}{1 - \lambda^{-2} (t_* - t')}     \qquad \text{for all} \qquad (x_*,t_* ) \in P
\end{equation}
and
\begin{equation} \label{eq:uppercurvboundtheta}
 |{\Rm}| < C_{\theta} \lambda^{-2} \qquad \text{on} \qquad P,
\end{equation}
where $c > 0$ is some universal constant and $C_\theta < \infty$ is a constant that only depends on $\theta$.
For more details see \cite[sec 2]{PerelmanII}, \cite[Lemma 63.1]{KLnotes} or \cite[Theorem 7.3.2]{Bamler-diploma}.
In particular, (\ref{eq:lowerscaltheta}) implies that $\scal > c \lambda^{-2} > c' \delta^{-2} t^{\prime -1}$ on $P$ for some universal $c' > 0$.
So if $\delta < \frac12 \sqrt{c'} r$, then $\scal > 4 r^{-2} t^{\prime -1} > r_1^{-2}$ on $P$.
This implies that $P$ is disjoint from $P(x_1, t_0, r_1, - r_1^2)$.
So $\gamma |_{[t', t_0]}$ has to exit $P$ before entering $P(x_1, t_0, r_1, - r_1^2)$.

We will now fix the constants $\theta, D$.
Set
\[ \tau = \big( 400 ( \Lambda + 10 ) \big)^{-2} r^2. \]
Next, set
\[ \theta = \exp \Big( - \frac{2 (\Lambda + 10)}{c r \sqrt{\tau}} \Big) \qquad \text{and} \qquad D = \frac{2 (\Lambda + 10)}{r \sqrt{\tau}} \exp (2C_{\theta}). \]
Note that these constants only depend on $\Lambda$ and $r$

We first prove that 
\begin{equation} \label{eq:gammastaysinP}
  \gamma \big( [t_0 - \tau r_1^2, t_0] \big) \subset P (x_1, t_0, r_1, - r_1^2).
\end{equation}
Assume not and let $t'' \in [t_0 - \tau r_1^2, t_0]$ be maximal such that $\gamma (t'') \not \in P (x_1, \linebreak[1] t_0, \linebreak[1] r_1,\linebreak[1] - r_1^2)$.
Then $\gamma ( (t'', t_0] ) \subset P (x_1, \linebreak[1] t_0, \linebreak[1] r_1, \linebreak[1] - r_1^2)$.
Using the fact that the distance distortion on $P(x_1, \linebreak[1] t_0, \linebreak[1] r_1, \linebreak[1] - r_1^2)$ is bounded by a factor of 10 and Cauchy-Schwarz, we get
\begin{multline*}
 \LL_+ \big( \gamma |_{[t'', t_0]} \big) \geq \int_{t''}^{t_0} \sqrt{t_0 - t_*} \; |\gamma' |^2_{t_*} (t_*) dt_* \displaybreak[1] \\
\geq \frac1{100} \bigg ( \int_{t''}^{t_0} \sqrt{t_0 - t_*} \; |\gamma' |^2_{t_0} (t_*) dt_* \bigg) \bigg( \int_{t''}^{t_0} \frac1{\sqrt{t_0 - t_*}} dt_* \bigg) \cdot \frac1{2 \sqrt{t_0 - t''}} \displaybreak[1] \\
\geq \frac1{200 \sqrt{\tau r_1^2}} \bigg( \int_{t''}^{t_0} |\gamma' |_{t_0} (t_*) dt_* \bigg)^2 
\geq \frac{r_1^2}{200 \sqrt{\tau r_1^2}} \displaybreak[1] \\
= \frac{r_1}{200 \sqrt{\tau}} > 2 (\Lambda + \tfrac32) \frac{r_1}r >  (\Lambda + \tfrac32) \sqrt{t_0}  >  (\Lambda + \tfrac32) r_0 ,
\end{multline*}
in contradiction to (\ref{eq:LLplussmall}).
So we have verified (\ref{eq:gammastaysinP}).

Next we consider the case in which $\gamma |_{[t', t_0]}$ exits $P$ through its final time-slice.
By this we mean that $\gamma ([t', t' + \sigma \lambda^2]) \subset P$ (note that $P \cap \MM ( t' + \sigma \lambda^2)$ is open).
This is only possible if the point $\gamma (t' + \sigma \lambda^2)$ survives past time $t' + \sigma \lambda^2$, which implies by our earlier discussion that $\sigma = 1 - \theta$.
Moreover, due to (\ref{eq:gammastaysinP}), we must have $t' + (1-\theta) \lambda \leq t_0 - \tau r_1^2$.
Using (\ref{eq:lowerscaltheta}), we can now compute that
\begin{multline*}
  \LL_+ \big( \gamma |_{[t', t' + \sigma \lambda^2]} \big) \geq \int_{t'}^{t' + \sigma \lambda^2} \sqrt{t_0 - t_*} \; \scal (\gamma (t_*), t_*) dt_* \\
>  \int_{t'}^{t' + (1-\theta) \lambda^2} \frac{ c \lambda^{-2} \sqrt{\tau r_1^2}}{1 - \lambda^{-2} (t_*- t')} dt_* 
= c r_1 \sqrt{\tau}  \int_0^{1-\theta} \frac{ 1}{1 - u} du \\ = c r_1 \sqrt{\tau} |{\log \theta}| > \frac{c}2 r \sqrt{t_0} \sqrt{\tau} |{\log \theta}| > (\Lambda + \tfrac32 ) r_0,
\end{multline*}
in contradiction to (\ref{eq:LLplussmall}).
So $\gamma |_{[t', t_0]}$ cannot exit $P$ through its final time-slice.

It follows that $\gamma$ exists $P$ through the boundary $\partial B( y', t', D \lambda) \times [t', t' + \sigma \lambda^2]$.
In other words, there is some $\sigma' \in (0, \sigma] \subset (0, 1- \theta]$ such that for $t'' = t' + \sigma' \lambda^2$, we have $\gamma ( [t', t'') ) \subset P$ and $\gamma (t'') \in \partial B( y', t', D \lambda)$.
By (\ref{eq:uppercurvboundtheta}), we have $|{\Rm}| < C_\theta  \lambda^{-2}$ on $P$.
So distance elements on $P$ are distorted by a factor of at most $\exp ( C_\theta )$.
We can now estimate
\begin{multline*}
\LL_+ \big(  \gamma |_{[t', t'']} \big) \geq \int_{t'}^{t' + \sigma' \lambda^2} \sqrt{t_0 - t} \; |\gamma' |^2_{t_*} (t_*) dt_*
\geq \frac{\sqrt{\tau r_1^2}}{\exp( 2C_\theta)} \int_{t'}^{t' + \sigma' \lambda^2}  |\gamma' |^2_{t'} (t_*) dt_* \\
\geq \frac{\sqrt{\tau r_1^2}}{\exp (2C_\theta ) \sigma' \lambda^2} \bigg ( \int_{t'}^{t' + \sigma' \lambda^2} |\gamma' |_{t'} (t_*) dt_* \bigg)^2 
\geq  \frac{\sqrt{\tau r_1^2}}{\exp ( 2C_\theta ) \lambda^2} (D \lambda)^2 \\
 \geq \frac{r \sqrt{\tau}D}{2 \exp ( 2C_\theta )} \sqrt{t_0} 
> (\Lambda + \tfrac32) r_0.
\end{multline*}
Hence we obtain another contradiction to (\ref{eq:LLplussmall}).
This finishes the proof of the claim.
\end{proof}

We can now carry out the main argument.
Recall that for all $t \in [t_0 - r_0^2, t_0]$, we have
\[ \scal (\cdot, t) \geq - \frac{3}{2t} \geq - 3 r_0^{-2}. \]
So
\[ \ov{L}(\cdot, t) \geq - 6 \sqrt{t_0 - t} \int_{t_0 - t}^{t_0} r_0^{-2} \sqrt{t_0 - t'} dt' = - 4 r_0^{-2} (t_0 - t)^2. \]
Hence, for $t \in [t_0 - \frac14 r_0^2, t_0]$ we have 
\[ \widehat{L}(\cdot, t) := \ov{L}(\cdot, t) + 2 r_0 \sqrt{t_0 - t} > r_0 \sqrt{t_0 - t} > 0. \]
Let $\phi$ be a cutoff function that is equal to $1$ on $(-\infty, \frac1{20}]$, equal to $\infty$ on $[\frac1{10}, \infty)$ and everywhere greater or equal to $1$ and that satisfies
\[ 2 \frac{(\phi')^2}{\phi} - \phi '' \geq (2A + 300) \phi' - C(A) \phi. \]
Here $C(A) < \infty$ is a positive constant, which only depends on $A$.
For more details see \cite[sec 28]{KLnotes}.
Then set for all $t \in [t_0 - \frac14 r_0^2, t_0]$ and $y \in \MM (t)$
\[ h(y, t) = \phi \big( r_0^{-1} \dist_t (x_0, y) - A(1 - 2 r_0^{-2} (t_0 -  t) ) \big) \widehat{L}(y,t). \]
So $h(\cdot, t)$ is infinite outside $B(x_0, t, A(1 - 2(t_0 - t) r_0^{-2}) r_0 + \frac1{10} r_0) \subset \MM (t) \setminus \partial \MM (t)$ (compare with assumption (ii)) and hence it attains a minimum $h_0(t)$ at some interior point $y \in \MM (t)$.

Assume first that $h(y, t) \leq 2 r_0 \sqrt{t_0 - t} \exp (C(A) + 100)$.
So $L(y, t) \leq r_0 \exp \linebreak[2] (C(A) + \linebreak[1] 100)$.
Then by Claim 2, assuming $\delta < \delta^* (\exp (C(A) + 100), r, \eta)$ and $Z > \exp (C(A) + 100)$, we have $y \in D_t$ and we can compute (cf \cite[6.3]{PerelmanII}, \cite[sec 85]{KLnotes}) that in the barrier sense
\[ r_0^2 \Big( \frac{\partial}{\partial t^-} - \triangle \Big) h(y,t) \geq - C(A) h(y,t) - \Big( 6 + \frac{r_0}{\sqrt{t_0 - t}} \Big) \phi r_0^2.  \]
So by the maximum principle we have in the barrier sense (compare with \cite[(85.7)]{KLnotes})
\[ r_0^2 \frac{d}{d t^-} \Big( \log \frac{h_0(t)}{\sqrt{t_0 - t}} \Big) \geq - C(A) - \frac{50 r_0}{\sqrt{t_0 - t}} \]
if $h_0(t) \leq 2 r_0 \sqrt{t_0 - t} \exp (C(A) + 100)$.

Since $\frac{h_0(t)}{r_0 \sqrt{t_0 - t}} \to 2$ for $t \to t_0$, we find that if $h_0 (t') \leq 2 r_0 \sqrt{t_0 - t'} \exp (C(A) + 100)$ for all $t' \in [t, t_0]$, then
\begin{multline*}
 h_0 (t) \leq 2 r_0 \sqrt{t_0 - t} \exp \big( C(A) r_0^{-2} (t_0 - t) + 100 r_0^{-1} \sqrt{t_0 - t} \big) \\
  < 2 r_0 \sqrt{t_0 - t} \exp (C(A) + 100).
\end{multline*}
This implies that the assumption $h_0 (t) \leq 2 r_0 \sqrt{t_0 - t} \exp (C(A) + 100)$ is actually satisfied for all $t \in [t_0 - \frac14 r_0^2, t_0]$.
So we can find a $y \in B(x_0, t_0 - \frac14 r_0^2, \frac1{10} r_0)$ such that $L(y, t_0 - \frac14 r_0^2) \leq r_0 \exp (C(A) + 100) = C'(A) r_0$.

Since by length distortion estimates $B(x_0, t_0 - \frac14 r_0^2, \frac1{10} r_0) \subset B(x_0, t_0, \frac12 r_0)$, we find by joining paths that for all $x \in B(x_0, t_0, r_0)$ we have $L(x, t_0 - r_0^2) < C''(A) r_0$.
So, assuming $\delta < \delta^* (C''(A), r, \eta)$ and $Z > C''(A)$, we can use Claim 2 to conclude that $P(x_0, t_0, r_0, - r_0^2) \cap \MM(t_0 - r_0^2) \subset D_{t_0 - r_0^2}$ and we have
\[ \widetilde{V}(t_0 - r_0^2) > v(w, A) \]
for some constant $v(w, A) > 0$, which only depends on $A$ and $w$.
This implies a uniform lower bound on $r_1^{-3} \vol_{t_0} B(x_1, t_0, r_1)$ (cf \cite[7.3]{PerelmanI}, \cite[Theorem 26.2]{KLnotes}, \cite[Lemma 4.2.3]{Bamler-diploma}).
\end{proof}

The noncollapsing result from Lemma \ref{Lem:6.3a} will be applied in Lemma \ref{Lem:6.3bc} below.
Before we continue, we introduce the concept of $\kappa$-solutions, which will be used as models for singularities and for regions of high curvature.
The definition makes sense in all dimensions.

\begin{Definition}[$\kappa$-solution]
Let $\kappa > 0$.
An ancient Ricci flow $(M, (g_t)_{t \in (-\infty, 0]})$ is called a \emph{$\kappa$-solution} if
\begin{enumerate}[label=(\arabic*)]
\item The curvature is uniformly bounded on $M \times (-\infty, 0]$.
\item The metric on every time-slice is complete and has non-negative curvature operator (i.e. it has non-negative sectional curvature in dimension $3$).
\item The scalar curvature at time $0$ is positive.
\item At every point the scalar curvature is non-decreasing in time.
\item The solution is $\kappa$-noncollapsed on all scales at all points.
\end{enumerate}
\end{Definition}
Note that, by Hamilton's Harnack inequality (cf \cite{Hamilton-Harnack}), condition (4) follows from conditions (1), (2).

We also mention that there is a universal $\kappa_0 > 0$ such that every $3$ dimensional $\kappa$-solution that is not round (i.e. isometric to a quotient of a round sphere), is in fact a $\kappa_0$-solution (cf \cite[11.9]{PerelmanI}, \cite[Proposition 50.1]{KLnotes}).
$\kappa$-solutions can be used to detect strong $\varepsilon$-necks or $(\varepsilon, E)$-caps or, more generally, to verify the canonical neighborhood assumptions, as explained in the next Lemma.

\begin{Lemma} \label{Lem:kappasolCNA}
There is an $\eta > 0$ and for any $\varepsilon > 0$ there is an $E = E(\varepsilon) < \infty$ such that for every orientable $3$ dimensional $\kappa$-solution $(M, (g_t)_{t \in (-\infty, 0]})$ the following holds:
For all $r > 0$, the canonical neighborhood assumptions $CNA (r, \varepsilon, E, \eta)$ hold everywhere on $M \times (- \infty, 0]$.
More precisely, $M$ is diffeomorphic to a spherical space form and has positive, $E^2$-pinched sectional curvatures or for any $(x,t) \in M \times (-\infty, 0]$ we have:
\begin{enumerate}[label=(\alph*)]
\item $(x,t)$ is a center of a strong $\varepsilon$-neck or an $(\varepsilon, E)$-cap $U \subset M$. \\ 
If $U \subset \IR P^3 \setminus \ov{B}^3$, then there is a double cover of $M$ such that any lift of $(x,t)$ is the center of a strong $\varepsilon$-neck.
\item $|\nabla |{\Rm}|^{-1/2}| (x,t) < \eta^{-1}$ and $| \partial_t |{\Rm}|^{-1} | (x,t) < \eta^{-1}$.
\item $\vol_t B(x,t, r') > \eta r^{\prime 3}$ for all $0 < r' < |{\Rm}|^{-1/2} (x,t)$.
\end{enumerate}
\end{Lemma}

\begin{proof}
See \cite[11.8]{PerelmanI}, \cite[Corollary 48.1]{KLnotes} or \cite[Theorem 5.4.11]{Bamler-diploma}.
\end{proof}

The following Lemma will enable us to identify $\kappa$-solutions as limits of Ricci flows with surgeries under very weak curvature bounds.
We first need to coin the following notion.

\begin{Definition}[convergence of pointed Ricci flows with surgery]
Let $\MM^\alpha$, $\alpha = 1, 2, \ldots$, be a sequence of Ricci flows with surgery and let $(x^\alpha, t^\alpha) \in \MM^\alpha$ be basepoints.
Furthermore, consider a constant $0 < T \leq \infty$, a non-singular Ricci flow $(M^\infty, (g^\infty_t)_{t \in (- T, 0]})$ and a basepoint $(x^\infty, t^\infty) \in M^\infty \times (-T, 0]$.
We say that the \emph{pointed Ricci flows with surgery $(\MM^\alpha, (x^\alpha, t^\alpha))$ converge to the pointed Ricci flow $(M^\infty, \linebreak[1] (g^\infty_t)_{t \in (- T, 0]}, \linebreak[1] (x^\infty, t^\infty))$}, if the following holds:
We can find an increasing sequence of open subsets $U^\alpha \subset M^\infty$, open subsets $V^\alpha \subset \MM^\alpha (t^\alpha)$, diffeomorphisms $\Phi^\alpha : U^\alpha \to V^\alpha$, and numbers $0 < T^\alpha < T$ such that the following holds:
\begin{enumerate}[label=(\arabic*)]
\item $\lim_{\alpha \to \infty} T^\alpha = T$.
\item $\bigcup_{\alpha = 1}^\infty U^\alpha = M^\infty$.
\item For any $\alpha$, all points of $V^\alpha$ survive until time $t^\alpha - T^\alpha$.
In other words, the flow restricted to $V^\alpha \times [t^\alpha - T^\alpha, t^\alpha]$ is non-singular.
\item Denote by $(\ov{g}^\alpha)_{t \in [t^\infty - T^\alpha, t^\infty]}$ the pullbacks $\ov{g}^\alpha_{t^\infty + t} := (\Phi^\alpha)^* g^\alpha (t^\alpha + t)$, $t \in [-T^\alpha, 0]$.
Then $(\ov{g}^\alpha_t)_{t \in [t^\infty - T^\alpha, t^\infty]}$ converges to $(g^\infty_t)_{t \in (t^\infty - T, t^\infty]}$ locally in any $C^m$-norm on $M^\infty \times (- T, 0]$.
\end{enumerate}
\end{Definition}

Note that in the case in which all flows $\MM^\alpha$ are non-singular, this notion coincides with the smooth convergence of Ricci flows as introduced by Hamilton (cf \cite{Hamilton-compactness}).

\begin{Lemma} \label{Lem:lmiitswithCNA}
There is an $\varepsilon_0 > 0$ such that:
Let $\MM^\alpha$ be a sequence of ($3$ dimensional) Ricci flows with surgery on the time-intervals $[- \tau_0^\alpha, 0]$, $\tau^\alpha \leq \tau_0^\alpha$, $x^\alpha_0 \in \MM^\alpha (0)$ a sequence of basepoints that survive until time $- \tau^\alpha$, and $a^\alpha \to \infty$ a sequence of positive numbers such that for $P^\alpha = \{ (x,t) \in \MM^\alpha \;\; : \;\; t \in [- \tau^\alpha, 0], \; \dist_t(x_0^\alpha, x) < a^\alpha \}$ the following conditions hold:
\begin{enumerate}[label=(\roman*)]
\item the ball $B(x_0^\alpha, t, a^\alpha)$ is relatively compact in $\MM^\alpha (t)$ and does not hit the boundary $\partial \MM^\alpha (t)$ for all $t \in [- \tau^\alpha, 0]$,
\item $|{\Rm}|(x^\alpha,0) \leq 1$,
\item the curvature on $P^\alpha$ is $\varphi^\alpha$-positive for some $\varphi^\alpha \to 0$,
\item all points of $P^\alpha$ are $\kappa$-noncollapsed on scales $< a^\alpha$ for some uniform $\kappa > 0$,
\item all points on $P^\alpha$ satisfy the canonical neighborhood assumptions $CNA(r, \linebreak[1] \varepsilon_0, \linebreak[1] E, \eta)$ for some uniform $r, E, \eta > 0$,
\item there is a sequence $K^\alpha \to \infty$ such that for every surgery point $(x', t') \in P^\alpha$ we have $|{\Rm}| (x', t') > K^\alpha$.
\end{enumerate}
Then whenever $\tau^\infty = \limsup_{\alpha \to \infty} \tau^\alpha > 0$, a subsequence of the pointed Ricci flows with surgery $(\MM^\alpha, (x_0^\alpha, 0))$ converges to some pointed, non-singular Ricci flow $(M^\infty, \linebreak[1] (g_t^\infty)_{t \in (- \tau^\infty, 0]}, \linebreak[1] (x_0^\infty,0))$ on a manifold $M^\infty$ without boundary.
Moreover, this limiting Ricci flow has complete time-slices and bounded, non-negative sectional curvature.
If $\tau^\infty = \infty$ and $|{\Rm}|(x^\infty, 0) > 0$, then $(M^\infty, \linebreak[1] (g_t^\infty)_{t \in (- \infty, 0]})$ is a $\kappa$-solution.
\end{Lemma}

\begin{proof}
We follow the lines of the proofs of \cite[Proposition 6.3.1]{Bamler-diploma}, \cite[4.2]{PerelmanII}, \cite[12.1]{PerelmanI} and \cite[Theorem 52.7]{KLnotes}.

We first use assumptions (i), (iv), (v) at time $0$ and assumptions (ii), (iii) to apply Perelman's ``bounded curvature at bounded distance estimate''.
For more details see \cite[Claim 2, 4.2]{PerelmanII}, the proof of \cite[Lemma 89.2]{KLnotes} or \cite[Lemma 70.2]{KLnotes} or \cite[Proposition 6.2.4]{Bamler-diploma}.
In order to carry out this estimate, we need to assume that $\varepsilon$ is smaller than some universal constant $\varepsilon_0 > 0$.
The ``bounded curvature at bounded distance estimate'' yields a function $K_1^* : [0, \infty) \to (0, \infty)$ such that for every $d > 0$ we have
\[ |{\Rm}| (\cdot, 0) < K_1^* (d) \qquad \text{on} \qquad B(x^\alpha_0, 0, d) \subset \MM^\alpha (0) \]
for large $\alpha$ (depending on $d$).
Using Lemma \ref{Lem:shortrangebounds}(b) and assumption (v), we obtain functions $\tau^*_2, K^*_2 : [0, \infty) \to (0,\infty)$ such that for all $d > 0$ we have
\[ |{\Rm}| < 2 K^{*}_2 (d) \qquad \text{on} \qquad P(x_0^\alpha, 0, d, - \tau^*_2 (d)) \]
for large $\alpha$ (depending on $d$).
By assumption (vi), this implies that for any $d > 0$ and large $\alpha$, the parabolic neighborhood $P(x_0^\alpha, 0, d, - \tau^*_2 (d))$ is non-singular.
So we obtain uniform bounds on the curvature derivatives on slightly smaller parabolic neighborhoods.
This and assumption (iv) implies that, after passing to a subsequence, the pointed Riemannian manifolds $(\MM^\alpha (0), x^\alpha_0)$ converge to a complete pointed Riemannian manifold $(M^\infty, g^\infty, x^\infty_0)$ in the smooth Cheeger-Gromov sense.
In the following, we will only work with this subsequence.

By assumption (iii), we conclude that $(M^\infty, g_\infty)$ has non-negative sectional curvature.
Moreover, by assumption (v), we find that any point $x \in M^\infty$ with $|{\Rm}| (x) > 2 E^2 r^{-2}$ is the center of a $2\eps$-neck or a $(2\eps, 2E)$-cap.
This fact implies that the curvature on $(M^\infty, g_\infty)$ is uniformly bounded.
For more details see the proof of \cite[Theorem 52.7]{KLnotes} (see also step 3 in the proof of \cite[Theorem 52.7]{KLnotes}), the proof of \cite[Proposition 6.3.1]{Bamler-diploma} or the second paragraph on page 34 of \cite{PerelmanI}.

So there is a constant $K^*_3 < \infty$ such that for any $d > 0$ we have
\[ |{\Rm}|(\cdot, 0) < K^*_3 \qquad \text{on} \qquad B(x_0^\alpha, 0, d) \subset \MM^\alpha (0) \]
for sufficiently large $\alpha$ (depending on $d$).
Again, by Lemma \ref{Lem:shortrangebounds} and assumption (v), we obtain constants $\tau^*_4 > 0$ and $K^*_4 < \infty$ such that for all $d > 0$ we have
\[ |{\Rm}| < K^*_4 \qquad \text{on} \qquad P(x_0^\alpha, 0, d, - \tau^*_4) \]
for sufficiently large $\alpha$ (depending on $d$).
So, again by assumption (vi), for large $\alpha$ (depending on $d$) the parabolic neighborhood $P(x_0^\alpha, 0, d, - \tau^*_4)$ is non-singular.

Choose now $0 < \tau^* \leq \tau^\infty$ maximal with the following property:
After possibly passing to a subsequence, the following holds:
For any $0 < \tau^{**} < \tau^*$ there is a constant $K_{\tau^{**}} < \infty$ such that for all $d > 0$ we have
\[ |{\Rm}| < K^*_{\tau^{**}} \qquad \text{on} \qquad P(x_0^\alpha, 0, d, - \tau^{**}) \]
for large $\alpha$ (depending on $\tau^{**}$ and $d$).
By our previous conclusions, $\tau^* \geq \tau^*_4 > 0$.
It follows that we can pick sequences $d^\alpha \to \infty$ and $\tau^{**}_\alpha \to \tau^*$ such that the parabolic neighborhoods $P(x_0^\alpha, 0, d^\alpha, - \tau^{**}_\alpha)$ are non-singular.
So we can apply Hamilton's compactness theorem for (non-singular) Ricci flows to conclude that the pointed Ricci flows with surgery $(\MM^\alpha, (x^\alpha_0, 0))$ subconverge to some non-singular Ricci flow $(M^\infty, (g_t^\infty)_{t \in (-\tau^*, 0]}, (x^\infty_0, 0))$ with the property that $g_0^\infty = g^\infty$.
Moreover, $(M^\infty, (g_t^\infty)_{t \in (-\tau^*, 0]})$ has bounded curvature on compact time-intervals, complete time-slices and, by assumption (iii), non-negative sectional curvature.

Next we show that $(M^\infty, (g_t^\infty)_{t \in (-\tau^*, 0]})$ has uniformly bounded curvature.
In the case in which $\tau^* = \infty$, this fact follows from Hamilton's Harnack inequality (cf \cite{Hamilton-Harnack}), which implies that the scalar curvature is pointwise non-decreasing along the flow.
In the case in which $\tau^* < \infty$, we can argue as in step 4 of the proof of \cite[Theorem 52.7]{KLnotes}.
See also the proof of \cite[Proposition 6.3.1]{Bamler-diploma} or the third paragraph on page 34 of \cite{PerelmanI}.
For these proofs the following statement, which follows from assumption (v), is important: any $(x,t) \in M^\infty \times (-\tau^*, 0]$ with $|{\Rm}| (x,t) > 2 r^{-2}$ is the center of a $2\varepsilon$-neck or a $(2\varepsilon, 2E)$-cap.

So it follows that, after passing to a subsequence once again, there is a uniform constant $K^*_5 < \infty$ such that for all $0 < \tau^{**} < \tau^*$ and all $d > 0$ we have
\[ |{\Rm}| < K^*_5 \qquad \text{on} \qquad P(x_0^\alpha, 0, d, - \tau^{**}) \]
for sufficiently large $\alpha$ (depending on $\tau^{**}$ and $d$).
Now assume that $\tau^* < \tau^\infty$.
Then, using Lemma \ref{Lem:shortrangebounds}(b) and assumption (v), we can find some $\tau^*_6$ with $\tau^* < \tau^*_6 < \tau^\infty$ and some $K^*_6 < \infty$ such that for any $d > 0$ we have
\[ |{\Rm}| < K^*_6 \qquad \text{on} \qquad P(x_0^\alpha, 0, d, - \tau^*_6) \]
for sufficiently large $\alpha$ (depending on $d$).
This, however, contradicts the choice of $\tau^*$.
So we conclude that indeed $\tau^* = \tau^\infty$.

It remains to consider the case in which $\tau^\infty = \infty$ and $|{\Rm}|(x^\infty, 0) > 0$.
Note that in this case $(M^\infty, (g^\infty_t)_{t \in (-\infty, 0]})$ is an ancient solution with uniformly bounded, non-negative sectional curvature and complete time-slices.
Since $|{\Rm}| \linebreak[1] (x^\infty, \linebreak[1] 0) > 0$, we must have $\scal (x^\infty, t) > 0$ for some $t < 0$.
So by the strong maximum principle, we have $\scal (\cdot, 0) > 0$ everywhere on $M^\infty$.
The fact that the scalar curvature is pointwise non-decreasing in time follows from Hamilton's Harnack inequality (cf \cite{Hamilton-Harnack}) and the fact that $(M^\infty, (g^\infty_t)_{t \in (-\infty, 0]})$ is $\kappa$-noncollapsed on all scales at all points is a consequence of assumption (iv).
\end{proof}

We now state the second main Lemma.

\begin{Lemma}[cf \hbox{\cite[6.3(b)+(c)]{PerelmanII}}] \label{Lem:6.3bc}
There are constants $\eta_0, \varepsilon_0 > 0$ and for every $\varepsilon \in (0, \varepsilon_0]$ there is a constant $E_0 = E_0(\varepsilon) < \infty$ such that: \\
For any $1 \leq A < \infty$, $w, r > 0$, $\eta \in (0, \eta_0]$ and $E \geq E_0$ there are constants $K = K(w, A, E, \eta), Z = Z(A) < \infty$ and $\td\rho = \td\rho (w, A, \varepsilon, E, \eta), \ov{r}=\ov{r}(A, w, E, \eta), \delta = \delta (w, A, r, \varepsilon, E, \eta) > 0$ such that: \\
Let $r_0^2 \leq t_0/2$ and let $\MM$ be a Ricci flow with surgery (whose time-slices are allowed to have boundary) on the time-interval $[ t_0 - r_0^2, t_0]$ that is performed by $\delta$-precise cutoff and consider a point $x_0 \in \MM(t_0)$.
Assume that the canonical neighborhood assumptions $CNA (r \sqrt{t_0}, \varepsilon, E, \eta)$ hold on $\MM$.
We also assume that the curvature on $\MM$ is uniformly bounded on compact time-intervals that don't contain surgery times and that all time-slices of $\MM$ are complete.

Assume that the parabolic neighborhood $P(x_0, t_0, r_0, -r_0^2)$ is non-singular, that $|{\Rm}| \leq r_0^{-2}$ on $P(x_0, t_0, r_0, - r_0^2)$ and $\vol_{t_0} B(x_0, t_0, r_0) \geq w r_0^3$.

In the case in which some time-slices of $\MM$ have non-empty boundary, we assume that
\begin{enumerate}[label=(\roman*)]
\item every space-time curve $\gamma : [t_1, t_2] \to \MM$ with $t_2 \in [t_0 - \frac1{10} r_0^2, t_0]$ and $\gamma(t_2) \in B(x_0, t_2, (A+1) r_0)$ that meets the boundary $\partial \MM$ somewhere, has $\LL(\gamma) > Z r_0$ (based in $t_2$),
\item for all $t \in [t_0 - \frac15 r_0^2, t_0]$, the ball $B(x_0, t, 2 (A+3)r_0 + r \sqrt{t_0})$ does not meet the boundary $\partial \MM (t)$.
\end{enumerate}

Then
\begin{enumerate}[label=(\alph*)]
\item Every point $x \in B(x_0, t_0, A r_0)$ satisfies the canonical neighborhood assumptions $CNA( \td\rho r_0, \varepsilon, E , \eta )$.
\item If $r_0 \leq \ov{r} \sqrt{t_0}$, then $|{\Rm}| \leq K r_0^{-2}$ on $B(x_0, t_0, A r_0)$.
\end{enumerate}
\end{Lemma}

It is important in this lemma that $\td\rho$, unlike $\delta$, may not depend on $r$.

\begin{proof}
The proof follows the lines of \cite[6.3(b), (c)]{PerelmanII}.

Let $\varepsilon_0$ be the constant from Lemma \ref{Lem:lmiitswithCNA}.
Choose $\eta_0$ and $E_0 = E_0 (\varepsilon)$ to be strictly less/larger than the constants $\eta$, $E(\varepsilon)$ in Lemma \ref{Lem:kappasolCNA}.
By choosing $\td\rho$ small and $K$ large enough, we can again exclude the case in which  for some time $t \leq t_0$ the component of $\MM(t)$ that contains $x_0$ has positive, $E^2$-pinched sectional curvatures.

We first establish part (a).
Assume that, given some small $\td\rho$, there is a point $x \in B(x_0, t_0, A r_0)$ such that $(x, t_0)$ does not satisfy the canonical neighborhood assumptions $CNA(\td\rho r_0, \varepsilon, E, \eta)$, i.e. we have $|{\Rm}|(x, t_0) \geq \td{\rho}^{-2} r_0^{-2}$ and $(x, t_0)$ does not satisfy the assumptions (1)--(3) in Definition \ref{Def:CNA}.
Set for $\ov{t} \in [t_0 - r_0^2, t_0]$, $\ov{x} \in \MM(\ov{t})$
\begin{multline*}
 P_{\ov{x}, \ov{t}} = \big\{ (y,t) \in \MM \;\; : \;\; t \in  [\ov{t} - \tfrac1{20} \td{\rho}^{-2} |{\Rm}|^{-1} (\ov{x},\ov{t}), \ov{t}], \;\; y \in \MM(t), \\
 \;\; \dist_t (x_0, y) \leq \dist_{\ov{t}}(x_0, \ov{x}) + \tfrac14 \td{\rho}^{-1} |{\Rm}|^{-1/2}(\ov{x},\ov{t}) \big\}.
\end{multline*}
We will now find a  particular $(\ov{x}, \ov{t}) \in \MM$ with $\ov{t} \in [t_0 - \frac1{10} r_0^2, t_0]$ and $\ov{x} \in B(x_0, \ov{t}, (A+\frac12) r_0)$ by a point-picking process:
Set first $(\ov{x}, \ov{t}) = (x, t_0)$.
Let $\ov{q} = |{\Rm}|^{-1/2} (\ov{x}, \ov{t}) \leq \td{\rho} r_0$.
If every $(x',t') \in P_{\ov{x}, \ov{t}}$ satisfies the canonical neighborhood assumptions $CNA(\frac12 \ov{q}, \varepsilon, E, \eta)$, then we stop.
If not, we replace $(\ov{x}, \ov{t})$ by such a counterexample and start over.
In every step of this algorithm, $\ov{q}$ decreases by at least a factor of $\frac12$, which implies that the algorithm has to terminate after a finite number of steps since after a finite number of steps we have $\ov{q} < r \sqrt{t_0}$ and we can make use of the canonical neighborhood assumptions $CNA(r \sqrt{t_0}, \varepsilon, E, \eta)$ from the assumptions of the lemma.
So the algorithm yields an $(\ov{x}, \ov{t}) \in \MM$ and a $\ov{q} = |{\Rm}|^{-1/2} (\ov{x}, \ov{t}) \leq \td\rho r_0$ such that $(\ov{x}, \ov{t})$ does not satisfy the canonical neighborhood assumptions $CNA (\ov{q}, \varepsilon, E, \eta)$, but all points in $P_{\ov{x}, \ov{t}}$ satisfy the canonical neighborhood assumptions $CNA( \frac12 \ov{q}, \varepsilon, E, \eta)$.
By convergence of the geometric series, we conclude $\ov{t} - \frac1{20} \td\rho^{-2} \ov{q}^2 \in [t_0 - \frac1{10} r_0^2, t_0]$ and $\dist_{\ov{t}}(x_0, \ov{x}) < (A + \frac12) r_0$.
Moreover, for all $(x', t') \in P_{\ov{x}, \ov{t}}$ we have $\dist_{t'} (x_0, x') < (A+1) r_0$.

We now claim that there is a constant $\td\rho = \td\rho (w, A, \varepsilon, E, \eta) > 0$ such that assertion (a) holds for $Z (A) = Z_{\ref{Lem:6.3a}} ( 10 (A+1) )$ and
\[ \delta = \delta ( w, A, r, \varepsilon, E, \eta ) = \min \big\{ \delta_{\ref{Lem:6.3a}} \big( 10 (A+1), r, \eta \big), r^2 \big\}, \]
where $Z_{\ref{Lem:6.3a}}$ and $\delta_{\ref{Lem:6.3a}}$ are the constants from Lemma \ref{Lem:6.3a}.
Assume that this was wrong, i.e. that for fixed parameters $w, A, \varepsilon, E, \eta$, there is no such constant $\td\rho$.
Then we can find a sequence $\td\rho^\alpha \to 0$ and a sequence of counterexamples $\MM^\alpha$, $t_0^\alpha$, $r_0^\alpha$, $x_0^\alpha$, $r^\alpha$ that satisfy the assumptions of the Lemma for $Z = Z ( A )$ and $\delta^\alpha = \delta (w, A, r^\alpha, \varepsilon, E, \eta)$, but for which there are points $x^\alpha \in B(x_0^\alpha, t_0^\alpha, A r_0^\alpha)$ such that $(x^\alpha, t_0^\alpha)$ doesn't satisfy the canonical neighborhood assumptions $CNA (\td{\rho}^\alpha r_0^\alpha, \linebreak[1] \varepsilon, \linebreak[1] E, \linebreak[1] \eta)$.
Note that by assumption, the point $(x^\alpha, t^\alpha_0)$ satisfies the canonical neighborhood assumptions $CNA ( r^\alpha \sqrt{t^\alpha_0}, \varepsilon, E, \eta)$.
So we must have $\td\rho^\alpha r^\alpha_0 > r^\alpha \sqrt{t^\alpha_0} > r^\alpha r^\alpha_0$ and hence $r^\alpha \to 0$ for $\alpha \to \infty$.
By the choice of $\delta$ this implies that that $\delta^\alpha / r^\alpha \to 0$ for $\alpha \to \infty$.

First, let $(\ov{x}^\alpha, \ov{t}^\alpha)$ and $\ov{q}^\alpha$ be the point and the constant obtained by the algorithm two paragraphs earlier.
We now apply Lemma \ref{Lem:6.3a} with 
\begin{multline*}
 r_0 \leftarrow \tfrac1{10} r_0^\alpha, \; x_0 \leftarrow x_0^\alpha, \; t_0 \leftarrow t \in [\ov{t}^\alpha - \tfrac1{20} (\td\rho^\alpha)^{-2} (\ov{q}^\alpha)^{2}, \ov{t}^\alpha], \\ \; w \leftarrow c w, \; A \leftarrow 10(A+1), \; r \leftarrow r^\alpha,
\end{multline*}
where $c > 0$ is a universal constant, which arises from volume comparison and distortion estimates on $P( x_0^\alpha, t_0^\alpha, r_0^\alpha, - r_0^\alpha)$ and which has the property that $\vol_t B(x_0^\alpha, \linebreak[1] t, \linebreak[1] \frac1{10} r_0^\alpha) > c w ( \frac1{10} r_0^\alpha )^3$.
We conclude that any $(x',t') \in \MM^\alpha$ with $t' \in [\ov{t}^\alpha - \frac1{20} (\td\rho^\alpha)^{-2} (\ov{q}^\alpha)^2, \ov{t}^\alpha]$ and $x' \in B(x_0^\alpha, t', (A+1) r_0^\alpha)$ is $\kappa$-noncollapsed for some uniform $\kappa > 0$ on scales less than $\frac1{10} r_0^\alpha$.
This implies that the points on $P_{\ov{x}^\alpha, \ov{t}^\alpha}$ are $\kappa$-noncollapsed on scales less than $\frac1{10} r_0^\alpha$.

Observe that the assumption on $\delta^\alpha$ and Definition \ref{Def:precisecutoff}(3) imply that there is a universal constant $c' > 0$ such that for every surgery point $(x',t') \in \MM^\alpha$ with $t' \leq t^\alpha_0$ we have 
\begin{multline} \label{eq:surgpointhighcurv}
 |{\Rm}|(x',t') > c' (\delta^\alpha)^{-2} t^{\prime -1} = c' \Big( \frac{\delta^\alpha}{r^\alpha} \Big)^{-2} \big( r^\alpha \sqrt{t'} \big)^{-2} \\ \geq c' \Big( \frac{\delta^\alpha}{r^\alpha} \Big)^{-2} \big( r^\alpha \sqrt{t^\alpha_0} \big)^{-2} >  c' \Big( \frac{\delta^\alpha}{r^\alpha} \Big)^{-2} (\ov{q}^\alpha)^{-2}.
\end{multline}
Here we have again made use of the inequality $\ov{q}^\alpha > r^\alpha \sqrt{t_\alpha}$, which follows from the fact that the point $(\ov{x}^\alpha, \ov{t}^\alpha)$ satisfies the canonical neighborhood assumptions $CNA (r^\alpha \sqrt{t^\alpha_0}, \linebreak[1] \varepsilon, \linebreak[1] E, \linebreak[1] \eta)$, but not $CNA (\ov{q}^\alpha, \varepsilon, E, \eta)$.
Recall moreover, that the factor $(\delta^\alpha / r^\alpha)^{-2} \to \infty$ as $\alpha \to \infty$.

So for large $\alpha$ the point $(\ov{x}^\alpha, \ov{t}^\alpha)$ is not a surgery point.
Pick $0 < \tau^\alpha \leq \frac1{20} (\td{\rho}^\alpha)^{-2}$ maximal such that the point $\ov{x}^\alpha$ survives until time $\ov{t}^\alpha - \tau^\alpha (\ov{q}^\alpha)^2$ and such that
\[ \dist_t (x_0^\alpha, \ov{x}^\alpha ) < \dist_{\ov{t}^\alpha_0} (x^\alpha_0, \ov{x}^\alpha ) + \tfrac18 (\td{\rho}^\alpha)^{-1} \ov{q}^\alpha \qquad \text{for all} \qquad t \in \big( \ov{t}^\alpha - \tau^\alpha ( \ov{q}^\alpha)^2, \ov{t}^\alpha \big]. \]
This implies
\begin{equation} \label{eq:BisinsideP} 
B(\ov{x}^\alpha, t, \tfrac18 (\td{\rho}^\alpha)^{-1} \ov{q}^\alpha ) \subset P_{\ov{x}^\alpha, \ov{t}^\alpha} \qquad \text{for all} \qquad t \in \big[ \ov{t}^\alpha - \tau^\alpha ( \ov{q}^\alpha)^2, \ov{t}^\alpha \big]. 
\end{equation}
By passing to a subsequence, we may assume that the limit $\tau_\infty = \lim_{\alpha \to \infty} \tau^\alpha \in [0, \infty]$ exists.
So after parabolically rescaling by $(\ov{q}^\alpha)^{-1}$, the Ricci flows with surgery $\MM^\alpha$ restricted to the time-interval $[\ov{t}^\alpha - \tau^\alpha (\ov{q}^\alpha)^2, \ov{t}^\alpha]$ and based at $\ov{x}^\alpha$ satisfy the assumptions of Lemma \ref{Lem:lmiitswithCNA} for some sequence $a^\alpha \to \infty$ (we also need to make use of assumption (ii) here).
Hence, again after passing to a subsequence, these flows subconverge to some non-singular Ricci flow on $M_\infty \times (- \tau_\infty, 0]$ of bounded curvature.

The previous conclusion has the following implication:
There is a uniform constant $4 \leq D < \infty$ such that whenever $0 <\tau' < \tau_\infty$ or $\tau' = 0$, then we have
\begin{equation} \label{eq:curvboundedbyD}
 |{\Rm}| (\ov{x}^\alpha, t) < D (\ov{q}^\alpha)^{-2} \qquad \text{for all} \qquad t \in \big[ \ov{t}^\alpha - \tau' (\ov{q}^\alpha)^2, \ov{t}^\alpha \big]
\end{equation}
for large $\alpha$ (in the case $\tau_\infty = 0$ the statement holds for  $D = 4$).

Assume first that $\tau_\infty < \infty$.
Observe that by (\ref{eq:BisinsideP}) the point $(\ov{x}^\alpha, t)$ satisfies the canonical neighborhood assumptions $CNA (\frac12 \ov{q}^\alpha, \varepsilon, E, \eta)$ for all $t \in [\ov{t}^\alpha - \tau^\alpha (\ov{q}^\alpha)^2, \ov{t}^\alpha]$.
This implies that (cf Definition \ref{Def:CNA}(2))
\begin{multline} \label{eq:dtcurvboundfromcna}
 |{\Rm}| (\ov{x}^\alpha, t) < 4 (\ov{q}^\alpha)^{-2}  \leq D (\ov{q}^\alpha )^{-2} \quad \text{or} \quad  |{\partial_t |{\Rm}|}^{-1}| (\ov{x}^\alpha, t) < \eta^{-1} \\ \text{for all} \qquad t \in \big[ \ov{t}^\alpha - \tau^\alpha ( \ov{q}^\alpha)^2, \ov{t}^\alpha \big].
\end{multline}
We now use (\ref{eq:curvboundedbyD}) for $\tau' = \max \{ \tau_\infty - \frac{\eta}{4 D}, 0 \}$ and integrate the derivative bound of (\ref{eq:dtcurvboundfromcna}) from $\ov{t}^\alpha - \tau' (\ov{q}^\alpha)^2$ backwards in time to any $t  \in [\ov{t}^\alpha - \tau^\alpha (\ov{q}^\alpha)^2, \ov{t}^\alpha - \tau' (\ov{q}^\alpha )^2]$, for large $\alpha$.
Note that for large $\alpha$ and any such $t$, we have $t - (\ov{t}^\alpha - \tau^\alpha (\ov{q}^\alpha )^2 ) \leq (\frac{\eta}{4D} + (\tau^\alpha - \tau_\infty) ) (\ov{q}^\alpha)^2 < 2 \cdot \frac{\eta}{4 D} (\ov{q}^\alpha)^2$.
So we obtain that for large $\alpha$ we have (compare with Lemma \ref{Lem:shortrangebounds})
\[ |{\Rm}| (\ov{x}^\alpha, t) < 2D (\ov{q}^\alpha)^{-2} \qquad \text{for all} \qquad t \in \big[ \ov{t}^\alpha - \tau^\alpha (\ov{q}^\alpha)^2, \ov{t}^\alpha \big]. \]
In particular, it follows from (\ref{eq:surgpointhighcurv}) that for large $\alpha$ none of the points $(\ov{x}^\alpha, t)$ for $t \in [\ov{t}^\alpha - \tau^\alpha (\ov{q}^\alpha)^2, \ov{t}^\alpha]$ are surgery points.
So $(\ov{x}^\alpha, \ov{t}^\alpha)$ even survives past time $\ov{t}^\alpha - \tau^\alpha (\ov{q}^\alpha)^2$.

Next, we use the following consequence of the canonical neighborhood assumptions $CNA ( \frac12 \ov{q}^\alpha, \varepsilon, E, \eta)$, which hold on $P_{\ov{x}, \ov{t}}$:
\begin{multline*}
 |{\Rm}| (x, t) < 4 (\ov{q}^\alpha)^{-2} \leq D (\ov{q}^\alpha)^{-2} \quad \text{or} \quad  |{\nabla|{\Rm}|^{-1/2}}| (x, t) < \eta^{-1} \\ \text{for all} \qquad t \in \big[ \ov{t}^\alpha - \tau^\alpha ( \ov{q}^\alpha)^2, \ov{t}^\alpha \big] \qquad
 \text{and}  \qquad x \in B(\ov{x}^\alpha, t, \tfrac18 (\td\rho^\alpha)^{-1} \ov{q}^\alpha).
\end{multline*}
Integrating these assumptions as in the proof of Lemma \ref{Lem:shortrangebounds} yields that for large $\alpha$
\[ |{\Rm}| < 16 D (\ov{q}^\alpha)^{-2} \quad \text{on} \quad B(\ov{x}^\alpha, t, \tfrac14 \eta D^{-1/2} \ov{q}^\alpha ) \quad \text{for all} \quad t \in \big[ \ov{t}^\alpha - \tau^\alpha (\ov{q}^\alpha)^2, \ov{t}^\alpha \big]. \]
Note that here we have used the fact that $\tfrac18 (\td{\rho}^\alpha)^{-1} > \frac14 \eta D^{-1/2}$ for large $\alpha$.
By distance distortion estimates (Lemma \ref{Lem:distdistortion}(b)) and (\ref{eq:surgpointhighcurv}), we then obtain that for large $\alpha$ and some universal constant $C < \infty$
\[ \frac{d}{dt} \dist_t (x_0^\alpha, \ov{x}^\alpha) \geq - C \eta^{-1} \sqrt{D} (\ov{q}^\alpha)^{-1} \qquad \text{for all} \qquad t \in \big[ \ov{t}^\alpha - \tau^\alpha (\ov{q}^\alpha)^2, \ov{t}^\alpha \big]. \]
Integrating this estimate yields that for large $\alpha$
\[ \dist_t (x_0^\alpha, \ov{x}^\alpha) < \dist_{\ov{t}^\alpha}(x_0^\alpha, \ov{x}^\alpha) + C \eta^{-1} \tau^\alpha \sqrt{D}  \ov{q}^\alpha \qquad \text{for all} \qquad t \in \big[ \ov{t}^\alpha - \tau^\alpha (\ov{q}^\alpha)^2, \ov{t}^\alpha \big]. \]
Since $C \eta^{-1} \sqrt{D} < \frac1{16} (\td{\rho}^\alpha)^{-1}$ for large $\alpha$, this implies that for large $\alpha$
\[ \dist_t (x_0^\alpha, \ov{x}^\alpha) < \dist_{\ov{t}^\alpha}(x_0^\alpha, \ov{x}^\alpha) + \tfrac1{16} (\td{\rho}^\alpha)^{-1}  \ov{q}^\alpha \qquad \text{for all} \qquad t \in \big[ \ov{t}^\alpha - \tau^\alpha (\ov{q}^\alpha)^2, \ov{t}^\alpha \big]. \]
This fact, however, contradicts the definition of $\tau^\alpha$.

So it follows that $\tau_\infty = \infty$.
Hence, after parabolically rescaling by $(\ov{q}^\alpha)^{-1}$, the Ricci flows with surgery $\MM^\alpha$ restricted to the time-interval $[\ov{t}^\alpha - \tau^\alpha (\ov{q}^\alpha)^2, \ov{t}^\alpha]$ and based at $\ov{x}^\alpha$ subconverge to a $\kappa$-solution $M_\infty \times (-\infty, 0]$.
Using Lemma \ref{Lem:kappasolCNA}, we finally obtain a contradiction to the assumption that the points $(\ov{x}^\alpha, \ov{t}^\alpha)$ do not satisfy the canonical neighborhood assumptions $CNA(\ov{q}^\alpha, \varepsilon, E, \eta)$.

Part (b) follows exactly the same way as in \cite[6.3]{PerelmanII}.
See also \cite[Lemma 70.2]{KLnotes} and \cite[Proposition 6.2.4]{Bamler-diploma}.
The boundary $\partial \MM (t_0)$ does not create any issues since it is far enough away from $x_0$.
\end{proof}

We now prepare for the proof of the next main result, Lemma \ref{Lem:6.4}.
We believe that we have to modify the result in \cite[6.5]{PerelmanII} as follows to make its proof work.
\begin{Lemma}[\hbox{\cite[6.5]{PerelmanII}}] \label{Lem:6.5}
For all $w > 0$ there exist $\tau_0 = \tau_0(w) > 0$ and $K_0 = K_0(w) < \infty$, such that: \\
Let $\MM$ be a Ricci flow with surgery with complete time-slices that is defined on the time-interval $[- \tau, 0]$ and let $x_0 \in \MM(0)$.
Assume that $(x_0, 0)$ survives until time $-\tau$, that for all $t \in (-\tau, 0]$ the ball $B(x_0, t, 1)$ does not intersect any surgery points or the boundary $\partial \MM (t)$, that $\sec \geq -1$ on $\bigcup_{t \in [-\tau,0]} B(x_0, t, 1) \cap P(x_0, 0, 1, - \tau)$ and that $\vol_0 B(x_0, 0, 1) \geq w$.
Then
\begin{enumerate}[label=(\alph*)]
\item $|{\Rm}| \leq K_0 \tau^{-1}$ in $P(x_0, 0, \frac14, -\tau/2)$.
\item All points in $B(x_0, - \tau, \frac14)$ survive until time $0$ and $B(x_0, - \tau, \frac14) \subset B(x_0, \linebreak[1] 0, \linebreak[1] 1)$.
\item $\vol_{-\tau} B(x_0, -\tau, \frac14) > \frac12 w (\frac14)^3$.
\end{enumerate}
\end{Lemma}
\begin{proof}
See \cite[Lemma 82.1]{KLnotes} for a proof of the first part and the proof of \cite[Corollary 45.1(b)]{KLnotes} for the third.
The second part follows from the lower bound on the sectional curvature.
\end{proof}

\begin{Lemma}[\hbox{\cite[6.6]{PerelmanII}}] \label{Lem:6.6}
For any $w > 0$ there is a $\theta_0 = \theta_0 (w) > 0$ such that:
Let $(M, g)$ be a Riemannian $3$-manifold and $B(x,1) \subset M$ a ball of volume at least $w$ that is relatively compact and does not meet the boundary of $M$.
Assume that $\sec \geq - 1$ on $B(x, 1)$.
Then there exists a ball $B(y,\theta_0) \subset B(x,1)$, such that every subball $B(z,r) \subset B(y, \theta_0)$ of any radius $r$ has volume at least $\frac1{10} r^3$.
\end{Lemma}
\begin{proof}
See \cite[Lemma 83.1]{KLnotes}.
\end{proof}

\begin{Lemma} \label{Lem:HIbound}
For any $K < \infty$ there is an $\ov{r} = \ov{r} (K) < \infty$ such that:
Let $r_0 \leq \ov{r} \sqrt{t_0}$ and $\frac12 t_0 \leq t \leq t_0$.
Assume that $(M, g)$ is a Riemannian manifold of $t^{-1}$-positive curvature and $|{\Rm}| < K r_0^{-2}$ on $M$.
Then the sectional curvature is bounded from below: $\sec \geq - \frac12 r_0^{-2}$.
\end{Lemma}

\begin{proof}
The claim is clear for $r_0 = 1$.
The rest follows from rescaling.
\end{proof}

\begin{Lemma}[\hbox{\cite[6.4]{PerelmanII}}] \label{Lem:6.4}
There is a constant $\varepsilon_0 > 0$ such that for all $r, \eta > 0$ and $E < \infty$ there are constants $\tau = \tau(\eta, E), \ov{r} = \ov{r} (\eta, E), \delta = \delta (r, \eta, E) > 0$ and $K = K( E),  C_1 = C_1 (E), Z = Z( \eta, E) < \infty$ such that: \\
Let $r_0^2 < t_0 /2$ and let $\MM$ be a Ricci flow with surgery (whose time-slices are allowed to have boundary) on the time-interval $[ t_0 - r_0^2, t_0]$ that is performed by $\delta'$-precise cutoff for some $0 < \delta' \leq \delta$ and consider a point $x_0 \in \MM(t_0)$.
Assume that the canonical neighborhood assumptions $CNA (r \sqrt{t_0}, \varepsilon_0, E, \eta)$ hold on $\MM$.
We also assume that the curvature on $\MM$ is uniformly bounded on compact time-intervals that don't contain surgery times and that all time-slices of $\MM$ are complete.

In the case in which some time-slices of $\MM$ have non-empty boundary, we assume that
\begin{enumerate}[label=(\roman*)]
\item For all $t_1 < t_2 \in [t_0 - \frac1{10} r_0^2, t_0]$ we have: if some $x \in B(x_0, t_0, r_0)$ survives until time $t_2$ and $\gamma : [t_1, t_2] \to \MM$ is a space-time curve with endpoint $\gamma(t_2) \in B(x, t_2, 3 r_0)$ that meets the boundary $\partial\MM$ somewhere, then $\LL(\gamma) > Z r_0$ (where $\LL$ is based in $t_2$).
\item For all $t \in [t_0 - \frac1{10} r_0, t_0]$ we have: if some $x \in B(x_0, t_0, r_0)$ survives until time $t$, then $B(x, t, 5 r_0 + r \sqrt{t_0})$ does not meet the boundary $\partial \MM (t)$.
\end{enumerate}
Now assume that 
\begin{enumerate}[label=(\roman*), start=3]
\item $C_1 \delta' \sqrt{t_0} \leq r_0 \leq \ov{r} \sqrt{t_0}$,
\item $\sec \geq - r_0^{-2}$ on $B(x_0, t_0, r_0)$ and
\item $\vol_{t_0} B(x_0, t_0, r_0) \geq \frac1{10} r_0^3$.
\end{enumerate}
Then the parabolic neighborhood $P(x_0, t_0, \frac14 r_0, - \tau r_0^2)$ is non-singular and we have $|{\Rm}| < K r_0^{-2}$ on $P(x_0, t_0, \frac14 r_0, - \tau r_0^2)$.
\end{Lemma}

\begin{proof}
Before we start with the main argument, we first discuss the case in which $r_0 \leq r \sqrt{t_0}$:
We first show that for a universal $K' = K' (E) < \infty$ and sufficiently small but universal $\varepsilon_0$, we can guarantee that $|{\Rm}| < \frac12 K' r_0^{-2}$ on $B(x_0, t_0, \frac14 r_0)$.
The constant $K'$ and the smallness of the constant $\varepsilon_0$ will be determined in the course of this paragraph.
Assume the assumption was wrong, i.e. there is a point $x \in B(x_0, t_0, \frac14 r_0)$ such that $Q = |{\Rm}|(x,t_0) \geq \frac12 K' r_0^{-2}$.
By the canonical neighborhood assumptions $CNA (r \sqrt{t_0}, \varepsilon_0, E, \eta)$, we know that $(x,t_0)$ is either a center of a strong $\varepsilon_0$-neck or of an $(\varepsilon_0, E)$-cap or $\MM (t_0)$ has positive $E^2$-pinched curvature (here we assumed that $K' > 2$).
The latter case cannot occur by assumption (v), for large enough $K'$, so assume that $(x, t_0)$ is a center of a strong $\varepsilon_0$-neck or an $(\varepsilon_0, E)$-cap.
In both of these cases there is a $y \in \MM(t_0)$ with $\dist_{t_0}(x, y) < E Q^{-1/2}$ such that $(y,t_0)$ is a center of an $\varepsilon_0$-neck and $E^{-2} Q < |{\Rm}|(y, t_0) < E^2 Q$.
Assuming $K' > 72 E^2$, we conclude that $y \in B(x_0, t_0, \frac13 r_0)$.
Since $\varepsilon_0$-necks are sufficiently collapsed for small enough $\varepsilon_0$, we can make the following conclusion:
For every $w > 0$ there is an $\varepsilon'_0 = \varepsilon'_0 (w) > 0$ and a $D = D(w) < \infty$ such that if $\varepsilon_0 < \varepsilon'_0$, then $\vol_{t_0} B(y, t_0, D Q^{-1/2}) < w D^3 Q^{-3/2}$.
By assumption (v) and by volume comparison, there is a universal constant $w_0 > 0$ such that $\vol_{t_0} B(y, t_0, d) \geq w_0 d^3$ for all $0 < d < \frac12 r_0$.
Assume now that $\varepsilon_0 < \varepsilon'_0 (w_0)$ and $K' > 8 D^2 (w_0)$.
Then we obtain a contradiction for $d = D(w_0) Q^{-1/2} < \frac12 r_0$.
So we indeed have $|{\Rm}| < \frac12 K' r_0^{-2}$ on $B(x_0, t_0, \frac14 r_0)$.
Next, by Lemma \ref{Lem:shortrangebounds}, assumption (iii) and the fact that at every surgery point $(x',t')$ we have (compare with (\ref{eq:surgpointhighcurv}))
\[ |{\Rm}| (x', t') > c' \delta^{\prime -2} t^{\prime -1} \geq c' \delta^{-2} t^{\prime -1} \geq c' \delta^{-2} t_0^{-1} \geq c' C_1^2 r_0^{-2}, \]
we conclude that there is a $\tau' = \tau' (\eta, E) > 0$ such that if $C_1 = C_1 (E) = c^{\prime -1/2} K^{\prime 1/2}$, then $P(x_0, t_0, \frac14 r_0, - \tau' r_0^2)$ is non-singular and $|{\Rm}| < K' r_0^{-2}$ on $P(x_0, \linebreak[1] t_0, \linebreak[1] \frac14 r_0, \linebreak[1] - \tau' r_0^2)$.

Now we return to the general case, allowing $r_0 \geq r \sqrt{t_0}$.
We will first fix some constants:
Let $\varepsilon_0, C_1$ be the constants from the last paragraph.
Without loss of generality, we may assume that $\varepsilon_0$ is smaller than the corresponding constant from Lemma \ref{Lem:6.3bc}.
Next assume that the constants $r, \eta, E$ have already been chosen.
Consider the constants $\tau_{0, \ref{Lem:6.5}}$ and $K_{0, \ref{Lem:6.5}}$ from Lemma \ref{Lem:6.5}, $\theta_{0, \ref{Lem:6.6}}$ from Lemma \ref{Lem:6.6}, $K_{\ref{Lem:6.3bc}}$, $\ov{r}_{\ref{Lem:6.3bc}}$, $Z_{\ref{Lem:6.3bc}}$ and $\delta_{\ref{Lem:6.3bc}}$ from Lemma \ref{Lem:6.3bc} and $\ov{r}_{\ref{Lem:HIbound}}$ from Lemma \ref{Lem:HIbound} and set:
\begin{alignat*}{1}
\tau &= \min \{ \tau', \tfrac12 \tau_{0, \ref{Lem:6.5}}(\tfrac1{10}), \tfrac1{100}\}  \displaybreak[1] \\
K &= \max\{ K', K_{0, \ref{Lem:6.5}} (\tfrac1{10}) \tau^{-1} \}, \displaybreak[1] \\
\theta_0 &= \min \{ \tfrac14 \theta_{0, \ref{Lem:6.6}}(\tfrac1{20}), \linebreak[1] \tfrac1{10} \} \displaybreak[1] \\
r^* &= \theta_0 \min \{ \tau^{1/2}, \linebreak[1] K^{-1/2}, \linebreak[1] \tfrac1{10} \} \displaybreak[1] \\
K^* &= (r^*)^{-2} K_{\ref{Lem:6.3bc}}(\tfrac1{10}, 2 (r^*)^{-1}, E, \eta) \displaybreak[1] \\
Z &= Z_{\ref{Lem:6.3bc}}(2 (r^*)^{-1}) \displaybreak[1] \\
\ov{r} &= \min \{ \ov{r}_{\ref{Lem:6.3bc}}(\tfrac1{10}, 2 (r^*)^{-1}, E, \eta), \ov{r}_{\ref{Lem:HIbound}} (K^*) \} \displaybreak[1] \\
\delta &= \min \{ C_1^{-1} \theta_0 r, \delta_{\ref{Lem:6.3bc}} ( \tfrac1{10}, 2 (r^*)^{-1}, r, \varepsilon_0, E, \eta), {c'}^{1/2} (K^*)^{-1/2} r \}
\end{alignat*}
Here $c'$ is again the constant from (\ref{eq:surgpointhighcurv}).

We now claim that the conclusion of the Lemma holds with this choice of the constants $\tau, \ov{r}, \delta, K, C_1, Z$ and for any $0 < \delta' \leq \delta$.
Assume not, i.e. that $P(x_0, t_0, \frac14 r_0, - \tau r_0^2)$ is singular or we don't have $|{\Rm}| < K r_0^{-2}$ on $P (x_0, \linebreak[1] t_0, \linebreak[1] \frac14 r_0, \linebreak[1] - \tau r_0^2)$.
We now carry out a point-picking process.
In the first step set $x'_0 = x_0$, $t'_0 = t_0$ and $r'_0 = r_0$.
If there are $x''_0$, $t''_0$ and $r''_0$,  such that
\begin{enumerate}
\item  $t''_0 \in [t'_0 - 2\tau (r'_0)^2, t'_0]$,
\item the point $x'_0$ survives until time $t''_0$ and for all $t \in (t''_0, t'_0]$ there are no surgery points in $B(x'_0, t, r'_0)$ and $B(x'_0, t, r'_0) \cap \partial \MM (t) = \emptyset$,
\item $\sec \geq - (r'_0)^{-2}$ on $\bigcup_{t \in [t''_0, t'_0]} B(x'_0, t, r'_0)$,
\item $x''_0 \in B(x'_0, t''_0, r'_0/4)$,
\item $r''_0 = \theta_0 r'_0 \geq C_1 \delta' \sqrt{t'_0}$,
\item $\vol_{t''_0} B(x''_0, t''_0, r''_0) \geq \frac1{10} (r''_0)^3$ and
\item we don't have $|{\Rm}| < K (r''_0)^{-2}$ on $P(x''_0, t''_0, \frac14 r''_0, - \tau (r''_0)^2)$ or the parabolic neighborhood $P(x''_0, t''_0, \frac14 r''_0, - \tau (r''_0)^2)$ is singular,
\end{enumerate}
then we replace the triple $(x'_0, t'_0, r'_0)$ by $(x''_0, t''_0, r''_0)$ and repeat.
If not, we stop the process.
Observe that here and in the rest of the proof the parabolic neighborhoods are not assumed to be non-singular unless otherwise noted (compare with Definition \ref{Def:parabnbhd}).
By the smallness of $\tau$, we have $t'_0 > t_0 - \frac1{10} r_0^2$ at every step of this process.
So by condition (5) this process always has to terminate after a finite number of steps.

Observe that by the smallness of $\tau$, $\theta_0$, distance distortion estimates and condition (3), we have in every step of this process
\[ P(x''_0, t''_0, r''_0, - \tfrac1{10} (r''_0)^2) \subset P(x'_0, t'_0, r'_0, - \tfrac1{10} (r'_0)^2). \]
So these parabolic neighborhoods for each step are nested and we have for the final triple $(x'_0, t'_0, r'_0)$
\[  P(x'_0, t'_0, r'_0, - \tfrac1{10} (r'_0)^2) \subset P(x_0, t_0, r_0, - \tfrac1{10} (r_0)^2). \]
So the triple $(x'_0, t'_0, r'_0)$ satisfies assumptions (i) and (ii) of the Lemma.
By conditions (3) and (6), also assumptions (iv) and (v) are satisfied.
Moreover, we have $(r'_0)^2 < t'_0 / 2$ and by condition (5) we have after the first step $C_1 \delta' \sqrt{t'_0} \leq r'_0 \leq \theta_0 r_0 \leq \frac1{10} r_0 \leq \frac1{10} \ov{r} \sqrt{t_0} \leq \ov{r} \sqrt{t'_0}$.
So the triple $(x'_0, t'_0, r'_0)$ also satisfies assumption (iii) of the Lemma.
However, by condition (7), the assertion of the Lemma fails for the triple $(x'_0, t'_0, r'_0)$.
Note also that the Ricci flow with surgery $\MM$ restricted to $[t'_0 - (r'_0)^2 , t'_0]$ satisfies the canonical neighborhood assumptions $CNA ( r \sqrt{t'_0}, \varepsilon_0, E, \eta)$.
Thus, after passing to this restriction and the triple $(x'_0, t'_0, r'_0)$, we may assume, without loss of generality, that $x_0 = x'_0$, $t_0 = t'_0$ and $r_0 = r'_0$ and add to our assumptions that whenever we find $x''_0$, $t''_0$ and $r''_0$ satisfying the assumptions (1)--(6) above, then the opposite of assumption (7) holds (and hence we have curvature control on $P(x''_0, t''_0, \frac14 r''_0, - \tau (r''_0)^2)$).
By the discussion at the beginning of this proof and the fact that $\tau \leq \tau'$, $K \geq K'$, we must have $r_0 > r \sqrt{t_0}$.

Now let $\ov{\tau} \leq 2\tau$ be maximal with the property that
\begin{enumerate}[label=$-$]
\item the point $x_0$ survives until time $t_0 - \ov\tau r_0^2$,
\item for all $t \in (t_0 - \ov\tau r_0^2, t_0]$, there are no surgery points in $B(x_0, t, r_0)$,
\item $\sec \geq - r_0^{-2}$ on $\bigcup_{t \in [t_0 - \ov{\tau} r_0^2, t_0]} B(x_0, t, r_0)$.
\end{enumerate}
Note that by assumption (ii), we have $B(x_0, t, r_0) \cap \partial\MM(t) = \emptyset$ for all $t \in [t_0 - \ov{\tau} r_0^2, t_0]$.
If $\ov{\tau} = 2\tau$, then the assertion of the Lemma follows using Lemma \ref{Lem:6.5}.

So assume now $\ov{\tau} < 2\tau$.
We will derive a curvature bound at times $[t_0 - \ov\tau r_0^2, t_0]$, which implies a better lower bound on the sectional curvature and hence contradicts the maximality of $\ov\tau$.
Fix for a moment $t \in [t_0 - \ov\tau r_0^2, t_0]$.
By Lemma \ref{Lem:6.5} we first conclude $\vol_t B(x_0, t, \frac14 r_0) > \frac1{20} (\frac14)^3 r_0^3$.
Hence, using Lemma \ref{Lem:6.6}, we can find a ball $B(y, t, \theta_0 r_0) \subset B(x_0, t, \frac14 r_0)$ such that $\vol_t B(y, t, \theta_0 r_0) \geq \frac1{10} \theta_0^3 r_0^3$ and such that every sub-ball of $B(y, t, \theta_0 r_0)$ has volume ratio of at least $\frac1{10}$.
Moreover, by the choice of $\delta$, we have $\theta_0 r_0 > \theta_0 r \sqrt{t_0} \geq C_1 \delta \sqrt{t_0} \geq C_1 \delta' \sqrt{t_0}$.
So the triple $(y, t, \theta_0 r_0)$ satisfies the assumptions (1)--(6) above and hence, by choice of the triple $(x_0, t_0, r_0)$, we find that the parabolic neighborhood $P(y, t, \tfrac14 \theta_0 r_0, -\tau \theta_0^2 r_0^2)$ is non-singular and
\[ |{\Rm}| < K \theta_0^{-2} r_0^{-2} \qquad \text{on} \qquad P(y, t, \tfrac14 \theta_0 r_0, - \tau \theta_0^2 r_0^2). \]

This implies that $|{\Rm}| < (r^* r_0)^{-2}$ on $P(y, t, r^* r_0, - (r^* r_0)^2)$.
Recall that by Lemma \ref{Lem:6.6} we have $\vol_t B(y, t, r^* r_0) \geq \frac1{10} (r^* r_0)^3$.
Applying Lemma \ref{Lem:6.3bc}(b) for $x_0 \leftarrow y$, $t_0 \leftarrow t$, $r_0 \leftarrow r^* r_0$, $w \leftarrow \frac1{10}$, $A \leftarrow A^* = 2 (r^*)^{-1}$, $E \leftarrow E$, $\eta \leftarrow \eta$, $r \leftarrow r$ yields
\[ |{\Rm}|(\cdot, t) < K^* r_0^{-2} \qquad \text{on} \qquad B(y, t, 2 r_0) \qquad \text{for all} \qquad t \in [t_0 - \ov\tau r_0^2, t_0]. \] 
Observe here that by the choice of $Z$ and assumptions (i), (ii) of this Lemma, the assumptions (i), (ii) of Lemma \ref{Lem:6.3bc} are satisfied.
We conclude that
\begin{equation} \label{eq:Kstarboundonball}
 |{\Rm}|(\cdot, t) < K^* r_0^{-2} \quad \text{on} \quad B(x_0, t, r_0) \quad \text{for all} \quad t \in [t_0 - \ov{\tau} r_0^2, t_0].
\end{equation}
By Lemma \ref{Lem:HIbound} and the choice of $\ov{r}$, this curvature bound implies $\sec \geq - \frac12 r_0^{-2}$ on $B(x_0, t, r_0)$ for all $t \in [t_0 - \ov{\tau} r_0^2, t_0]$.
We now argue that even for $t = t_0 - \ov{\tau} r_0^2$ there are no surgery points in $B(x_0, t, r_0)$:
By (\ref{eq:surgpointhighcurv}), the choice of $\delta$ and the fact that $r_0 > r \sqrt{t_0}$, we find that at any such surgery point $(z,t)$ we have
\[ |{\Rm}| (z,t) > c' \delta^{-2} t^{-1} \geq K^* r^{-2} t^{-1} \geq K^* r^{-2} t_0^{-1} > K^* r_0^{-2}, \]
in contradiction to (\ref{eq:Kstarboundonball}).
This also implies that the point $x_0$ survives until some time that is strictly smaller than $t_0 - \ov\tau r_0^2$ and that $B(x_0, t, r_0)$ does not contain surgery points or meet the boundary for times which are strictly smaller than $t_0 - \ov\tau r_0^2$.
This contradicts the maximality of $\ov{\tau}$ and hence finishes the proof.
\end{proof}

\begin{proof}[Proof of Proposition \ref{Prop:genPerelman}]
Let $\varepsilon_0$ be smaller than the corresponding constants from Lemmas \ref{Lem:6.3bc} and \ref{Lem:6.4}.
By Lemma \ref{Lem:6.6} we can find a ball $B(y, t_0, \theta_0(w) r_0) \subset B(x_0, t_0, r_0)$ with $\vol_{t_0} B(y, t_0, \theta_0(w) r_0) \geq \frac1{10} (\theta_0 r_0)^3$.
So we can apply Lemma \ref{Lem:6.4} with $t_0 \leftarrow t_0$, $x_0 \leftarrow y$, $r_0 \leftarrow \theta_0 r_0$, $\varepsilon_0 \leftarrow \varepsilon_0$, $E \leftarrow E$, $\eta \leftarrow \eta$, $r \leftarrow r$ and obtain that if $\delta < \delta_{\ref{Lem:6.4}} (r, \eta, E)$, if the surgeries on $\MM$ are performed by $\delta'$-cutoff for some $0 < \delta' \leq \delta$ for which
\[ C_{1, \ref{Lem:6.4}} ( E) \delta' \sqrt{t_0} \leq \theta_0 r_0, \]
if $Z > Z_{\ref{Lem:6.4}} (\eta, E)$ and if $r_0 < \ov{r}_{\ref{Lem:6.4}} (\eta, E) \sqrt{t_0}$, then the parabolic neighborhood $P(y, t_0, \frac14 \theta_0 r_0, - \tau_{\ref{Lem:6.4}}(\eta, E) \theta_0^2 r_0^2)$ is non-singular and
\[ |{\Rm}| < K_{\ref{Lem:6.4}} (E) \theta_0^{-2} r_0^{-2} \qquad \text{on} \qquad P(y, t_0, \tfrac14 \theta_0 r_0, - \tau_{\ref{Lem:6.4}} \theta_0^2 r_0^2). \]
Now choose $r^* = r^*(w, \eta, E) \in (0, \frac1{100})$ so small that $P(y, t, r^* r_0, -(r^* r_0)^2) \subset P(y, t_0, \linebreak[1] \frac14 \theta_0 r_0 , \linebreak[1] - \tau_{\ref{Lem:6.4}} \theta^2_0 r_0^2)$ for all $t \in [t_0 - (r^* r_0)^2, t_0]$ and $|{\Rm}| < (r^* r_0)^{-2}$ there.
By volume comparison and distortion estimates and the conclusion of Lemma \ref{Lem:6.6}, we find that $\vol_t B(y,t, r^* r_0) > c (r^* r_0)^3$ for all $t \in [t_0 - (r^* r_0)^2, t_0]$ and for some universal constant $c > 0$.
We can then invoke Lemma \ref{Lem:6.3bc}(b) with $t_0 \leftarrow t \in [t_0 - (r^* r_0)^2, t_0]$, $x_0 \leftarrow y$, $r_0 \leftarrow r^* r_0$, $w \leftarrow c$,  $A \leftarrow (A+2) (r^*)^{-1}$, $r \leftarrow r$, $\eta \leftarrow \eta$, $E \leftarrow E$ and obtain that if $\delta < \delta_{\ref{Lem:6.3bc}} (c, (A+2) (r^*)^{-1}, r, \varepsilon_0, E, \eta)$, $Z > Z_{\ref{Lem:6.3bc}} ( (A+2) (r^*)^{-1})$ and $r_0 < \ov{r}_{\ref{Lem:6.3bc}} ( (A+2) (r^*)^{-1}, c, E, \eta) \sqrt{t_0}$, then
\begin{equation} \label{eq:RmboundedbyKonAplus2}
 |{\Rm}| < K r_0^{-2} \qquad \text{on} \qquad B(y, t, (A+ 2)r_0) \qquad \text{for all} \qquad t \in [t_0 - (r^* r_0)^2, t_0]
\end{equation}
for $K = K_{\ref{Lem:6.3bc}} (c, (A+2) (r^*)^{-1}, E, \eta) (r^*)^{-2} < \infty$.
Fix $\delta$ for the rest of this paragraph.
We now argue that there is a constant $C_1 = C_1 (w, A, \eta, E) < \infty$ such that the following holds:
If the surgeries on $\MM$ are even performed by $\delta'$-precise cutoff for some $0 < \delta' \leq \delta$ for which $C_1 \delta' \sqrt{t_0} \leq r_0$, then there are no surgery points in $B(y, t, (A+ 2)r_0)$ for all $t \in [t_0 - (r^* r_0)^2, t_0]$.
Similarly as in  (\ref{eq:surgpointhighcurv}), at every surgery point $(z,t)$
\[ |{\Rm}| (z,t) > c' \delta^{\prime -2} t^{-1} \geq c' \delta^{\prime -2} t_0^{-1} \geq c' C_1^{2} r_0^{-2}, \]
which contradicts (\ref{eq:RmboundedbyKonAplus2}) if $C_1 > c^{\prime -1/2} K^{1/2}$.
So we can find a $\tau = \tau(w, A, \eta, E) > 0$ such that for all $t \in [t_0 - 2 \tau r_0^2, t_0]$ the points in $B(y, t_0, (A+1.5) r_0)$ survive until time $t$ and $B(y, t_0, (A+1.5) r_0) \subset B(y, t, (A+2) r_0)$.
Hence $P(y, t_0, (A+1.5) r_0, \linebreak[1] - 2\tau r_0^2)$ is non-singular and we have $|{\Rm}| < K r_0^{-2}$ on $P(x_0, t_0, (A + 0.5) r_0, - 2\tau r_0^2) \linebreak[1] \subset \linebreak[1] P(y, t_0 (A+1.5) r_0, - 2\tau r_0^2)$.
The higher derivative estimates follow from Shi's estimates on $P(x_0, t_0, (A + 0.5) r_0, - 2 \tau r_0^2)$.
Fix $C_1$ for the rest of the proof.
Note that $C_1$ can be chosen independently of $\delta$.
So we may decrease $\delta$ depending on $C_1$ and $r$ and assume that $C_1 \delta < r$.

It remains to consider the case $C_1 \delta \sqrt{t_0} > r_0$, which implies $r_0 < r \sqrt{t_0}$.
Let $Q = |{\Rm}|(x_0, t_0)$.
In the next paragraph we show that $Q r_0^2$ is bounded by a constant, which only depends on $w$, $E$ and $\eta$.

For the next two paragraphs fix $w$, $E$ and $\eta$ and assume that $Q r_0^2 > 1$.
Using the same reasoning as in the proof of Lemma \ref{Lem:6.3bc}(b) (compare with the ``bounded curvature at bounded distance''-estimate in \cite[Claim 2, 4.2]{PerelmanII}, see also the proof of \cite[Lemma 89.2]{KLnotes} or \cite[Lemma 70.2]{KLnotes} or \cite[Proposition 6.2.4]{Bamler-diploma}) we can conclude that $|{\Rm}| > K^*_1 (Q r_0^2) r_0^{-2}$ on $B(x_0, t_0, r_0)$ if $Q r_0^2 > S_0$ and $r_0 < \ov{r}^*(Q r_0^2) \sqrt{t_0}$ for some constant $S_0 < \infty$ and some functions $K^*_1, \ov{r}^* : [0, \infty) \to (0, \infty)$ with $K^*_1 (s) \to \infty$ and $s \to \infty$, which only depend on $w$, $E$ and $\eta$ (we remark that for this argument the basepoint has to be chosen at a point $x' \in B(x_0, t_0, r_0)$ with $r_0^{-2} \leq |{\Rm}| (x', t_0) \leq K^*_1(Q r_0^2) r_0^{-2}$).
So there is some $S_1 = S_1 (w, E, \eta) < \infty$ such that if $Q r_0^2 > S_1$ and $r_0 < \ov{r}^* (S_1) \sqrt{t_0}$, then all points on $B(x_0, t_0, r_0)$ are centers of strong $\varepsilon$-necks or $(\varepsilon, E)$-caps, whose cross-sectional $2$-spheres have diameter at most $C (K^*_1 (Q r_0^2))^{-1/2} r_0$, where $C < \infty$ is a universal constant.
These necks and caps can be glued together to give long tubes as described in \cite[Proposition 5.4.7]{Bamler-diploma} or \cite[sec 58]{KLnotes} and we conclude that $\vol_{t_0} B(x_0, t_0, r_0) < w^* (Q r_0^2) r_0^3$ for some function $w^* : [0, \infty) \to (0, \infty)$ with $w^*(s) \to 0$ as $s \to \infty$.
Choose now $S_2 = S_2 (w, E, \eta) < \infty$ large enough such that $w^* (S_2) < w$.
Then by assumption (v) we get $Q r_0^2 < S_2$ assuming $\ov{r} < \ov{r}^* (S_2)$.

Again, by the same reasoning as before (this time, we choose the basepoint to be $(x_0, t_0)$), we obtain the estimate $|{\Rm}| < K^*_2 r_0^{-2}$ on $B(x_0, t_0, (A+1) r_0)$ for some universal constant $K^*_2 = K^*_2 (w, A, E, \eta) < \infty$ if $r_0 < \ov{r}^{**}(w, A, E, \eta) \sqrt{t_0}$.
Let $\ov\tau \geq 0$ be maximal such that the parabolic neighborhood $P(x_0, t_0, (A+1) r_0, - \ov\tau r_0^2)$ is non-singular.
By Lemma \ref{Lem:shortrangebounds}, we conclude that there is a constant $\tau_0 > 0$ such that  $|{\Rm}| < 2 K^*_2 r_0^{-2}$ on $P (x_0, t_0, (A+1) r_0, - \min \{ \ov\tau, \tau_0 \} r_0^2)$.
If $\ov\tau \geq \tau_0$, then we can deduce curvature derivative bounds on $B(x_0, t_0, A r_0)$ by Shi's estimates.
On the other hand, if $\ov\tau < \tau_0$, then by assuming $\delta$ to be sufficiently small depending on $m$, we can use Definition \ref{Def:precisecutoff}(3) to conclude that $|{\nabla^m \Rm}| < C_k r_0^{-2-k}$ for all $k \leq m$ on initial time-slice of $P (x_0, t_0, (A+1) r_0, - \ov\tau r_0^2)$.
So by a modified version of Shi's estimates (see \cite[sec 14.4]{Cetal}), we obtain a bound on $r_0^{2+k} |{\nabla^k \Rm}|$ in $B(x_0, t_0, A r_0)$ for all $k \leq m$.

Finally, we consider the case $r_0 = \rho(x_0, t_0)$.
Applying the Proposition with $A \leftarrow 1$ yields $|{\Rm}| < K r_0^{-2}$ on $B(x_0, t_0, r_0)$ for some $K = K(w, E, \eta) < \infty$.
So by Lemma \ref{Lem:HIbound}, if we had $r_0 < \ov{r}_{\ref{Lem:HIbound}} (K) \sqrt{t_0}$, then $\sec \geq - \frac12 r_0^{-2}$ on $B(x_0, t_0, r_0)$, which would contradict the choice of $r_0$.
\end{proof}

\begin{proof}[Proof of Corollary \ref{Cor:Perelman68}]
Let $\varepsilon_0$ be the constant from Proposition \ref{Prop:genPerelman}.
Observe that by Proposition \ref{Prop:CNThm-mostgeneral} there are constants $\un\eta > 0$, $\un{E}_{\varepsilon_0} < \infty$ and decreasing, continuous, positive functions $\un{r}_{\varepsilon_0}, \un{\delta}_{\varepsilon_0} : [0, \infty) \to (0, \infty)$ such that if $\delta(t) \leq \un\delta_{\varepsilon_0} (t)$ for all $t \in [0, \infty)$, then every point $(x,t) \in \MM$ satisfies the canonical neighborhood assumptions $CNA (\un{r}_{\varepsilon_0} (t) \sqrt{t}, \varepsilon_0, \un{E}_{\varepsilon_0}, \un\eta)$.
Now consider the constant $\delta_{\ref{Prop:genPerelman}} = \delta_{\ref{Prop:genPerelman}}(r, w, A, E, \eta, m)$ from Proposition \ref{Prop:genPerelman}.
We can assume that it depends on its parameters $r$, $w$ and $A$ in a monotone way, i.e. $\delta_{\ref{Prop:genPerelman}}(r', w', A', E, \eta, m') \leq \delta_{\ref{Prop:genPerelman}}(r, w, A, E, \eta, m)$ if $r' \leq r$, $w' \leq w$, $A' \geq A$ and $m' \geq m$.
Assume now that for all $t > 0$
\begin{equation} \label{eq:deltatlessdeltas}
 \delta(t) < \min \big\{ \delta_{\ref{Prop:genPerelman}} (\tfrac12 \un{r}_{\varepsilon_0} (2 t), t^{-1}, t, \un{E}_{\varepsilon_0}, \un\eta, [t]), \; \un\delta_{\varepsilon_0} (t), \; t^{-1} \big\}.
\end{equation}
Let $w, A, m$ be given.
Choose $T = T(w, A, m) < \infty$ such that $2T^{-1} < w$, $\frac12 T > A$ and $\frac12 T > m$.

Consider the point $x$, the time $t > T$ and the scale $r$ from part (a) of the corollary.
We may assume $\ov{r} < \frac12$ such that $r^2 < t/4$.
The flow $\MM$ satisfies the canonical neighborhood assumptions $CNA (\frac12 \un{r}_{\varepsilon_0} (t) \sqrt{t}, \varepsilon_0, \un{E}_{\varepsilon_0}, \un\eta)$ on $[t - r^2, t]$.
Moreover, by (\ref{eq:deltatlessdeltas}), the surgeries on $[t - r^2, t]$ are performed by $\delta_{\ref{Prop:genPerelman}} (\frac12 \un{r}_{\varepsilon_0} (t), \linebreak[1] 2 t^{-1}, \linebreak[1] \frac12 t, \linebreak[1] \un{E}_{\varepsilon_0}, \linebreak[1] \un\eta, [\frac12 t])$-precise cutoff.
By the choice of $T$ and the monotonicity of $\delta_{\ref{Prop:genPerelman}}$, this implies that the surgeries on $[t - r^2, t]$ are performed by $\delta_{\ref{Prop:genPerelman}} (\frac12 \un{r}_{\varepsilon_0} (t), \linebreak[1] w, \linebreak[1] A, \linebreak[1] \un{E}_{\varepsilon_0}, \linebreak[1] \un\eta, m)$-precise cutoff.
So we can apply Proposition \ref{Prop:genPerelman} with $x_0 \leftarrow x$, $t_0 \leftarrow t$, $r_0 \leftarrow r$, $r \leftarrow \frac12 \un{r}_{\varepsilon_0} (t)$, $w \leftarrow w$, $A \leftarrow A$, $E \leftarrow \un{E}_{\varepsilon_0}$, $\eta \leftarrow \un\eta$, $m \leftarrow m$ to conclude that if $r \leq \ov{r}_{\ref{Prop:genPerelman}} (w, A, \un{E}_{\varepsilon_0}, \un\eta) \sqrt{t}$, then $|{\nabla^k \Rm}| < K_{m, \ref{Prop:genPerelman}} (w, A, \un{E}_{\varepsilon_0}, \un\eta) r^{-2-k}$ on $B(x, t, A r)$ for all $k \leq m$.
If the surgeries on $[t-r^2, t]$ are performed by $C_{1, \ref{Prop:genPerelman}}^{-1} (w, A, \un{E}_{\varepsilon_0}, \un\eta) r t^{-1/2}$-precise cutoff, then the second part of Proposition \ref{Prop:genPerelman} gives us that $P(x, t, A r, - \tau_{\ref{Prop:genPerelman}}(w, A, \un{E}_{\varepsilon_0}, \un\eta))$ is non-singular and $|{\nabla^k \Rm}| < K_{m, \ref{Prop:genPerelman}} \linebreak[1] (w, \linebreak[1] A, \linebreak[1] \un{E}_{\varepsilon_0}, \linebreak[1] \un\eta)  r^{-2-k}$ there for all $k \leq m$.
This establishes assertion (a).

For part (b) we argue as follows:
If $\rho(x,t) \leq \ov{r} \sqrt{t}$, then by our discussion in the last paragraph for $r = \rho(x,t)$ and Proposition \ref{Prop:genPerelman}, we obtain $r = \rho(x,t) > \widehat{r}_{\ref{Prop:genPerelman}} (w, \un{E}_{\varepsilon_0}, \un\eta) \sqrt{t}$.
So, in general, we have $\rho(x,t) > \ov\rho \sqrt{t}$ for $\ov\rho = \ov\rho (w) = \min \{ \ov{r}, \widehat{r}_{\ref{Prop:genPerelman}} \}$  and we can apply assertion (a) with $r \leftarrow \ov\rho \sqrt{t}$ and $A \leftarrow A \ov\rho^{-1}$ to deduce a curvature bound on $P(x,t, A \sqrt{t}, - \tau \ov{\rho}^2 t)$.
For this application it is important that all surgeries on $[t - \ov{\rho}^2 t, t]$ are performed by $c_1 \ov{\rho}$-precise cutoff.
This is certainly the case for sufficiently large $T= T(w, A)$, because for large $t$ we have $\delta (t) \leq t^{-1} < c_1 \ov{\rho}$.
\end{proof}

\subsection{The thick-thin decomposition} \label{subsec:thickthin}
We now describe how, in the long-time picture, Ricci flows with surgery decompose the manifold into a thick and a thin part.
In this process, the thick part approaches a hyperbolic metric while the thin part collapses at local scales.
Compare this Proposition with \cite[7.3]{PerelmanII} and \cite[Proposition 90.1]{KLnotes}.

\begin{Proposition} \label{Prop:thickthindec}
There is a function $\delta : [0, \infty) \to (0, \infty)$ such that, given a Ricci flow with surgery $\MM$ with normalized initial conditions that is performed by $\delta(t)$-precise cutoff and defined on the interval $[0,\infty)$, we can find a constant $T_0 < \infty$, a function $w : [T_0, \infty) \to (0, \infty)$ with $w(t) \to 0$ as $t \to \infty$ and a collection of orientable, complete, finite volume hyperbolic (i.e. of constant sectional curvature $-1$) manifolds $(H'_1, g_{\hyp,1}), \ldots, (H'_k, g_{\hyp, k})$ such that: \\
There are finitely many embedded $2$-tori $T_{1,t}, \ldots, T_{m,t} \subset \MM(t)$ for $t \in [T_0, \infty)$ that move by isotopies and don't hit any surgery points and that separate $\MM(t)$ into two (possibly empty) closed subsets $\MM_{\thick}(t), \linebreak[1] \MM_{\thin}(t) \subset \MM(t)$ such that
\begin{enumerate}[label=(\textit{\alph*})]
\item $\MM_{\thick}(t)$ does not contain surgery points for any $t \in [T_0, \infty)$.
\item The $T_{j,t}$ are incompressible in $\MM(t)$ and $\diam_t T_{j,t} < w(t) \sqrt{t}$.
\item The topology of $\MM_{\thick}(t)$ stays constant in $t$ and $\MM_{\thick}(t)$ is a disjoint union of components $H_{1,t}, \ldots, H_{k,t} \subset \MM_{\thick}(t)$ such that the interior of each $H_{i,t}$ is diffeomorphic to $H'_i$.
\item We can find an embedded cross-sectional torus $T'_{j,t}$ inside each cusp of each $H'_i$, at a distance of at least $w^{-1}(t)$ from a fixed base point, which moves by isotopy and speed at most $w(t) t^{-1/2}$ such that the following holds:
Chop off the ends of the $H'_i$ along the $T'_{j,t}$ and call the remaining open manifolds $H''_{i,t}$.
Then there are smooth families of diffeomorphisms $\Psi_{i,t} : H''_{i,t} \to H_i$, which become closer and closer to being isometries, i.e.
\[ \big\Vert \tfrac1{4t} \Psi_{i,t}^* g(t) - g_{\hyp,i} \big\Vert_{C^{[w^{-1}(t)]}(H''_{i,t})} < w(t) \]
and which move slower and slower in time, i.e.
\[  \sup_{H''_{i,t}} t^{1/2} | \partial_t \Psi_{i,t} |  < w(t) \]
for all $t \geq T_0$ and $i = 1, \ldots, k$.
Moreover, the sectional curvatures on a $w^{-1}(t) \sqrt{t}$-tubular neighborhood of $\MM_{\thick}(t)$ lie in the interval $(\frac1t (-\frac14 - w(t)), \frac1t (-\frac14 + w(t)))$ for all $t \geq T_0$.
And for every $2$-torus $T_{j,t}$, $j = 1, \ldots, m$ and all $t \geq T_0$, there are neighborhoods $P_{j,t} \subset \MM_{\thin} (t)$ with $T_{j,t} \subset P_{j,t}$ that have the following properties: $P_{j,t} \approx T^2 \times I$, $P_{j,t}$ has a $T^2$-fibration over an interval whose fibers have time-$t$ diameter $< w(t) \sqrt{t}$, one of these fibers is $T_{j,t}$ and the boundary components of $P_{j,t}$ have time-$t$ distance of at least $w^{-1} (t) \sqrt{t}$ from $T_{j,t}$.
\item A large neighborhood of the part $\MM_{\thin}(t)$ is better and better collapsed, i.e. for every $t \geq T_0$ and $x \in \MM(t)$ with 
\[ \dist_t (x, \MM_{\thin}(t)) < w^{-1}(t) \sqrt{t} \]
we have
\[ \vol_t B ( x,t, \rho_{\sqrt{t}} (x,t) ) < w(t) \rho_{\sqrt{t}}^3 (x,t). \]
\end{enumerate}
\end{Proposition}

\section{Long-time estimates under the presence of collapse} \label{sec:maintools}
In the following we derive more specialized estimates using the methods and results presented in the previous section.
Those statements will be used in \cite{Bamler-LT-main}.

\subsection{The goodness property} \label{subsec:goodness}
The following notion will become important for us.

\begin{Definition}[goodness] \label{Def:goodness}
Let $(M,g)$ be a Riemannian 3-manifold (possibly with boundary), $r_0 > 0$ and consider the function $\rho_{r_0} : M \to (0, \infty)$ from Definition \ref{Def:rhoscale}.
Let $w > 0$ be a constant and $x \in M$ be a point.
\begin{enumerate}[label=(\arabic*)]
\item Let $\td{x}$ be a lift of $x$ in the universal cover $\td{M}$ of $M$.
Then $x \in M$ is called \emph{$w$-good at scale $r_0$} if $\vol B(\td{x}, \rho_{r_0}(x)) > w \rho_{r_0}^3(x)$.
Here $B(\td{x}, \rho_{r_0}(x))$ denotes the $\rho_{r_0}(x)$-ball in the universal cover $\td{M}$ of $M$.
\item Let $U \subset M$ be an open subset and assume that $x \in U$.
Assume now that $\td{x}$ is a lift of $x$ in the universal cover $\td{U}$ of $U$.
Then $x$ is called \emph{$w$-good at scale $r_0$ relative to $U$} if either $B(x, \rho_{r_0}(x)) \not\subset U$ or $\vol B(\td{x}, \rho_{r_0}(x)) > w \rho_{r_0}^3(x)$, where now $B(\td{x}, \rho_{r_0}(x))$ denotes the $\rho_{r_0}(x)$-ball in the universal cover $\td{U}$ of $U$.
\item The point $x$ is called \emph{locally $w$-good at scale $r_0$} if it is $w$-good at scale $r_0$ relative to $B(x, \rho_{r_0}(x))$.
\end{enumerate}
\end{Definition}

Observe that the choice of the lift $\td{x}$ of $x$ is not essential.
We remark that the property ``$w$-good'' implies the properties ``$w$-good relative to a subset $U$'' and ``locally $w$-good''.
The opposite implication, however, is generally false: Consider for example a smoothly embedded solid torus $S \subset M$, $S \approx S^1 \times D^2$ and a collar neighborhood $U$ of $\partial S$ in $S$, i.e. $U \subset S$, $U \approx T^2 \times (0,100)$ and $\partial S \subset \partial U$, such that the geometry on $U$ is close to a product geometry $T^2 \times (0,100)$ in which the $T^2$-factor is very small.
Then for some $w > 0$ all points of $U$ are $w$-good relative to $U$ as well as locally $w$-good, but none of the points of $U$ are $w$-good (see \cite[Figure \ref{fig:expandingsolidtorus}]{Bamler-LT-Introduction} for an illustration).

We also remark that by volume comparison there is a universal constant $\td{c} > 0$ such that if $x \in M$ is $w$-good at scale $r_0 > 0$ for some $w > 0$, then $x$ is also $\td{c} w$-good at any scale $r'_0 \leq r_0$.

\subsection{Universal covers of Ricci flows with surgery}
In the following subsections we will need to carry out Perelman's methods in the universal covering flow $\td\MM$ of a given Ricci flow with surgery $\MM$.
In the case in which $\MM$ is non-singular, $\td\MM$ is just the universal cover of the underlying manifold equipped with the pullback of the time-dependent metric.
In the general case, the existence of $\td\MM$ is established by the following lemma.

\begin{Lemma} \label{Lem:tdMM}
Let $\MM$ be a Ricci flow with surgery on a time-interval $I \subset [0, \infty)$ that is performed by precise cutoff.
Then there is a Ricci flow with surgery $\td\MM$ (called the \emph{universal covering flow}) that is performed by precise cutoff and a family of Riemannian coverings $\pi_t : \td\MM(t) \to \MM(t)$ that are locally constant in time away from surgery points such that the components of all time-slices $\td\MM(t)$ are simply connected (i.e. $\td\MM(t)$ is the disjoint union of components that are isometric to the universal cover of $\MM(t)$).

Moreover, if $\MM$ is performed by $\delta(t)$-precise cutoff for some $\delta : I \to (0, \infty)$, then so is $\td\MM$.
If all time-slices of $\MM$ are complete, then the same is true for $\td\MM$.
If the curvature on $\MM$ is bounded on compact time-intervals that don't contain surgery times, then this property also holds on $\td\MM$.
\end{Lemma}

\begin{proof}
Recall that $\MM = ((T^i), (M^i \times I^i, g_t^i), (\Omega^i), (U^i_{\pm}), (\Phi^i))$ where each $g^i_t$ is a Ricci flow on the $3$-manifold $M^i$ defined for times $I^i$.
We can lift each of these flows to the universal cover $\widetilde{M}_0^i$ of $M^i$ via the natural projections $\pi^i_0 : \td{M}^i_0 \to M^i$ and obtain families of metrics $\widetilde{g}^i_{0, t}$, which still satisfy the Ricci flow equation.
If $M^i$ is disconnected, then we define $\td{M}_0^i$ to be the disjoint union of the universal covers of the components of $M^i$.

We will now assemble the flows $(\widetilde{M}_0^i \times I^i, \widetilde{g}_{0, t}^i)$ to a Ricci flow with surgery $\td\MM$.
Each time-slice $\td\MM (t)$ of the resulting flow will be composed of a (possibly infinite) number of copies of components of $(\td{M}_0^i, \td{g}_{0, t}^i)$ if $t \in I^i$.
If there are no surgery times in $I$, i.e. $I = I^1$, then we set $\MM = (\cdot, (\td{M}_0^1, \td{g}_{0, t}^1), \cdot, \cdot, \cdot)$ and we are done.
Assume now that there are surgery times.
For any $i$ let $\MM^i$ be the restriction of $\MM$ to the time-interval $I \cap (-\infty, T^i)$ and if $T^{i-1}$ is the last surgery time, set $\MM^i = \MM$.
By induction, we can assume that $\td\MM^i$ already exists and we only need to prove that we can extend this flow to a Ricci flow with surgery $\td\MM^{i+1}$, which is the universal covering flow of $\MM^{i+1}$.
In order to do this, it suffices to construct the objects $(\td{M}^{i+1} \times I^{i+1}, \td{g}^{i+1}_t)$, $\td\Omega^i$, $\td{U}^i_\pm$, $\td{\Phi}^i$ and the projection $\pi^{i+1} : \td{M}^{i+1} \to M^{i+1}$.

Fix $i$ and consider $(\td{M}^i \times I^i, \td{g}^i_t)$ from $\td{\MM}^i$ and the projection $\pi^i : \td{M}^i \to M^i$ corresponding to $\pi_t$ for $t \in I^i$.
Denote by $\widetilde{\Omega}^i \subset \td{M}^i$ the preimage of $\Omega^i$ and by $\td{U}^i_- \subset \td{\Omega}^i$ the preimage of $U^i_-$ under $\pi^i$ and let $\td{U}^i_{0, +} \subset \td{M}^{i+1}_0$ be the preimage of $U^i_+$ under $\pi^{i+1}_0$.
Recall that by Definition \ref{Def:precisecutoff}(6) the subset $U^i_- \subset M^i$ is bounded by pairwise disjoint, embedded $2$-spheres.
So for every point $p \in U^i_-$, the natural map $\pi_1(U^i_-, p) \to \pi_1(M^i,p)$ is an injection.
Consider now the set $\widetilde{U}^i_{0, +} \subset \widetilde{M}^{i+1}_0$.
Recall that by Definition \ref{Def:precisecutoff}(2) the complement of $U^i_+$ in $M^{i+1}$ is a collection of pairwise disjoint, embedded $3$-disks.
So the complement of $\td{U}^i_{0, +}$ in $\td{M}^{i+1}_0$ is still a collection of pairwise disjoint, embedded $3$-disks and hence the preimage of each component of $U^i_+$ under $\pi^i_0$ is simply connected.
The map $(\Phi^i)^{-1} \circ \pi^{i+1}_0 : \widetilde{U}^i_{0, +}  \to U^i_-$ is a covering map.
Consider a component $\CC_+$ of $\td{U}^i_{0, +}$ and a component $\CC_-$ of $\td{U}^i_-$ with $\pi^{i+1}_0 (\CC_+) = \Phi^i (\pi^i (\CC_-))$.
Since $\CC_+$ is simply-connected, we find a lift $\phi^i_{\CC_+, \CC_-} : \CC_+ \to \td{U}^i_-$ of $(\Phi^i)^{-1} \circ \pi^{i+1}_0 |_{\CC_+} : \CC_+  \to U^i_- \subset M^i$ such that $\phi^i_{\CC_+, \CC_-} (\CC_+) = \CC_-$ and $\pi^i \circ \phi^i_{\CC_+, \CC_-} = ( \Phi^i)^{-1} \circ \pi^{i+1}_0 |_{\CC_+}$.
Since $U^i_- \to M^i$ is $\pi_1$-injective, the map $\phi^i_{\CC_+, \CC_-}$ must be injective.

We can now construct $\td{M}^{i+1}$.
For every component $\CC_-$ of $\td{U}^i_-$ there is a (unique) component $\CC_+ = \CC_+(\CC_-)$ of $\td{U}^i_{0, +}$ such that $\pi^{i+1}_0 (\CC_+) = \Phi^i (\pi^i (\CC_-))$.
Let $\td{M}^{i+1}_{0}(\CC_+)$ be the component of $\td{M}^{i+1}_0$ that contains $\CC_+$ (observe that $\CC_+$ is the only component of $\td{U}^i_{0, +}$ in $\td{M}^{i+1}_{0}(\CC_+)$).
Now define $\td{M}^{i+1}$ to be the disjoint union of all components $\td{M}^{i+1}_0 (\CC_+(\CC_-))$ where $\CC_-$ runs through all components of $\td{U}^i_-$.
The set $\td{U}^i_+$ is the disjoint union of all the $\CC_+(\CC_-)$ and the diffeomorphism $\td{\Phi}^i$ is defined to be the inverse of $\phi^i_{\CC_+(\CC_-), \CC_+}$ on each $\CC_-$.
We also define the projection $\pi^{i+1} : \td{M}^{i+1} \to M^{i+1}$, corresponding to $\pi_t$ for $t \in I^{i+1}$, to be equal to $\pi^{i+1}_0 : \td{M}^{i+1}_0 \to M^{i+1}$ restricted to $\td{M}^{i+1}_0 ( \CC_+ (\CC_-))$ for every component $\CC_-$ of $\td{U}^i_-$.
Finally, we set $\td{g}^{i+1}_t = (\pi^{i+1} )^* g_t^{i+1}$ for all $t \in I^{i+1}$.
This finishes the proof.
\end{proof}

\subsection{Quotients of necks}
Before we discuss the main results of this section, we need to establish the following Lemma, which asserts that sufficiently precise $\varepsilon$-necks cannot have arbitrarily small quotients.

\begin{Lemma} \label{Lem:neckhasfewquotients}
There are constants $\td\varepsilon_0, \td{w}_0 > 0$ such that the following is true:
Let $(M, g)$ be a (possibly open) Riemannian manifold, $\varepsilon \leq \td\varepsilon_0$, assume that $x_0 \in M$ is a center of an $\varepsilon$-neck and that $0 < r < |{\Rm}|^{-1/2} (x)$.
Consider a local isometry $\pi : (M, g) \to (M', g')$ (i.e. $\pi^* g' = g$) such that $\pi(M) \subset M'$ is not compact (i.e. $\pi(M)$ is not a closed manifold) and let $x'_0 = \pi (x_0) \in M'$.
Then $\vol_{g'} B(x'_0, r) > \td{w}_0 r^3$.
\end{Lemma}
\begin{proof}
We may assume without loss of generality that the scale $\lambda$ in Definition \ref{Def:epsneck} is equal to $1$ (and hence $r < 1.1$ for small $\varepsilon$), that $M$ is an $\varepsilon$-neck and that $\pi$ is surjective.
So, we can identify $M = S^2 \times (- \frac1{\varepsilon}, \frac1{\varepsilon} )$ such that $x_0 \in S^2 \times \{ 0 \}$ and assume that $\Vert g - g_{S^2 \times \IR} \Vert_{C^{[\varepsilon^{-1}]}} < \varepsilon$.
If $\varepsilon$ is small enough, there is a smooth unit vector field $X$ on $M$, pointing in the direction of the eigenspace of $\Ric$ associated to the smallest eigenvalue, which is unique up to sign.
For any $y_1, y_2 \in M$ with $\pi (y_1) = \pi(y_2)$, we have $d\pi (X_{y_1}) = \pm d\pi (X_{y_2})$.
So by possibly passing to a $2$-fold cover of $M'$, we can assume that $d\pi ( X )= X'$ for some smooth vector field $X'$ on $M'$ (passing to a $2$-fold cover may change the constant $\td{w}_0$ by a factor of $2$).
Moreover, by possibly passing to another $2$-fold cover, we can assume that $M'$ is orientable.
Let $\Sigma \subset M$ be the embedded $2$-sphere corresponding to $S^2 \times \{ 0 \}$.
If $\varepsilon$ is small enough, the trajectories of $X$ cross $\Sigma$ exactly once and transversely.
Finally, let $U_0 \subset M$ be the open set corresponding to $S^2 \times (-30, 30)$ and assume that $\varepsilon^{-1} > 100$.

We will first show by contradiction that $\pi$ restricted to the ball $B(x_0, 1)$ is injective.
So assume that there are two distinct points $y_1, y_2 \in B(x_0, 1)$ with $\pi(y_1) = \pi(y_2)$.
Consider a geodesic segment $\gamma$ between $y_1$ and $x_0$ and lift its projection $\pi \circ \gamma$ starting from $y_2$.
This produces a point $x_1 \in M$ with $\pi(x_0) = \pi(x_1)$ and $\dist(x_0, x_1) < 2$.

We now show that we can construct an isometric local deck transformation $\varphi : U_0 \to U_1 \subset M$, that is a smooth map that satisfies $\pi \circ \varphi = \pi$ and $\varphi^* g = g$, with $\varphi(x_0) = x_1$.
Fix some point $p \in U_0$.
We can find a piecewise smooth curve $\gamma : [0,1] \to M$ of length less than $40$ such that $\gamma(0) = x_0$ and $\gamma (1) = p$.
Moreover, any two such curves are homotopic relative endpoints to one another, through curves of length less than $50$.
Now consider the projection $\pi \circ \gamma : [0,1] \to M'$.
Observe that $\pi (\gamma (0)) = \pi (x_0) = \pi(x_1)$.
So since $B(x_1, 40) \subset B(x_0, 42)$ is relatively compact in $M$, we can lift $\pi \circ \gamma$ to a curve $\gamma^* : [0,1] \to M$ with $\gamma^* (0) = x_1$.
Then $\pi \circ \gamma = \pi \circ \gamma^*$ and, in particular, $\pi (\gamma^* (1) ) = \pi ( \gamma (1) ) = \pi (p)$.

We now argue that $\gamma^* (1)$ only depends on $p$ and not on the choice of $\gamma$.
Consider a homotopy $H : [0,1] \times [0,1] \to M$ between two curves $\gamma_0 = H(\cdot, 0), \gamma_1 = H(\cdot , 1) : [0,1] \to M$.
Assume that for all $s \in [0,1]$, the curve $H(\cdot, s)$ has length less than $50$ and $H(0,s) = x_0$ and $H(1,s) = p$.
Then $\pi \circ H : [0,1] \times [0,1] \to M'$ can be lifted to a homotopy $H^* : [0,1] \times [0,1] \to M$ such that $\pi \circ H = \pi \circ H^*$ and such that $H^* (0,s) = x_1$ for all $s \in [0,1]$.
Note that in order to carry out this lift, it is important that $B(x_1, 50) \subset B(x_0, 52)$ is relatively compact in $M$.
The curves $\gamma^*_0 := H^* (\cdot, 0)$ and $\gamma^*_1 := H^* (\cdot, 1)$ are lifts of $\pi \circ \gamma_0$ and $\pi \circ \gamma_1$ with $\gamma^*_0 (0) = \gamma^*_1 (0) = x_1$.
Since $\pi \circ H^* (1,s) = \pi \circ H (1,s) = \pi (p)$ is constant in $s$, it follows that $H^* (1,s)$ is constant in $s$ and hence $\gamma^*_0 (1) = H(1,0) = H(1,1) = \gamma^*_1 (1)$.
This shows that point $\gamma^* (1)$ in the previous paragraph does not depend on the choice of $\gamma^*$.
So we can define $\varphi (p) := \gamma^* (1)$.
Letting $p$ vary over $U_0$, defines a map $\varphi : U_0 \to M$ with $\pi \circ \varphi = \pi$.

In order to show that $\varphi$ is smooth and isometric, it remains to show that $\varphi$ is continuous.
Fix some point $p \in U_0$ and let $\gamma : [0,1] \to M$ be a curve between $x_0$ and $p$ of length $l < 40$.
Let moreover, $\gamma^* : [0,1] \to M$ be a lift of $\gamma$ at $x_1$.
Choose $0 < d < 40 - l$.
Then every point $q \in B(p, d)$ can be reached from $x_0$, by first following $\gamma$ and then following a geodesic between $p$ and $q$.
The length of the resulting curve $\gamma' : [0,1] \to M$ is less than $40$.
Let $\gamma^{\prime *} : [0,1] \to M$ be a lift of $\pi \circ \gamma'$ at $x_1$.
Note that $\gamma^{\prime *}$ arises from concatenating $\gamma^*$ with a curve of length less than $d$.
So $\varphi (q) = \gamma^{\prime *} (1)$ has distance less than $d$ from $\varphi (p) = \gamma^* (1)$.
This finishes the proof of the continuity of $\varphi$.

Since $\pi \circ \varphi = \pi$, the map $\varphi$ preserves orientation and the vector field $X$.
Now for any $x \in \Sigma$ define $\varphi'(x)$ to be the unique intersection point of the $X$-trajectory passing through $\varphi(x)$ with $\Sigma$.
Then $\varphi' : \Sigma \to \Sigma$ is bijective continuous and orientation preserving.
Hence it has a fixed point $z_0 \in \Sigma$.

Note that for sufficiently small $\varepsilon$ we have
\begin{multline*}
 \dist (z_0, \varphi (z_0)) \leq \dist (z_0, x_0) + \dist (x_0, \varphi(x_0)) + \dist (\varphi(x_0), \varphi (z_0))  \\
< 2 \dist (z_0, x_0) + \dist (x_0, x_1) < 7 + 2 < 10.
\end{multline*}
Let now $z_k = \varphi^{(k)}(z_0) \in U_1$ as long as this is defined.
Those points all lie on the trajectory through $z_0$ and have consecutive distance $\dist (z_0, \varphi (z_0)) < 10$.
Hence, there is a point $z_{k_0} \in U_1$ whose distance to $z_0$ is contained in the interval $[10,20]$.
This implies that $\Sigma' = \varphi^{(k_0)}(\Sigma)$ is disjoint from $\Sigma$.

For every $x \in \Sigma$ let $\sigma_x : (a_x, b_x) \to M$ be the trajectory of the vector field $X$ through $x$.
That is
\[ \sigma'_x (s) = X (\sigma (s)) \quad \text{for all} \quad s \in (a_x, b_x) \qquad \text{and} \qquad \sigma_x (0) = x. \]
Here $a_x < b_x$ are chosen such that $(a_x, b_x)$ is the maximal domain of $\sigma_x$.
Since every such trajectory intersects $\Sigma'$ exactly once, we can find a function $S : \Sigma \to \IR$ such that for all $x \in \Sigma$ we have $S(x) \in (a_x, b_x)$ and $\sigma_x (S(x)) \in \Sigma'$.
Since $\Sigma$ and $\Sigma'$ are disjoint, $S$ vanishes nowhere.
By transversality, we find that $S$ is smooth.
Since $\Sigma$ and $\Sigma'$ are disjoint, the function $S$ is never zero.
It follows that the map
\[ S^2 \times [0,1] \longrightarrow M, \qquad  (x,s) \longmapsto \sigma_x (s \cdot S(x)) \]
is a smooth embedding.
Denote its image by $P \subset M$.
Then $P$ is compact and $\partial P = \Sigma \cup \Sigma'$.
Moreover, the vector field $X$ points inwards on $\Sigma$ and outwards on $\Sigma'$ or vice versa.
Since $\pi (\Sigma) = \pi (\Sigma')$, $\pi \circ \varphi = \pi$ on $\Sigma$ and $d\pi (X |_{\Sigma} ) = d\pi (X |_{\Sigma'})$ it follows that $\pi (P) = \ov\pi (\ov{P})$, where $\ov{P}$ is the (closed) manifold that arises from $P$ by identifying each $x \in \Sigma$ with $\varphi (x) \in \Sigma'$ and $\ov{\pi} : \ov{P} \to M'$ is an open map.
It follows that $\pi ( P ) \subset M'$ is equal to a closed component of $M'$.
So $\pi (M) = \pi (P)$, which contradicts our assumptions.

So $\pi |_{B(x_0, 1)}$ is indeed injective.
This implies that for all $r \leq 1$ we have $\vol_{g'} B(x'_0, r) = \vol_g B(x_0, r) > c r^3$ for some universal $c > 0$.
This finishes the proof.
\end{proof}

\subsection{Bounded curvature around good points} \label{subsec:boundedcurvaroundgoodpts}
We start by presenting a simple generalization of Corollary \ref{Cor:Perelman68} and consequence of Proposition \ref{Prop:genPerelman}, which exhibits the flavor of the subsequent results.
We point out that the following Proposition is also a consequence of the far more general Proposition \ref{Prop:curvcontrolincompressiblecollapse} below.

\begin{Proposition} \label{Prop:curvcontrolgood}
There is a continuous positive function $\delta : [0, \infty) \to (0, \infty)$ such that for any $w, \theta > 0$ there are $\tau = \tau(w), \ov{\rho} = \ov{\rho} (w) > 0$ and $K = K(w), T = T(w, \theta) < \infty$ such that: \\
Let $\MM$ be a Ricci flow with surgery on the time-interval $[0, \infty)$ with normalized initial conditions that is performed by $\delta(t)$-precise cutoff.
Let $t_0 > T$ be a time, $x_0 \in \MM (t_0)$ a point and $r _0 > 0$ and assume that
\begin{enumerate}[label=(\roman*)]
\item $\theta \sqrt{t_0} \leq r_0 \leq \sqrt{t_0}$,
\item $x_0$ is $w$-good at scale $r_0$ and time $t_0$.
\end{enumerate}
Then we have $\rho(x_0, t_0) > r_1: = \min \{ \ov{\rho} \sqrt{t_0}, r_0 \}$ and the parabolic neighborhood $P(x_0, t_0,r_1 , - \tau r_1^2)$ is non-singular and $|{\Rm}| < K r_0^{-2}$ on $P(x_0, t_0,r_1 , - \tau r_1^2)$.
\end{Proposition}

\begin{proof}
The proof is very similar to that of Corollary \ref{Cor:Perelman68}.
Let $\varepsilon_0$ be the constant from Proposition \ref{Prop:genPerelman} and $\un{E}_{\varepsilon_0}$, $\un\eta$ and $\un{r}_{\varepsilon_0}$, $\un{\delta}_{\varepsilon_0} : [0, \infty) \to (0, \infty)$ the constants and decreasing functions from Proposition \ref{Prop:CNThm-mostgeneral}.
Consider the constant $\delta_{\ref{Prop:genPerelman}} = \delta_{\ref{Prop:genPerelman}} (r, w, A, E, \eta, m)$ from Proposition \ref{Prop:genPerelman} and assume again that it satisfies the same monotonicity property as explained in the beginning of the proof of  Corollary \ref{Cor:Perelman68}.
We now choose $\delta : [0, \infty) \to (0, \infty)$ such that for all $t > 0$
\[ \delta (t) < \min \big\{ \delta_{\ref{Prop:genPerelman}} (\tfrac12 \un{r}_{\varepsilon_0} (2t), t^{-1},  1, \un{E}_{\varepsilon_0}, \un\eta, 0 ), \un{\delta}_{\varepsilon_0} (t), t^{-1} \big\}. \]

Consider now the Ricci flow with surgery $\MM$.
Since $\delta(t) < \un{\delta}_{\varepsilon_0} (t)$, we get by Proposition \ref{Prop:CNThm-mostgeneral} that every point $(x,t) \in \MM$ with $t \in [\frac12 t_0, t_0]$ satisfies the canonical neighborhood assumptions $CNA (\frac12 \un{r}_{\varepsilon_0} (t_0) \sqrt{t_0}, \varepsilon_0, \un{E}_{\varepsilon}, \un\eta)$.
This implies that also every point $(x,t) \in \td\MM$  in the universal covering flow (see Lemma \ref{Lem:tdMM}) with $t \in [\frac12 t_0, t_0]$ satisfies the same canonical neighborhood assumptions $CNA (\frac12 \un{r}_{\varepsilon_0} (t_0) \sqrt{t_0}, \varepsilon_0, \un{E}_{\varepsilon_0}, \un\eta)$.

Next, we choose $T = T(w, \theta) < \infty$ such that
\[ T > \max \{ 2 \td{c}^{-1} w^{-1}, 2 C_{1, \ref{Prop:genPerelman}} \theta^{-1} \}, \]
where $C_{1, \ref{Prop:genPerelman}} = C_{1, \ref{Prop:genPerelman}} (\td{c} w, 1, \un{E}_{\varepsilon_0}, \un\eta )$ is the constant from Proposition \ref{Prop:genPerelman} and $\td{c}$ was defined at the end of subsection \ref{subsec:goodness}.

Assume that $t_0 > T$.
By the choice of $\delta(t)$ and $T$, the surgeries on $\td\MM$ restricted to the time-interval $[\frac12 t_0, t_0]$ are performed by $\delta(t_0 / 2)$-precise cutoff, where
\[ \delta (t_0 / 2) < \delta_{\ref{Prop:genPerelman}} (\tfrac12 \un{r}_{\varepsilon_0} (t_0), 2 t_0^{-1}, \linebreak[1] 1, \linebreak[1] \un{E}_{\varepsilon_0}, \linebreak[1] \un{\eta}, \linebreak[1] 0) \leq \delta_{\ref{Prop:genPerelman}} (\tfrac12 \un{r}_{\varepsilon_0} (t_0), \linebreak[1] \td{c} w, \linebreak[1] 1, \linebreak[1] \un{E}_{\varepsilon_0}, \linebreak[1] \un{\eta}, \linebreak[1] 0). \]
Moreover, $\td\MM$ has complete time-slices without boundary and the curvature on $\td\MM$ is uniformly bounded on compact time-intervals that don't contain surgery points.
Let $\td{x}_0 \in \td\MM (t_0)$ be a lift of $x_0 \in \MM (t_0)$.
Set $r_2  = \min\{ \rho_{r_0}(x_0, t_0), \ov{r}_{\ref{Prop:genPerelman}} \sqrt{t_0}, \frac12 \sqrt{t_0} \}$, where $\ov{r}_{\ref{Prop:genPerelman}}  = \ov{r}_{\ref{Prop:genPerelman}} (\td{c} w,1, \un{E}_{\varepsilon_0}, \un\eta)$ is the constant from Proposition \ref{Prop:genPerelman}. 
Then $\sec_{t_0} \geq - r_2^{-2}$ on $B (\td{x}_0, t_0, r_2)$ and $\vol_{t_0} B(\td{x}_0, t_0, r_2) \geq \td{c} w r_2^3$.
We now apply Proposition \ref{Prop:genPerelman} to $\td\MM$ with $x_0 \leftarrow \td{x}_0$, $t_0 \leftarrow t_0$, $r_0 \leftarrow r_2$, $r \leftarrow \frac12 \un{r}_{\varepsilon_0} (t_0)$, $w \leftarrow \td{c} w$, $A \leftarrow 1$, $E \leftarrow \un{E}_{\varepsilon_0}$, $\eta \leftarrow \un\eta$, $m \leftarrow 0$.
Then we obtain that if $r_2 = \rho(x_0, t_0)$, then $r_2 > \widehat{r}_{\ref{Prop:genPerelman}}(\td{c} w, \un{E}_{\varepsilon_0}, \un\eta) \sqrt{t_0}$.
This implies that $\rho(x_0, t_0) > \min \{ \min \{ \widehat{r}_{\ref{Prop:genPerelman}}, \ov{r}_{\ref{Prop:genPerelman}}, \frac12 \} \sqrt{t_0}, r_0 \}$ and hence the first claim for $\ov{\rho} = \min \{ \widehat{r}_{\ref{Prop:genPerelman}}, \ov{r}_{\ref{Prop:genPerelman}}, \frac12 \}$.
Note that with this choice of $\ov\rho$, we have $r_1 = \min \{ \ov\rho \sqrt{t_0}, r_0 \} \leq r_2$.

Next, observe that by the choice of $T$, we have
\[ C_{1, \ref{Prop:genPerelman}} \delta (t_0/2) \sqrt{t_0} < C_{1, \ref{Prop:genPerelman}} (T/2)^{-1} \cdot \theta^{-1} r_0 \leq r_0. \]
So by the second part of Proposition \ref{Prop:genPerelman}, we obtain that the parabolic neighborhood $P(\td{x}_0, t_0, r_2, - \tau_{\ref{Prop:genPerelman}} (\td{c} w, 1, \un{E}_{\varepsilon_0} , \un\eta ) r_2^2)$ is non-singular and that we have $|{\Rm}| < K_{0, \ref{Prop:genPerelman}} (\td{c} w, 1, \un{E}_{\varepsilon_0} , \un\eta ) r_2^{-2}$ there.
This implies the second claim since $r_1 \leq r_2$.
\end{proof}

\subsection{Bounded curvature at bounded distance from sufficiently collapsed and good regions} \label{subsec:boundedcurvboundeddistgoodregions}
We now extend the curvature bound from Proposition \ref{Prop:curvcontrolgood} to balls of larger radii $A r_0$.
It is crucial here that by assuming sufficient collapsedness around the basepoint (depending on $A$), we don't have to impose an assumption of the form  $r_0 < \ov{r}(w, A) \sqrt{t_0}$ as in Proposition \ref{Prop:genPerelman}.
So the quantity $A r_0 t_0^{-1/2}$ can indeed be chosen arbitrarily large.
\begin{Proposition} \label{Prop:curvcontrolincompressiblecollapse}
There is a continuous positive function $\delta : [0, \infty) \to (0, \infty)$ such that for any $w, \theta > 0$ and $1 \leq A < \infty$ there are $\tau = \tau(w), \ov{\rho} = \ov{\rho} (w, A), \ov{w} = \ov{w}(w, A) > 0$ and $K(w, A), T (w, A, \theta) < \infty$ such that: \\
Let $\MM$ be a Ricci flow with surgery on the time-interval $[0, \infty)$ with normalized initial conditions that is performed by $\delta(t)$-precise cutoff and let $t_0 > T$.
Choose $x_0 \in \MM(t_0)$ and $r_0 > 0$ and assume that
\begin{enumerate}[label=(\roman*)]
\item $\theta \sqrt{t_0} \leq r_0 \leq \sqrt{t_0}$,
\item $x_0$ is $w$-good at scale $r_0$ and time $t_0$,
\item $\vol_{t_0} B(x_0, t_0, r_0) < \ov{w} r_0^3$.
\end{enumerate}
Then $|{\Rm}| < K r_0^{-2}$ on $B = \bigcup_{t \in [t_0 - \tau r_0^2, t_0]} B(x_0, t, A r_0)$ and there are no surgery points on $B$.

In particular, if $r_0 \geq \rho(x_0, t_0)$, then $\rho( x_0, t_0) > \ov\rho \sqrt{t_0}$ and the curvature estimate becomes $| {\Rm} | < K t_0^{-1}$.
\end{Proposition}

\begin{proof}
We first set up an argument in the spirit of the proof of Corollary \ref{Cor:Perelman68}.
Choose $\varepsilon_0 > 0$ to be smaller than the corresponding constant in Lemma \ref{Lem:6.3bc} and the constant $\td\varepsilon_0$ in Lemma \ref{Lem:neckhasfewquotients}.
By Proposition \ref{Prop:CNThm-mostgeneral} there are decreasing continuous positive functions $\un{r}_{\varepsilon_0}, \un{\delta}_{\varepsilon_0} : [0, \infty) \to (0, \infty)$ such that if $\delta(t) \leq \un\delta_{\varepsilon_0} (t)$ for all $t \in [0, \infty)$, then every point $(x,t) \in \MM$ satisfies the canonical neighborhood assumptions $CNA (\un{r}_{\varepsilon_0} (t) \sqrt{t}, \varepsilon_0, E, \eta)$ for any constants $0 < \eta < \un\eta$, $\un{E}_{\varepsilon_0} < E < \infty$.
Without loss of generality, we can assume that $E > E_{0, \ref{Lem:6.3bc}} (\varepsilon_0)$ and $\eta < \eta_{0, \ref{Lem:6.3bc}}$ where $E_{0, \ref{Lem:6.3bc}}$ and $\eta_{0, \ref{Lem:6.3bc}}$ are the constants from Lemma \ref{Lem:6.3bc}.
Consider the constant $\delta_{\ref{Lem:6.3bc}} (w', A', r', \varepsilon_0, E, \eta)$ and assume that it depends on its parameters $w', A', r'$ in a monotone way, i.e. $\delta_{\ref{Lem:6.3bc}} (w'', A'', r'', \varepsilon_0, E, \eta) \leq \delta_{\ref{Lem:6.3bc}} (w', A', r', \varepsilon_0, E, \eta)$ whenever $w'' \leq w'$, $A'' \geq A'$ and $r'' \leq r'$.
Let $\delta_{\ref{Prop:curvcontrolgood}}$ be the constant from Proposition \ref{Prop:curvcontrolgood} and assume that for all $t > 0$
\[ \delta(t) < \min \big\{ \delta_{\ref{Lem:6.3bc}} (t^{-1}, t, \tfrac12 \un{r}_{\varepsilon_0}(2t), \varepsilon_0, E, \eta), \; \un{\delta}_{\varepsilon_0} (t), \; \delta_{\ref{Prop:curvcontrolgood}}(t), \; t^{-1} \big\}. \]

By Proposition \ref{Prop:curvcontrolgood}, and for large enough $T$ depending on $w$ and $\theta$, we have $\rho(x_0, t_0) > r_ 1 = \min \{ \ov{\rho}_{\ref{Prop:curvcontrolgood}} (w)  \sqrt{t_0}, r_0 \}$ and $|{\Rm}|< K_{\ref{Prop:curvcontrolgood}} (w) r_0^{-2}$ on the non-singular parabolic neighborhood $P(x_0, \linebreak[1] t_0, \linebreak[1] r_1, \linebreak[1] -\tau_{\ref{Prop:curvcontrolgood}} \linebreak[1] (w) r_1^2)$.
In particular, this shows how the last assertion of the proposition follows from the first one.

It remains to prove the first assertion.
Consider the constants $\tau_{\ref{Prop:curvcontrolgood}}(w)$, $K_{\ref{Prop:curvcontrolgood}}(w)$ from Proposition \ref{Prop:curvcontrolgood} and set
\[ \gamma = \gamma(w) = \tfrac1{10} \min \big\{ 1, \tau_{\ref{Prop:curvcontrolgood}}^{1/2}(w), K_{\ref{Prop:curvcontrolgood}}^{-1/2}(w) \big\}. \]
Consider the universal covering flow $\td\MM$ of $\MM$ as described in Lemma \ref{Lem:tdMM} and let $\td{x}_0 \in \td\MM (t_0)$ be a lift of $x_0$.
By the choice of $\gamma$ we have
\[ |{\Rm}| < (\gamma r_0)^{-2} \quad \text{on} \quad P(\td{x}_0, t, \gamma r_0, - (\gamma r_0)^2) \quad \text{for all} \quad t \in [t_0 - (\gamma r_0)^2, t_0] \]
and $\vol_t B(\td{x}_0, t, \gamma r_0) > \frac1{10} \td{c} w (\gamma r_0)^3$ for all such $t$.
We now argue that for sufficiently large $T$, depending on $w, A$, we can apply Lemma \ref{Lem:6.3bc}(a) with $\MM \leftarrow \td\MM$, $\widetilde{x}_0 \leftarrow x_0$, $t_0 \leftarrow t \in [t_0 - (\gamma r_0)^2, t_0]$, $r_0 \leftarrow \gamma r_0$, $w \leftarrow \frac1{10} \td{c} w$, $A \leftarrow \gamma^{-1} (A + 1)$, $r \leftarrow \frac12 \un{r}_{\varepsilon_0} (t_0)$, $\eta \leftarrow \eta$, $\varepsilon \leftarrow \varepsilon_0$, $E \leftarrow E$.
First, choose $T = T(w, A, \theta)$ large enough such that $2 T^{-1} < \frac1{10} \td{c} w$ and $\frac12 T > \gamma^{-1} (A+1)$.
Observe that for all $(x',t') \in \MM$ with $t' \in [\frac12 t_0, t_0]$ the canonical neighborhood assumptions $CNA (\frac12 \un{r}_{\varepsilon} (t_0), \varepsilon_0, E, \eta)$ hold.
So these canonical neighborhood assumptions also hold for all $(x',t') \in \td\MM$ with $t' \in [\frac12 t_0, t_0]$.
Moreover, by the choice of $T$, we have $\delta(t') <  \delta_{\ref{Lem:6.3bc}} (\frac1{10} \td{c} w, \gamma^{-1}(A+1), \frac12 \un{r}_{\varepsilon_0}(t_0), \varepsilon_0, E, \eta)$ for all $t' \in [\frac12 t_0, t_0]$.
So Lemma \ref{Lem:6.3bc}(a) can be applied and we conclude that for any $t \in [t_0 - (\gamma r_0)^2, t_0]$ the points in $B(\td{x}_0, t, (A+1)r_0) \subset \td\MM (t)$ satisfy the canonical neighborhood assumptions $CNA(\gamma \td{\rho}_{\ref{Lem:6.3bc}} r_0, \varepsilon_0, E, \eta)$.
Here $\td{\rho}_{\ref{Lem:6.3bc}} = \td{\rho}_{\ref{Lem:6.3bc}}(\frac1{10} \td{c} w, \gamma^{-1} (A+1), \varepsilon_0, E, \eta)$.

Set $\tau = \tau(w) = \gamma^2(w)$ and $K = \gamma^{-2} \max \{ \td{\rho}^{-2}_{\ref{Lem:6.3bc}}, E^2 \}$.
So, if $|{\Rm}|(x,t) \geq K r_0^{-2}$ for some $t \in [t_0 - \tau r_0^2, t_0]$ and $x \in B(\td{x}_0, t, (A+1)r_0)$, then\begin{equation} \label{eq:deretabound}
 |\nabla |{\Rm}|^{-1/2}|(x,t) < \eta^{-1}
\end{equation}
and $(x,t)$ is a center of a strong $\varepsilon_0$-neck or an $(\varepsilon_0, E)$-cap or the component of $\MM (t)$ in which $x$ lies has positive, $E^2$-pinched sectional curvatures.
In the last case we are done, since $K \geq E^2$.
So assume that $(x,t)$ is a center of a strong $\varepsilon_0$-neck or an $(\varepsilon_0, E)$-cap.

Fix some $t \in [t_0 - \tau r_0^2, t_0]$.
Let $a \leq A$ be maximal with the property that $|{\Rm_t}| < K r_0^{-2}$ on $B(\td{x}_0, t, a r_0)$.
If $a = A$, we are done, so assume $a < A$.
By (\ref{eq:deretabound}), we can conclude that (compare also with Lemma \ref{Lem:shortrangebounds})
\begin{equation} \label{eq:etaextcurvbound}
 |{\Rm_t}| < 4 K r_0^{-2} \qquad \text{on} \qquad B(\td{x}_0, t, a r_0 + \tfrac12 \eta K^{-1/2} r_0).
\end{equation}
By the choice of $a$ we can find a point $\td{x}_1 \in \td\MM(t)$ of time-$t$ distance exactly $a r_0$ from $\td{x}_0$ with $|{\Rm}| (\td{x}_1, t) = K r_0^{-2}$.
So $(\td{x}_1, t)$ is a center of an $\varepsilon_0$-neck or an $(\varepsilon_0, E)$-cap in $\td\MM (t)$.

Let $x_1 \in \MM(t)$ be the projection of $\td{x}_1$.
By (\ref{eq:etaextcurvbound}) and volume comparison, we can crudely estimate that for some constant $C = C(w, A) < \infty$, which only depends on $w, A$, and some universal constant $C' < \infty$
\begin{multline}
 \vol_{t} B(x_1, t, \tfrac12 \eta K^{-1/2} r_0) < \vol_t B(x_0, t, a r_0 + \tfrac12 \eta K^{-1/2} r_0) \\
 < C(A, K) \vol_t B(x_0, t, \tfrac1{10} r_0) 
<  C' C(A, K) \vol_{t_0} B(x_0, t_0, r_0) \\
< C' C(w, A) \ov{w} r_0^3. \label{eq:etaballsmallvol}
\end{multline}
If $(\td{x}_1, t)$ is a center of an $\varepsilon_0$-neck, then we obtain a contradiction using Lemma \ref{Lem:neckhasfewquotients} assuming $\ov{w}$ is chosen small enough such that $C' C(w, A) \ov{w} < \td{w}_0 (\frac12 \eta K^{-1/2})^3$ (here $\td{w}_0$ is the constant from Lemma \ref{Lem:neckhasfewquotients}).
So assume for the rest of the proof that $(\td{x}_1, t)$ is a center of an $(\varepsilon_0, E)$-cap $U \subset \td\MM (t)$.
Let $K \subset U$ be a compact subset such that $\td{x}_1 \in K$ and $U \setminus K$ is an $\varepsilon_0$-neck and let $\td{y} \in U$ be a center of this neck.
By Definition \ref{Def:epscap} we have $\gamma^{-2} r_0^{-2} \leq E^{-2} K r_0^{-2} \leq |{\Rm}| \leq E^2 K r_0^{-2}$ on $U$.
So $\td{x}_0 \not\in U$ and hence the minimizing geodesic segment between $\td{x}_0$ and $\td{x}_1$ passes through the whole $\varepsilon_0$-neck $U \setminus K$.
So for sufficiently small $\varepsilon_0$ we have $\dist_t (\td{x}_0, \td{y}) < \dist_t(\td{x}_0, \td{x}_1) = a r_0$.
In particular, for the projection $y$ of $\td{y}$ we find $B(y, t, \frac12 \eta E^{-1} K^{-1/2} r_0) \subset B(x_0, t, ar_0 + \frac12 \eta K^{-1/2} r_0)$.
Now again, using Lemma \ref{Lem:neckhasfewquotients} and (\ref{eq:etaballsmallvol}), we conclude
\[ \td{w}_0 \big( \tfrac12 \eta E^{-1} K^{-1/2} \big)^3 r_0^3 < \vol_t B(y, t, \tfrac12 \eta E^{-1} K^{-1/2} r_0) < C' C(w, A) \ov{w} r_0^3. \]
This yields a contradiction for sufficiently small $\ov{w}$, depending on $w$ and $A$.

It remains to show that there are no surgery points on $B$.
To see this, observe that $|{\Rm}| < K \theta^{-2} t_0^{-1}$ on $B$, but by (\ref{eq:surgpointhighcurv}) we have
\[ |{\Rm}|(x, t) > c' \delta^{-2} (t) t^{-1} \geq c' \delta^{-2} (T/2) t_0^{-1} \geq \tfrac14 c' T^2 r_0^{-2} \]
at every surgery point $(x,t) \in \MM$ for some universal $c' > 0$.
So choosing $T$ large enough, depending on $w, A$ and $\theta$, yields the desired result.
\end{proof}

\subsection{Curvature control at points that are good relative to regions whose boundary is geometrically controlled} \label{subsec:curvboundinbetween}
Next, we generalize Proposition \ref{Prop:curvcontrolgood} to points that are good relative to some open set $U$.
In order to do this, we need to assume that the metric around the boundary of $U$ is sufficiently controlled on a time-interval of uniform size.

\begin{Proposition} \label{Prop:curvboundinbetween}
There is a continuous positive function $\delta : [0, \infty) \to (0, \infty)$ such that for any $w, \theta > 0$ there are $\alpha = \alpha (w) > 0$ and $K = K(w), T = T(w, \theta) < \infty$ such that: \\
Let $\MM$ be a Ricci flow with surgery on the time-interval $[0, \infty)$ with normalized initial conditions that is performed by $\delta(t)$-precise cutoff and let $t_0 > T$.
Let $0 < r_0 \leq \sqrt{t_0}$ and consider a sub-Ricci flow with surgery $U \subset \MM$ (see Definition \ref{Def:subRF}) on the time-interval $[t_0 - r_0^2, t_0]$, whose time-slices $U(t)$ are closed subsets of $\MM(t)$.
Finally, let $x_0 \in U (t_0)$ be a point and assume that
\begin{enumerate}[label=(\roman*)]
\item $\theta \sqrt{t_0} \leq r_0 \leq \sqrt{t_0}$,
\item for all $x \in \partial U (t_0)$, the parabolic neighborhood $P(x, t_0, 2 r_0, - r_0^2)$ is non-singular and we have $| {\Rm} | < r_0^{-2}$ there,
\item $x_0$ is $w$-good at scale $r_0$ relative to the interior of $U (t_0)$ at time $t_0$.
\end{enumerate}
Then the parabolic neighborhood $P(x_0, t_0,  \alpha r_0, - \alpha^2 r_0^2)$ is non-singular and we have $|{\Rm}| < K r_0^{-2}$ there.
\end{Proposition}
\begin{proof}
The idea of the proof will be to apply Proposition \ref{Prop:genPerelman} to the universal covering flow $\td{U}$ of $U$ (see Lemma \ref{Lem:tdMM}).
So our main task will be to verify assumptions (i) and (ii) of that Proposition.
Besides that, the proof essentially follows along the lines of the proof of Proposition \ref{Prop:curvcontrolgood}.

We first choose the function $\delta(t)$.
Let $\varepsilon_0 > 0$ be the constant from Proposition \ref{Prop:genPerelman} and consider the constants $\un{E}_{\varepsilon_0}$, $\un\eta$ and the functions $\un\delta_{\varepsilon_0} (t), \un{r}_{\varepsilon_0} (t)$ from Proposition \ref{Prop:CNThm-mostgeneral}.
So if $\delta(t) < \un\delta_{\varepsilon_0}(t)$ for all $t \geq 0$, then $\MM$ satisfies the canonical neighborhood assumptions $CNA (\un{r}_{\varepsilon_0} (t) \sqrt{t}, \varepsilon_0, \un{E}_{\varepsilon_0}, \un\eta)$.
Without loss of generality, we assume that $\un{r}_{\varepsilon_0} (t) \to 0$ as $t \to \infty$.
Similarly as in the proof of Proposition \ref{Prop:curvcontrolgood} or Corollary \ref{Cor:Perelman68}, we assume that
\begin{equation} \label{eq:deltalessinrelativecase}
 \delta (t) < \min \big\{ \delta_{\ref{Prop:genPerelman}} ( \tfrac12 \un{r}_{\varepsilon_0} (2t), t^{-1}, 1, \un{E}_{\varepsilon_0}, \un\eta, 0), \; \un\delta_{\varepsilon_0} (t), \; t^{-1}, \; 1 \big\},
\end{equation}
where $\delta_{\ref{Prop:genPerelman}}$ is the constant in Proposition \ref{Prop:genPerelman}, which we assume to satisfy the before-mentioned monotonicity property.
Furthermore, we assume that $T = T(w, \theta)$ is large enough such that $2 T^{-1} < \td{c} w$ and such that $\un{r}_{\varepsilon_0} (t) < \frac1{10}  \theta \min \{ 1, \linebreak[1] \un{E}_{\varepsilon_0}^{-1}, \linebreak[1]\varepsilon_0 \}$ for all $t \geq \frac12 T$.

We now present the main argument.
By assumption (ii), we can consider the case in which $B(x_0, t_0, r_0) \subset U(t_0)$.
Our goal will be to apply Proposition \ref{Prop:genPerelman} in the universal covering flow $\td{U}$ of $U$ (see Lemma \ref{Lem:tdMM}) at a lift $(\td{x}_0, t_0) \in \td{U} (t_0)$ of $(x_0, t_0) \in U (t_0)$.
We first check that all points $(x,t) \in \td{U}$ with $t \in [t_0 - \frac12 r_0^2, t_0]$ satisfy the canonical neighborhood assumptions $CNA ( \frac12 \un{r}_{\varepsilon_0} (t_0) \sqrt{t_0}, \varepsilon_0, \un{E}_{\varepsilon_0}, \un\eta)$.
To do this, consider first a point $(x,t) \in U \subset \MM$ with $t \in [ t_0 - \frac12 r_0^2, t_0] \subset [\frac12 t_0, t_0]$.
By the previous conclusion, $(x,t)$ satisfies the desired canonical neighborhood assumptions in $\MM$.
We now argue that $(x,t)$ satisfies those canonical neighborhood assumptions also in $U$.
If $|{\Rm}|^{-1/2} (x,t) > \frac12 \un{r}_{\varepsilon_0} (t_0) \sqrt{t_0}$, then there is nothing to show.
So assume that (the second inequality holds by the choice of $T$)
\begin{multline} \label{eq:Rmless1Eeps}
|{\Rm}|^{-1/2} (x,t) \leq \tfrac12 \un{r}_{\varepsilon_0} (t_0) \sqrt{t_0} 
< \tfrac1{20}  \theta \min \{ 1, \un{E}_{\varepsilon_0}^{-1}, \varepsilon_0 \} \sqrt{t_0} \\
 \leq \tfrac1{10} \big(\max \{ 1, \un{E}_{\varepsilon_0}, 2 \varepsilon_0^{-1} \} \big)^{-1} r_0.
\end{multline}
Then, in particular, $|{\Rm}|(x,t) > r_0^{-2}$, which implies by assumption (ii) that $(x,t) \not\in P(x', t_0, 2r_0, - r_0^2)$ for all $x' \in \partial U(t_0)$ and hence $B(x, t, \frac1{10} r_0) \subset U(t)$ (recall from Definition \ref{Def:subRF} that the boundary $\partial U (t')$ is constant in time).
The point $(x,t)$ is a center of a strong $\varepsilon_0$-neck or an $(\varepsilon_0, \un{E}_{\varepsilon_0})$-cap in $\MM$.
The time-$t$ slice of this strong $\varepsilon_0$-neck or $(\varepsilon_0, \un{E}_{\varepsilon_0})$-cap is contained in the ball (compare with (\ref{eq:Rmless1Eeps}))
\[ B \big( x, t, \max \{ \un{E}_{\varepsilon_0}, 2 \varepsilon_0^{-1} \} |{\Rm}|^{-1/2} (x,t) \big) \subset B(x,t, \tfrac1{10} r_0) \subset U(t). \]
Moreover, if $(x,t)$ is the center of a strong $\varepsilon_0$-neck, then parabolic domain of this strong neck can be chosen such that its initial time is larger than $t - 2 |{\Rm}|^{-1} (x,t) > t_0 - \frac12 r_0^2 - \frac1{50} r_0^2 > t - r_0^2$.
So $(x,t)$ in fact satisfies the canonical neighborhood assumptions $CNA (\frac12 \un{r}_{\varepsilon_0} (t_0) \sqrt{t_0}, \linebreak[1] \varepsilon_0, \linebreak[1] \un{E}_{\varepsilon_0}, \un\eta)$ in $U$.
It follows that all points $(x,t) \in \td{U}$ with $t \in [t_0 - \frac12 r_0^2, t_0]$ satisfy those canonical neighborhood assumptions in $\td{U}$.

Let $\td{x}_0 \in \td{U}(t_0)$ be a lift of $x_0 \in U(t_0)$.
Note that all surgeries on $\td{U}$ in the time-interval $[t_0 - \frac12 r_0^2, t_0]$ are performed by $\delta_{\ref{Prop:genPerelman}} ( \frac12 \un{r}_{\varepsilon_0} (t_0), \td{c} w, 1, \un{E}_{\varepsilon_0}, \un\eta, 0)$-precise cutoff (here we have used (\ref{eq:deltalessinrelativecase}) and $2 T^{-1} < \td{c} w$).
So whenever $r_1 \leq \min \{ \rho(x_0, t_0), \linebreak[1] \frac12 r_0, \linebreak[1] \ov{r}_{\ref{Prop:genPerelman}} (\td{c} w, 1, \un{E}_{\varepsilon_0}, \un\eta ) r_0 \}$, where $\ov{r}_{\ref{Prop:genPerelman}}$ is the constant from Proposition \ref{Prop:genPerelman}, then the first paragraph of the assumptions as well as assumptions (iii)--(v) of Proposition \ref{Prop:genPerelman} are satisfied for $\MM \leftarrow \td{U}$, $t_0 \leftarrow t_0$, $x_0 \leftarrow \td{x}_0$, $r_0 \leftarrow r_1$, $w \leftarrow \td{c} w$, $A \leftarrow 1$, $r \leftarrow \frac12 \un{r}_{\varepsilon_0}(t_0)$, $E \leftarrow \un{E}_{\varepsilon_0}$, $\eta \leftarrow \un\eta$, $m \leftarrow 0$.
We will now argue that assumptions (i) and (ii) are satisfied for the right choice of $r_1$, i.e. we will show that there is a constant $\beta = \beta (w) > 0$ (depending only on $w$) such that these assumptions hold whenever $r_1 \leq \min \{ \rho(x_0, t_0), \min \{ \frac12, \beta, \ov{r}_{\ref{Prop:genPerelman}} (\td{c} w, 1, \un{E}_{\eps_0}, \un\eta ) \} r_0 \}$.

Consider first assumption (ii).
Since $B(x_0, t_0, r_0) \subset U(t_0)$, we have $\dist_{t_0} (x_0, \linebreak[1] \partial U(t_0)) \geq r_0$.
Let $x \in B(\td{x}_0, t_0, \beta r_0)$ be a point that survives until some time $t \in [t_0 - \frac1{10} \beta^2 r_0^2, t_0]$.
Then $\dist_{t_0} (x, \partial \td{U}(t_0)) > \frac12 r_0$ for $\beta < \frac12$ and we conclude, using distance distortion estimates and assumption (ii) of this proposition, that $\dist_t (x, \partial \td{U}(t) ) > \frac1{20} r_0$.
So assumption (ii) of Proposition \ref{Prop:genPerelman} holds if $\beta < \frac1{200}$.

Assumption (i) of Proposition \ref{Prop:genPerelman} requires more work.
Set $Z = Z_{\ref{Prop:genPerelman}} ( \td{c} w, \linebreak[1] 1, \linebreak[1] \un{E}_{\varepsilon_0}, \linebreak[1] \un\eta)$.
Let $t_1, t_2 \in [t_0 - \frac1{10} \beta^2 r_0^2, t_0]$, $t_1 < t_2$ and consider some point $x \in B(\td{x}_0, t_0, \beta r_0)$ that survives until time $t_2$ and a space-time curve $\gamma : [t_1, t_2] \to \MM$ with endpoint $\gamma(t_2) \in B(x, t, 4 \beta r_0)$ and that meets the boundary $\partial \td{U}$.
We want to show that for a sufficiently small choice of $\beta$ we have $\LL(\gamma) > Z \beta r_0$.
Similarly as in the last paragraph, we conclude that $\dist_{t_0} (x, \partial \td{U}(t_0) ) > \frac12 r_0$ if $\beta < \frac1{2}$ and that $\dist_{t_0} (\gamma(t_2), \partial \td{U}(t_0) ) > \frac14 r_0$ if $\beta < \frac1{200}$.
So assume from now on that $\beta < \frac1{200}$.
Let now
\[ P = \bigcup_{x' \in \partial \td{U}(t_0)} P(x', t_0, \tfrac14 r_0, - r_0^2) \]
be a parabolic collar neighborhood of $\partial \td{U}$.
Recall that $P$ is non-singular and $|{\Rm}| < r_0^{-2}$ on $P$.
Since $\dist_{t_0} (\gamma(t_2), \partial \td{U}(t_0) ) > \frac14 r_0$, we have $(\gamma(t_2), t_2) \not\in P$ and we can find a time-interval $[t'_1, t'_2] \subset [t_1, t_2]$ such that $\gamma(t'_1) \in \partial \td{U}(t'_1)$, $\gamma( [t'_1, t'_2) ) \subset P$ and such that $\dist_{t_0} (\gamma(t'_2), \partial \td{U}' (t'_2)) \geq \frac14  r_0$.
Then we can estimate, using the $t^{-1}$-positive curvature condition and the fact that $t_2 - t_1 \leq \frac1{10} \beta^2 r_0^2$,
\begin{multline*}
 \LL(\gamma) \geq \int_{t_1}^{t_2} \sqrt{t_2 - t} \big( |\gamma'(t)|^2_t - \tfrac32 t^{-1} \big) dt \geq \int_{t'_1}^{t'_2} \sqrt{t_2 - t} \big| \gamma'(t) \big|_t^2 dt -  \beta r_0 \\
 \geq \frac1{100} \int_{t'_1}^{t'_2} \sqrt{t_2' - t} |\gamma'(t)|_{t_0}^2 dt -  \beta r_0. 
\end{multline*}
Substituting $s^2 = t'_2 - t$ and setting $s_1^2 = t'_2 - t'_1$ yields
\begin{multline*}
 \int_{t'_1}^{t'_2} \sqrt{t_2' - t} |\gamma'(t)|_{t_0}^2 dt = \frac12 \int_0^{s_1} \Big| \frac{d}{ds} \gamma (t'_2 - s^2) \Big|^2_{t_0} ds \\ \geq \frac1{2s_1} \dist_{t_0}^2 (\gamma(t'_2), \gamma (t'_1)) 
 \geq \frac{r^2_0}{32 \sqrt{t'_2 - t'_1}} \geq \frac1{32\beta} r_0.
\end{multline*}
Thus
\[ \LL(\gamma) > \Big( \frac{1}{4000 \beta} -  \beta \Big) r_0. \]
For sufficiently small $\beta$, depending only on $w$, the right hand side is larger than $Z \beta r_0$.
So we have verified all assumptions of Proposition \ref{Prop:genPerelman}.

We can finally apply Proposition \ref{Prop:genPerelman} with the parameters mentioned before and with $r_1 = \min \{ \rho(x_0, t_0), \min \{ \frac12, \beta, \ov{r}_{\ref{Prop:genPerelman}} (\td{c} w, 1, \un{E}_{\varepsilon_0}, \un\eta ) \} r_0 \}$.
We first obtain that if $r_1 = \rho (x_0, t_0)$, then $r_1 > \widehat{r}_{\ref{Prop:genPerelman}}(\td{c} w, 1, \un{E}_{\varepsilon_0}, \un\eta ) \sqrt{t_0}$, where $\widehat{r}_{\ref{Prop:genPerelman}}$ is the constant from Proposition \ref{Prop:genPerelman}.
This implies that we always have $r_1 > \widehat{r}_1 r_0$, where $\widehat{r}_1 = \widehat{r}_1 (w) = \min \{ \frac12, \beta, \ov{r}_{\ref{Prop:genPerelman}}, \widehat{r}_{\ref{Prop:genPerelman}} \}$.
Let us now assume that $T = T(w, \theta)$ is large enough such that $\delta(t) < \widehat{r}_1 \theta C_{1, \ref{Prop:genPerelman}}^{-1}(\td{c} w, 1, \un{E}_{\varepsilon_0}, \un\eta)$ for all $t \in [\frac12 t_0, t_0]$.
Here $C_{1, \ref{Prop:genPerelman}}$ is the constant from Proposition \ref{Prop:genPerelman}.
Then for all $t \in [\frac12 t_0, t_0]$, we have $C_{1, \ref{Prop:genPerelman}} \delta (t) \sqrt{t_0} \leq \widehat{r}_1 \theta \sqrt{t_0} \leq \widehat{r}_1 r_0 < r_1$.
So by the last part of Proposition \ref{Prop:genPerelman}, the parabolic neighborhood $P(\td{x}_0, t_0, r_1, - \tau_{\ref{Prop:genPerelman}} (\td{c} w, 1, \un{E}_{\varepsilon_0}, \un\eta) r_1^2)$ is non-singular and we have $|{\Rm}| < K_{0, \ref{Prop:genPerelman}}(\td{c} w, 1, \un{E}_{\varepsilon_0}, \un\eta) r_1^{-2}$ there.
This implies the result for $\alpha = \alpha (w) = \min \{ \widehat{r}_1, \tau_{\ref{Prop:genPerelman}}^{1/2} \}$.
\end{proof}

\subsection{Controlled diameter growth of regions whose boundary is sufficiently collapsed and good} \label{subsec:controlleddiamgrowth}
In this subsection we show that if a region $U$ in a Ricci flow with surgery has bounded diameter at some time $t_1$, then we can bound its curvature and diameter from above at some slightly later time $t_2 > t_1$ if the geometry around the boundary $\partial U$ satisfies certain collapsedness and goodness assumptions.
The important point is hereby that the size of the time-interval $[t_1, t_2]$ does not depend on the diameter of $U$ at time $t_1$.
We are able to guarantee this independence by imposing a collapsedness condition, which depends on the diameter of $U$ at time $t_1$.
Note that the opposite statement is most likely wrong, namely a bound on the diameter of $U$ at time $t_1$, does not necessarily imply a bound on the diameter at \emph{earlier} times $t_2 < t_1$, even under very strong collapsedness assumptions.
For example, if we consider a parabolic rescaling of the cigar soliton, with a very small scaling factor, then the diameter of a region around its tip contracts arbitrarily fast under the Ricci flow.
The statement of the following proposition is that in certain settings diameters, however, cannot ``\emph{expand} too fast''.

The idea of the following proof is that, by an $\LL$-geometry argument similar to Lemma \ref{Lem:6.3a}, we can deduce a $\kappa$-noncollapsedness result where the constant $\kappa$ only depends on the diameter of $U$ \emph{at time $t_1$}.
Then an argument similar to the one in the proof of Lemma \ref{Lem:6.3bc}(b) will help us derive more uniform canonical neighborhood assumptions on $U$ and finally an argument similar to the one in the proof of Proposition \ref{Prop:curvcontrolincompressiblecollapse} will yield a curvature bound on $U$.

\begin{Proposition} \label{Prop:slowdiamgrowth}
There is a continuous positive function $\delta : [0, \infty) \to (0, \infty)$ and for every $w > 0$ there is a $\tau_0 = \tau_0(w) > 0$ such that for all $\theta > 0$ and $A < \infty$ there are constants $\kappa = \kappa(w, A), \td\rho = \td\rho (w, A), \ov{w} = \ov{w} (w, A) > 0$ and $K = K (w, A), A' = A'(w, A), T = T(w, A, \theta) < \infty$ such that: \\
Let $\MM$ be a Ricci flow with surgery on the time-interval $[0, \infty)$ with normalized initial conditions that is performed by $\delta(t)$-precise cutoff and let $t_0 > T$.
Let $\tau \in (0, \tau_0]$ and $r_0 > 0$ and consider a sub-Ricci flow with surgery $U \subset \MM$ on the time-interval $[t_0 - \tau r_0^2, t_0]$.
Let $x_0 \in U(t_0)$ be a point that survives until time $t_0 - \tau r_0^2$.
Assume that
\begin{enumerate}[label=(\roman*)]
\item $\theta \sqrt{t_0} \leq r_0 \leq \sqrt{t_0}$,
\item $x_0$ is $w$-good at scale $r_0$ and time $t_0$,
\item $\vol_{t_0} B(x_0, t_0, r_0) < \ov{w} r_0^3$,
\item $\partial U(t) \subset B(x_0, t, A r_0)$ for all $t \in [t_0 - \tau r_0^2, t_0]$,
\item $U(t_0 - \tau r_0^2) \subset B(x_0, t_0 - \tau r_0^2, A r_0)$.
\end{enumerate}
Consider the universal covering flow $\td\MM$ of $\MM$, as described in Lemma \ref{Lem:tdMM}, and let $\td{U} \subset \td\MM$ be the sub-Ricci flow with surgery for which $\td{U}(t) \subset \td\MM(t)$ is the preimage of $U(t)$ under the universal covering projection $\pi_t : \td\MM (t) \to \MM (t)$ for all $t \in [t_0 - \tau r_0^2, t_0]$.
Then
\begin{enumerate}[label=(\alph*)]
\item For all $t \in [t_0 - \tau r_0^2, t_0]$ all points of $\td{U}(t)$ are $\kappa$-noncollapsed on scales $< r_0$ in $\td\MM$.
\item There are universal constants $\eta > 0$, $E < \infty$ and $0 < \varepsilon \leq \td{\varepsilon}_0$ (where $\td\varepsilon_0$ is the constant from Lemma \ref{Lem:neckhasfewquotients}), which don't depend on $w, \theta, A$ or $\MM$, such that for every $t \in [t_0 - \tau r_0^2, t_0]$ the points in $\td{U}(t)$ satisfy the canonical neighborhood assumptions $CNA( \td\rho r_0, \linebreak[1] \varepsilon, \linebreak[1] E, \linebreak[1] \eta)$ in $\td\MM$.
\item There are no surgery points in $U$ (i.e. the Ricci flow with surgery $U$ is non-singular and we can write $U = U(t_0) \times [t_0 - \tau r_0^2, t_0]$) and we have $|{\Rm}| < K r_0^{-2}$ on $U(t_0) \times [t_0 - \tau r_0^2, t_0]$.
\item $U(t) \subset B(x_0, t, A' r_0)$ for all $t \in [t_0 - \tau r_0^2, t_0]$.
\end{enumerate}
\end{Proposition}
\begin{proof}
Let $\varepsilon = \min \{ \td\varepsilon_0, \varepsilon_{0, \ref{Lem:kappasolCNA}} \}$, where $\varepsilon_{0, \ref{Lem:kappasolCNA}}$ is the constant from Lemma \ref{Lem:kappasolCNA}.
Consider the functions $\underline{\delta}_{\varepsilon} (t), \underline{r}_{\varepsilon} (t)$ and the constants $\un{E}_{\varepsilon}, \un\eta$ from Proposition \ref{Prop:CNThm-mostgeneral} and the function $\delta_{\ref{Prop:curvcontrolincompressiblecollapse}} (t)$ from Proposition \ref{Prop:curvcontrolincompressiblecollapse}.
Without loss of generality, we may assume that $\un{r}_{\varepsilon} (t) \to 0$ as $t \to \infty$.
Let furthermore $\delta_0$ be the constant from Claim 1 and $\delta^* (\Lambda, r, \eta )$ the function from Claim 2 in the proof of Lemma \ref{Lem:6.3a}.
We can assume without loss of generality that $\delta^*$ is monotone in the sense that $\delta^* (\Lambda', r', \eta) \leq \delta^* (\Lambda, r, \eta)$ whenever $\Lambda' \geq \Lambda$ and $r' \leq r$.
Assume now that for all $t \geq 0$
\[ \delta(t) < \min \big\{ \delta^*( t, \tfrac14 \un{r}_{\varepsilon} (2 t), \un\eta ), \; \un\delta_{\varepsilon} (t), \; \delta_{\ref{Prop:curvcontrolincompressiblecollapse}} (t), \; t^{-1}  \un{r}_{\varepsilon} (2 t), \; \delta_0, \; t^{-1/2} \big\}. \]
We note that then, by Proposition \ref{Prop:CNThm-mostgeneral}, the flows $\MM$ and $\td\MM$ satisfy the canonical neighborhood assumptions $CNA (\un{r}_{\varepsilon} (t) \sqrt{t}, \varepsilon, \un{E}_{\varepsilon}, \un\eta)$ at any time $t$.

Set $\tau_0(w) = \min \{ \frac12 \tau_{\ref{Prop:curvcontrolincompressiblecollapse}}(w), 1\}$ and assume that $\ov{w} < \ov{w}_{\ref{Prop:curvcontrolincompressiblecollapse}} (w, 2A)$ and $T > T_{\ref{Prop:curvcontrolincompressiblecollapse}} (w, 2A, \theta)$, where $\tau_{\ref{Prop:curvcontrolincompressiblecollapse}}$, $\ov{w}_{\ref{Prop:curvcontrolincompressiblecollapse}}$ and $T_{\ref{Prop:curvcontrolincompressiblecollapse}}$ are the constants from Proposition \ref{Prop:curvcontrolincompressiblecollapse}.
Then, by Proposition \ref{Prop:curvcontrolincompressiblecollapse}, there is a constant $0 < \tau' = \tau' (w, A) < \tau_0 (w)$ such that the parabolic neighborhood $P(x_0, t_0 - \tau r_0^2, A r_0, - \tau' r_0^2)$ is non-singular and
\begin{equation}  \label{eq:curvboundonAparnbdtauprime}
 |{\Rm}| < K_1^* (w, A) r_0^{-2} \qquad \text{on} \qquad P(x_0, t_0 - \tau r_0^2, A r_0, - \tau' r_0^2)
\end{equation}
and such that the distance distortion on $P(x_0, t_0 - \tau r_0^2, A r_0, - \tau' r_0^2)$ can be controlled by a factor of $2$, i.e. $U(t) \subset B(x_0, t, 2 A r_0)$ for all $t \in [t_0 - (\tau + \tau') r_0^2, t_0 - \tau  r_0^2]$ (note that since the previous parabolic neighborhood is non-singular, we can extend $U$ to the time-interval $[t_0 - (\tau + \tau') r_0^2, t_0]$).
Moreover, we obtain the bound
\begin{equation} \label{eq:curvboundon2Afortau}
 |{\Rm}| < K_1^*(w, A) r_0^{-2} \qquad \text{on} \qquad B = \bigcup_{t \in [t_0 - (\tau + \tau') r_0^2, t_0]} B(x_0, t, 2 A r_0)
\end{equation}
and we can assume that there are no surgery points in $B$.

\textit{Proof of assertion (a). \quad}
We follow a modified version of the proof of Lemma \ref{Lem:6.3a}.
Let $t_1 \in [t_0 - \tau r_0^2, t_0]$, $\td{x}_1 \in \td{U}(t_1) \subset \td\MM (t_1)$, $0 < r_1 < r_0$ such that $P(\td{x}_1, t_1, r_1, - r_1^2) \subset \td\MM$ is non-singular and $|{\Rm}| < r_1^{-2}$ on $P(\td{x}_1, t_1, r_1, - r_1^2)$.

We first explain that, for sufficiently large $T$, we can restrict ourselves to the case $r_1 > \frac12 \underline{r}_{\varepsilon} (t_1) \sqrt{t_1} \geq \frac14 \underline{r}_{\varepsilon} (t_0) \sqrt{t_0}$.
Compare this statement with Claim 1 in the proof of Lemma \ref{Lem:6.3a} (applied to $\td\MM$).
As in the proof of this claim, we choose $s > 0$ to be the supremum over all $0 < r'_1 < r_0$ that satisfy the properties above, meaning that $|{\Rm}| < r_1^{\prime -2}$ on the non-singular parabolic neighborhood $P(\td{x}_1, t_1, r'_1, - r_1^{\prime 2})$.
Consider first the case in which $s \leq \frac12 \un{r}_{\varepsilon} (t_1) \sqrt{t_1}$.
Then we can argue as in cases (1), (2) of the proof of this claim.
Note that case (3) does not occur since we can assume that for large enough $T$ we have $s \leq \frac12 \un{r}_{\varepsilon} (t_1) \sqrt{t_1} < \theta \sqrt{t_0} \leq r_0$.
So it remains to consider the case $s > \frac12 \un{r}_{\varepsilon} (t_1) \sqrt{t_1}$.
We can then replace $r_1$ by some $r'_1 \in ( \frac12 \un{r}_{\varepsilon} (t_1) \sqrt{t_1}, s)$.
If we can prove the assertion for $r_1$ replaced by $r'_1$, then, by volume comparison, the assertion also holds for the original $r_1$ after reducing $\kappa$ by some uniform factor.

Let $x_1 \in \MM(t_1)$ be the projection of $\td{x}_1$.
Consider the functions $L$, $\ov{L}$ and the family of domains $D_t$ on $\MM$ based in $(x_1, t_1)$ (see the proof of Lemma \ref{Lem:6.3a} for more details).
Our first goal will be to show that $L(x_0, t_0 - (\tau + \frac12 \tau') r_0^2) < C_3 r_0$ for some universal $C_3 = C_3 (w, A, \tau) < \infty$.
An important tool will hereby be the following claim, which is analogous to Claim 2 in the proof of Lemma \ref{Lem:6.3a}:

\begin{Claim}
For any $\Lambda < \infty$ there is a $T^* = T^* (\Lambda) < \infty$ such that whenever $t_0 \geq T^*$, then the following holds: If $t \in [t_0 - (\tau + \tau') r_0^2, t_1)$, $x \in \MM(t)$, $r_1 > \frac14 \un{r}_{\varepsilon} (t_0) \sqrt{t_0}$ and $L(x, t) \leq \Lambda r_0$, then $x \in D_t$ and $(x,t)$ is not a surgery point.
Even more generally, there is a minimizing $\LL$-geodesic between $(x_1, t_1)$ and $(x,t)$ and any such minimizing $\LL$-geodesic does not meet surgery points.
\end{Claim}

\begin{proof}
This follows by the choice of $\delta$ in (\ref{eq:curvboundon2Afortau}) along with Claim 2 in the proof of Lemma \ref{Lem:6.3a} (applied to $\MM$) for $T^* (\Lambda) = \Lambda$.
\end{proof}

In contrast to the proof of Lemma \ref{Lem:6.3a}, we don't need to localize the function $\ov{L}$.
So we will only make use of the inequality
\begin{equation} \label{eq:evolovLinU}
 \Big( \frac{\partial}{\partial t^-} - \triangle \Big) \ov{L} (x, t) \geq - 6,
\end{equation}
which holds on $D_t$ in the barrier sense (cf \cite[7.1]{PerelmanI}).
We will now apply a maximum principle argument to (\ref{eq:evolovLinU}) to show that the following holds: $\inf_{U(t)} \ov{L} (\cdot, t) \leq 6 (t_1 - t)$ for all $t \in [t_0 - \tau r_0^2, t_1)$ or there is a $t \in [t_0 - \tau r_0^2, t_1)$ such that $\inf_{\partial U(t)} \ov{L} (\cdot, t) \leq 6 (t_1 - t)$.
Assume that neither of these cases occurs.
Since, $\ov{L}(x_1, t) \sim \const (t_1 - t)^2$ as $t \to t_1$, there is some $t' \in [t_0 - \tau r_0^2, t_1)$ such that $\inf_{U(t)} \ov{L} (\cdot, t) \leq 6 (t_1 - t)$ for all $t \in [t', t_1)$.
Choose $\nu > 0$ small enough such that $\inf_{\partial U(t)} \ov{L} (\cdot, t) > (6+\nu) (t_1 - t)$ for all $t \in [t_0 - \tau r_0^2, t']$ and choose $t^* \in [t_0 - \tau r_0^2, t']$ minimal with the property that $\inf_{U(t)} \ov{L} (\cdot, t) \leq (6 + \nu) (t_1 - t)$ for all $t \in (t^*, t_1)$.
We now argue that then also 
\begin{equation} \label{eq:infimumbound}
 \inf_{U(t)} \ov{L} (\cdot, t) \leq (6 + \nu) (t_1 - t) \qquad \text{for all} \qquad t \in [t^*, t_1). 
\end{equation}
Choose $t^{*\prime} \in (t^*, t_1)$ such that there is no surgery time in $(t^*, t^{*\prime}]$.
So $\MM$ restricted to $[t^*, t^{*\prime}]$ is a non-singular Ricci flow and hence $\ov{L}$ is continuous on this restriction (see for example \cite[Lemma 78.3(2)]{KLnotes}).
It follows that $\inf_{U(t)} \ov{L} (\cdot, t)$ is continuous on $[t^*, t^{*\prime}]$, proving (\ref{eq:infimumbound}).

So $\ov{L} (\cdot, t^*)$ attains its minimum at an interior point $x^* \in U(t^*)$.
This implies that $\triangle \ov{L} (x^* , t^*) \geq 0$.
Since $\ov{L} (x^*, t^*) \leq (6+\nu) (t_1 - t^*)$, we have $L (x^*, t^*) \leq (3 + \nu) \sqrt{t_1 - t^*} \leq 4 r_0$.
Hence by the Claim, assuming $T \geq T^*(4)$, we conclude $x^* \in D_{t^*}$ and $(x^*, t^*)$ is not a surgery point.
By the assumption on $t^*$, we must then have $t^* = t_0 -  \tau r_0^2$ or $\ov{L} (x^*, t^*) = (6+\nu) (t_1 - t^*)$ and $\frac{\partial}{\partial t} \ov{L} (x^*, t^*) \leq - 6 - \nu$, which, however, contradicts (\ref{eq:evolovLinU}).
So $\inf_{U(t)} \ov{L} (\cdot, t) \leq (6 + \nu) (t_1 - t)$ holds for all $\nu > 0$ and $t \in [t_0 - \tau r_0^2, t_1)$ and, by letting $\nu$ go to zero, we obtain a contradiction.

Consider now the case in which there is a $t \in [t_0 - \tau r_0^2, t_1)$ such that $\inf_{\partial U(t)} \ov{L} (\cdot, t) \linebreak[1] \leq 6 (t_1 - t)$.
Let $x \in \partial U (t)$ such that $\ov{L} (x, t) \leq 6 (t_1 - t)$, i.e. $L (x, t) \leq 3 \sqrt{t_1 - t} \leq 3 r_0$.
By concatenating an $\LL$-geodesic between $(x_1, t_1)$ and $(x, t)$ with a constant space-time curve on the time-interval $[t_0 - \tau r_0^2, t]$, we conclude, using (\ref{eq:curvboundon2Afortau}), assumption (iv) and the fact that $\tau \leq \tau_0 \leq 1$,
\[ L (x, t_0 - \tau r_0^2) \leq L (x, t) + C_1 K^*_1 r_0^{-2} \int_{ t_0 - \tau r_0^2}^{t} \sqrt{t_1 - t'} dt' 
 \leq 3 r_0  + C_1 K^*_1  r_0. \]
 Here $C_1 < \infty$ is a universal constant.
Thus, in both cases (i.e. in the case in which the infimum of $\ov{L}$ can be controlled on the boundary of $U$ as well as in the case in which it can be controlled everywhere on $U$), we can find some point $y \in U( t_0 -  \tau r_0^2 )$ such that $L(y, t_0 - \tau r_0^2) < C_2 r_0$ for some constant $C_2 = C_2 (w, A) < \infty$.
Observe that by (v) we have $y \in B( x_0, t_0 - \tau r_0^2, A r_0)$.
So by extending an $\LL$-geodesic between $(x_1, t_1)$ and $(y,  t_0 - \tau r_0^2)$ by a time-$(t_0 - \tau r_0^2)$ geodesic segment, we can conclude, using (\ref{eq:curvboundonAparnbdtauprime}), that there is a constant $C_3 = C_3(w, A, \tau') = C_3 (w, A) < \infty$ such that $L(x_0,  t_0 - (\tau + \frac12 \tau') r_0^2) < C_3 r_0$.

By the Claim, assuming $T \geq T^* (C_3)$, we find that there is a smooth minimizing $\LL$-geodesic $\gamma$ between $(x_1, t_1)$ and $(x_0,  t_0 - (\tau + \frac12 \tau') r_0^2)$ that does not hit any surgery points.
We now lift $\gamma$ to an $\LL$-geodesic $\td\gamma$ in $\td\MM$ starting from $(\td{x}_1,  t_1)$ and going backwards in time.
If there are no surgery times on the time-interval $[t_0 - (\tau + \frac12 \tau') r_0^2, t_1]$, then this is trivial.
If there are, then let $T^i$ be the last surgery time that is $\leq t_1$ and lift $\gamma$ on the time-interval $[T^i, t_1]$ to $\td\MM(T^i)$.
Note that $\td\gamma (T^i) \in \td{U}^i_+$, so we can use the diffeomorphism $\td{\Phi}^i$ from Definition \ref{Def:RFsurg} to determine the limit $\lim_{t \nearrow T^i} \td\gamma(t)$.
Starting from this limit point, we can lift $\gamma$ on the interval $[T^{i-1}, T^i)$ or $[t_0 - (\tau + \frac12 \tau') r_0^2), T^i)$ and continue the process until we reach time $t_0 - (\tau + \frac12 \tau') r_0^2$.
Let $\td{x}_0 = \td\gamma (t_0 - ( \tau + \frac12 \tau') r_0^2) \in \td\MM (t_0 - (\tau + \frac12 \tau') r_0^2)$ be the initial point of $\td\gamma$.
Then $\td{x}_0$ is a lift of $x_0$ and, by (\ref{eq:curvboundon2Afortau}) and assumption (ii), there is a $v_1 = v_1 (w) > 0$ such that
\[  \vol_{t_0 - (\tau + \tau') r_0^2} B(\td{x}_0, t_0 - (\tau + \tau') r_0^2, r_0) > v_1 r_0^3. \]

We consider now the functions $L^{\td\MM}, \ell^{\td\MM}$, the domains $D^{\td\MM}_t$ and the reduced volume $\td{V}^{\td\MM} (t)$ in $\td\MM$ based in $(\td{x}_1, t_1)$.
By concatenating $\td\gamma$ with time-$(t_0 - (\tau + \frac12 \tau') r_0^2)$ geodesic segments, we conclude, using the curvature bound in (\ref{eq:curvboundon2Afortau}), that there is some $C_4 = C_4 (w, A) < \infty$ such that
\[ L^{\td\MM} (x, t_0 - (\tau + \tau') r_0^2) < C_4 r_0 \qquad \text{for all} \qquad x \in B(\td{x}_0, t_0 - (\tau + \tau') r_0^2, r_0). \]
Again, using the Claim and assuming $T \geq T^*(C_4)$, we conclude that $B(\td{x}_0, t_0 - (\tau+ \tau') r_0^2, r_0) \subset D^{\td\MM}_{t_0 - (\tau + \tau') r_0^2}$.
So, together with the inequality $t_1 - (t_0 - (\tau + \tau') r_0^2) \geq \frac12 \tau' r_0^2$, this implies that there is some $v_2 = v_2(w, A) > 0$ such that
\[ \td{V}^{\td\MM} (t_0 - (\tau + \tau') r_0^2) > v_2. \]
This implies the noncollapsedness in $(\td{x}_1, t_1)$ as in the end of the proof of Lemma \ref{Lem:6.3a} (see also \cite[7.3]{PerelmanI}, \cite[Theorem 26.2]{KLnotes}, \cite[Lemma 4.2.3]{Bamler-diploma}).

\textit{Proof of assertion (b). \quad} The proof of this part follows along the lines of the proof of Lemma \ref{Lem:6.3bc}(a).
The main difference is, however, that instead of invoking Lemma \ref{Lem:6.3a} for the non-collapsing statement, we make use of assertion (a) of this proposition.

Observe that by (\ref{eq:curvboundon2Afortau}), (ii) and basic volume comparison, we can choose $\kappa = \kappa (w, A) > 0$ such that the $\kappa$-noncollapsedness from assertion (a) even holds for all $t \in [t_0 - (\tau + \tau') r_0^2, t_0]$.

Let $w, A$ be given and let $E = \max \{ \un{E}_{\varepsilon}, E_{\ref{Lem:kappasolCNA}}(\varepsilon) \}$ and $\eta = \min \{ \un\eta, \eta_{\ref{Lem:kappasolCNA}} \}$, where $\un{E}_{\varepsilon}$, $\un\eta$ are the constants from Proposition \ref{Prop:CNThm-mostgeneral} and $E_{\ref{Lem:kappasolCNA}}(\varepsilon)$, $\eta_{\ref{Lem:kappasolCNA}}$ are the constants from Lemma \ref{Lem:kappasolCNA}.

Assume first that the statement is false for some small $0 < \td\rho < (K_1^*)^{-1/2}$, i.e. there is a time $t \in [t_0 - \tau r_0^2, t_0]$ and a point $\td{x} \in \td{U} (t)$ such that $(x,t)$ does not satisfy the canonical neighborhood assumptions $CNA (\td\rho r_0, \varepsilon, E, \eta)$ in $\td\MM$.
In particular $|{\Rm}|(\td{x}, t) \geq \td\rho^{-2} r_0^{-2}$.

We now use a point picking argument to find a time $\ov{t} \in [t_0 - \tau r_0^2, t_0]$ and a point $\ov{x} \in \td{U} ( \ov{t} )$ that also doesn't satisfy the canonical neighborhood assumptions $CNA (\td\rho r_0, \varepsilon, E, \eta)$ in $\td\MM$ and that additionally satisfies the following condition:
Set $\ov{q} = |{\Rm}|^{-1/2} (\ov{x}, \ov{t}) \leq \td\rho r_0$.
Then for any $t' \in [t_0 - (\tau + \tau') r_0^2, \ov{t}]$, all points in $\td{U}(t')$ satisfy the canonical neighborhood assumptions $CNA (\frac12 \ov{q}, \varepsilon, E, \eta)$ in $\td\MM$.
The point $(\ov{x}, \ov{t})$ can be obtained as follows:
Set first $(\ov{x}, \ov{t}) = (\td{x}, t)$ and $\ov{q} = |{\Rm}|^{-1/2} (\ov{x}, \ov{t}) \leq \td\rho r_0$.
If $(\ov{x}, \ov{t})$ satisfies the desired properties, then we stop.
Otherwise, we find another time $\ov{t}' \in [t_0 - (\tau + \tau') r_0^2, \ov{t}]$ and point $\ov{x}' \in U( \ov{t}')$ such that $(\ov{x}', \ov{t}')$ does not satisfy the canonical neighborhood assumptions $CNA (\frac12 \ov{q}, \varepsilon, E, \eta)$ in $\td\MM$.
Since we have assumed that $\td{\rho}^{-2} > K_1^*$ and due to (\ref{eq:curvboundon2Afortau}), we actually have $\ov{t}' \in [t_0 - \tau r_0^2, t_0]$.
We can now replace $(\ov{x}, \ov{t})$ by $(\ov{x}', \ov{t}')$ and repeat the process.
This process has to come to a close, since in every step we decrease $\ov{q}$ by a factor of $2$ and $\td\MM$ satisfies the canonical neighborhood assumptions $CNA (\frac12 \un{r}_\eps (t_0) \sqrt{t_0}, \eps, \un{E}_\eps, \un\eta)$ on $[t_0 - \tau r_0^2, t_0]$.
So we obtain $(\ov{x}, \ov{t})$ with the desired properties.
We furthermore conclude from (\ref{eq:curvboundon2Afortau}) that
\begin{equation} \label{eq:xbarfarawayfromdellU}
\dist_{\ov{t}} (\ov{x}, \partial \td{U} (\ov{t})) > 2A r_0.
\end{equation}

We now assume that there are no uniform constants $\td\rho$ and $T$ such that assertion (b) holds.
Then for some given $w, A$, we can find a sequence $\td\rho^\alpha \to 0$ and a sequence of counterexamples $\td\MM^\alpha$, $U^\alpha$, $t_0^\alpha$, $r_0^\alpha$, $\tau^\alpha$, $\theta^\alpha$, $x_0^\alpha$ with $t_0^\alpha \to \infty$ and $t^\alpha_0 > T_{(a)} (w, A, \theta^\alpha)$ (here $T_{(a)}$ is the constant for which assertion (a) holds) such that there are times $t^\alpha \in [t^\alpha_0 - \tau^\alpha (r_0^\alpha)^2, t_0^\alpha]$ and points $\td{x}^\alpha \in \td{U}^\alpha (t^\alpha)$ that do not satisfy the canonical neighborhood assumptions $CNA (\td\rho^\alpha r_0^\alpha, \varepsilon, E, \eta)$ in $\td\MM$.
By the last paragraph, we find times $\ov{t}^\alpha \in [t_0^\alpha - \tau^\alpha (r_0^\alpha)^2, t_0^\alpha]$ and points $\ov{x}^\alpha \in \td{U}(\ov{t}^\alpha)$ such that $(\ov{x}^\alpha, \ov{t}^\alpha)$ does not satisfy the canonical neighborhood assumptions $CNA(\ov{q}^\alpha, \varepsilon, E, \eta)$ in $\td\MM$ with $\ov{q}^\alpha = |{\Rm}|^{-1/2} (\ov{x}^\alpha, \ov{t}^\alpha) \leq \td\rho^\alpha r_0^\alpha$, but for any $t' \in [\ov{t}^\alpha - (\tau^\alpha + \tau') (r_0^\alpha)^2, \ov{t}^\alpha]$, all points in $\td{U}^\alpha (t')$ satisfy the canonical neighborhood assumptions $CNA (\frac12 \ov{q}^\alpha, \varepsilon, E, \eta)$ in $\td\MM$.

Recall that by Proposition \ref{Prop:CNThm-mostgeneral} we must have
\[ \ov{q}^\alpha > \un{r}_{\varepsilon} (\ov{t}^\alpha) (\ov{t}^\alpha)^{1/2} \linebreak[1] \geq \tfrac12 \un{r}_{\varepsilon} (t_0^\alpha) (t^\alpha_0)^{1/2}. \]
Let $(x', t') \in \MM^\alpha$ be a surgery point with $t' \in [t_0^\alpha - \tau^\alpha (r_0^\alpha)^2, t_0^\alpha]$.
Then as in (\ref{eq:surgpointhighcurv}) we have by the choice of $\delta$
\begin{equation} \label{eq:surgptRmlargeU}
 |{\Rm}| (x', t') > c' \delta^{-2} (t') t^{\prime -1} \geq c' t' \un{r}_{\varepsilon}^{-2} (2 t') \geq \tfrac{c'}2 t_0^\alpha \un{r}_{\varepsilon}^{-2} (t^\alpha_0) \geq \tfrac{c'}8 (t_0^\alpha)^2 (\ov{q}^\alpha )^{-2}
\end{equation}
and hence $(\ov{q}^\alpha)^2 |{\Rm}| (x',t') > \frac{c'}8 (t_0^\alpha)^2 \to \infty$.
So, as in the proof of Lemma \ref{Lem:6.3bc}(a), we conclude, using Lemma \ref{Lem:shortrangebounds}, that there is a constant $c > 0$ such that for large $\alpha$ the parabolic neighborhood $P(\ov{x}^\alpha, \ov{t}^\alpha, c \ov{q}^\alpha, - c(\ov{q}^\alpha)^2)$ is non-singular and we have $|{\Rm}| < 8 (\ov{q}^\alpha)^{-2}$ there.

Again, as in the proof of Lemma \ref{Lem:6.3bc}(a), we choose $\tau^*_\alpha > 0$ maximal with the property that $\ov{t}^\alpha - \tau^*_\alpha (\ov{q}^\alpha )^2 \geq t_0^\alpha - (\tau^\alpha + \tau') (r_0^\alpha)^2$ and such that the point $(\ov{x}^\alpha, \ov{t}^\alpha)$ survives until time $\ov{t}^\alpha - \tau^*_\alpha (\ov{q}^\alpha)^2$.
After passing to a subsequence, we may assume that the limit $\tau^*_\infty = \lim_{\alpha \to \infty} \tau^*_\alpha \in [0, \infty]$ exists.
By the conclusion in the previous paragraph, we must have $\tau^*_\infty \geq c > 0$.
Recall that by (\ref{eq:xbarfarawayfromdellU}) we have $\dist_{\ov{t}^\alpha} (\ov{x}^\alpha, \partial \td{U}^\alpha ( \ov{t}^\alpha )) > 2 A r_0^\alpha$.
By (\ref{eq:curvboundon2Afortau}) and a distance distortion estimate in $B$, we obtain that $\dist_{t} (\ov{x}^\alpha, \partial \td{U}^\alpha ( t )) > b r_0^\alpha$ for all $t \in [ \ov{t}^\alpha - \tau^*_\alpha (\ov{q}^\alpha)^2, \ov{t}^\alpha]$ and some $b = b(w, A) > 0$ (actually we can choose $b = b(w) > 0$).
So for every $a < \infty$ we have $\dist_{t} (\partial \td{U}^\alpha (t), \ov{x}^\alpha) > a \ov{q}^\alpha$ for all $t \in [\ov{t}^\alpha - \tau^*_\alpha (\ov{q}^\alpha)^2, \ov{t}^\alpha]$ whenever $\alpha$ is sufficiently large.

So by assertion (a) of this proposition and the choice of $(\ov{x}^\alpha, \ov{t}^\alpha)$, there is a uniform constant $\kappa > 0$ such that:
For all $a < \infty$ and for sufficiently large $\alpha$ (depending on $a$) we have that for all $t \in [\ov{t}^\alpha - \tau^*_\alpha (\ov{q}^\alpha)^2, \ov{t}^\alpha]$ the points in the ball $B(\ov{x}^\alpha, t, a \ov{q}^\alpha)$ are $\kappa$-noncollapsed on scales $< r_0^\alpha$ and satisfy the canonical neighborhood assumptions $CNA(\frac12 \ov{q}^\alpha, \varepsilon, E, \eta)$.
Therefore, we can follow the reasoning of the proof of Lemma \ref{Lem:6.3bc}(a) and apply Lemma \ref{Lem:lmiitswithCNA} to the flows $\td\MM^\alpha$ restricted to $[\ov{t}^\alpha - \tau^*_\alpha (\ov{q}^\alpha)^2, \ov{t}^\alpha]$ and parabolically rescaled by $(\ov{q}^\alpha)^{-1}$.
We conclude that, after passing to a subsequence, we have convergence to a non-singular Ricci flow on the time-interval $(-\tau^*_\infty, 0]$, which has bounded curvature.
So there is a $K^*_2 < \infty$ such that that for all $0 < \tau^{**} < \tau^*_\infty$ and for large $\alpha$ (depending on $\tau^{**}$) we have $(\ov{q}^\alpha)^2 |{\Rm}|(\ov{x}^\alpha,t) < K_2^*$ for all $t \in [\ov{t}^\alpha - \tau^{**} (\ov{q}^\alpha)^2, \ov{t}^\alpha]$.
Using Lemma \ref{Lem:shortrangebounds} and  (\ref{eq:surgptRmlargeU}), we can find a constant $c'' > 0$ such that for large $\alpha$ the point $(\ov{x}^\alpha, \ov{t}^\alpha)$ survives until time $\ov{t}^\alpha - (\tau^{**} + 2c'') (\ov{q}^\alpha)^2$ and we have $(\ov{q}^\alpha)^2 |{\Rm}|(\ov{x}^\alpha,t) < 2 K_2^*$ for all $t \in [\ov{t}^\alpha - (\tau^{**} + 2c'') (\ov{q}^\alpha)^2, \ov{t}^\alpha]$.
If $\tau^*_\infty < \infty$, we may choose $\tau^*_\infty - c'' < \tau^{**} < \tau^*_\infty$ and conclude that for large $\alpha$ the point $(\ov{x}^\alpha, \ov{t}^\alpha)$ survives until time $\ov{t}^\alpha - (\tau^{*}_\infty + c'') (\ov{q}^\alpha)^2$ and we have $(\ov{q}^\alpha)^2 |{\Rm}|(\ov{x}^\alpha,t) < 2 K_2^*$ for all $t \in [\ov{t}^\alpha - (\tau^{*}_\infty + c'') (\ov{q}^\alpha)^2, \ov{t}^\alpha]$.
Since for large $\alpha$
\[ \ov{t}^\alpha - (\tau^*_\infty + c'') (\ov{q}^\alpha)^2 \geq t_0^\alpha - \tau^\alpha (r_0^\alpha)^2 - (\tau^*_\infty + c'') (\td\rho^\alpha r^\alpha_0)^2 > t_0^\alpha - (\tau^\alpha + \tau') (r_0^\alpha)^2, \]
this gives us a contradiction to the choice of the $\tau_\alpha^*$.
So $\tau^*_\infty = \infty$ and again Lemma \ref{Lem:lmiitswithCNA} yields that the pointed Ricci flows with surgery $(\MM^\alpha, (\ov{x}^\alpha, \ov{t}^\alpha))$ subconverge to a $\kappa$-solution after parabolically rescaling by $(\ov{q}^\alpha)^{-1}$.
Using Lemma \ref{Lem:kappasolCNA}, this yields a contradiction to the assumption that the points $(\ov{x}^\alpha, \ov{t}^\alpha)$ don't satisfy the canonical neighborhood assumptions $CNA(\ov{q}^\alpha, \varepsilon, E, \eta)$.

\textit{Proof of assertion (c). \quad}
The proof is similar to the proof of Proposition \ref{Prop:curvcontrolincompressiblecollapse}.
However, instead of using Lemma \ref{Lem:6.3bc}(a), we will invoke the canonical neighborhood assumptions from assertion (b), which are independent of the distance to $x_0$.
Choose $E$ and $\eta$ according to assertion (b) and set $K = \max \{ \td\rho^{-2} (w, A), E^2 K^*_1(w,A), E^2 \}$.
In the following we may assume without loss of generality that $A > 10$.

Note first that by (\ref{eq:curvboundon2Afortau}) we have $|{\Rm}| < K r_0^{-2}$ on $U(t_0 - \tau r_0^2)$.
Consider a time $t_1 \in [t_0 - \tau r_0^2, t_0]$ with the property that $U$ restricted to $[t_0 - \tau r_0^2, t_1]$ is non-singular and for which
\begin{equation} \label{eq:Rmless2K}
|{\Rm}| < 2 K r_0^{-2} \qquad \text{on} \qquad U(t) \qquad \text{for all} \qquad t \in [t_0 - \tau r_0^2, t_1].
\end{equation}
We will then show that we actually have
\begin{equation} \label{eq:RmlessK}
|{\Rm}| < K r_0^{-2} \qquad \text{on} \qquad U(t) \qquad \text{for all} \qquad t \in [t_0 - \tau r_0^2, t_1], 
\end{equation}
if $\ov{w}$ is chosen small enough depending on $w$ and $A$.
This fact and the observation that for every surgery point $(x',t')$ we have
\[ |{\Rm}| (x', t') > c' \delta^{-2} (t') t^{\prime -1} \geq c' \geq c' \theta^2 t_0 \cdot r_0^{-2} \]
will then imply assertion (c) for sufficiently large $T$, depending on $K$ and $\theta$.

So assume that $U$ restricted to $[t_0 - \tau r_0^2, t_1]$ is non-singular and that (\ref{eq:Rmless2K}) holds.
It suffices to prove the curvature bound in (\ref{eq:RmlessK}) for $t = t_1$, because otherwise we may replace $t_1$ by $t$.
Assume that there was a point $x_1 \in U(t_1)$ with $|{\Rm}|(x_1, t_1) \geq K r_0^{-2}$.
Then $x_1 \not\in B(x_0, t_1, 2 A r_0)$ by (\ref{eq:curvboundon2Afortau}).
So by assumption (iv) we have
\begin{equation} \label{eq:x1partialUdistance}
 \dist_{t_1} (x_1, \partial U (t_1)) > A r_0.
\end{equation}
Using (\ref{eq:Rmless2K}), the distance distortion estimates from Lemma \ref{Lem:distdistortion}(a) and assumption (v) we conclude that
\[ \dist_{t_1} ( x_1, \partial U (t_1)) < e^{4K r_0^{-2} (t_1 - t_0 + \tau r_0^2)} \dist_{t_0 - \tau r_0^2} (x_1, \partial U (t_0 - \tau r_0^2)) < e^{4K \tau_0} \cdot 2 A r_0. \]
So
\begin{equation} \label{eq:distx0x1byexp}
\dist_{t_1} (x_0, x_1) < \big( 2 A e^{4K \tau_0} + A \big) r_0. 
\end{equation}

Let $\td{x}_1 \in \td{U}(t_1) \subset \td\MM (t_1)$ be a lift of $x_1$.
By assertion (b) we know that $(\td{x}_1, t_1)$ satisfies the canonical neighborhood assumptions $CNA(\td{\rho} r_0, \td{\varepsilon}_0, E, \eta)$ in $\td\MM$.
Since $|{\Rm}| (\td{x}_1, t_1) \geq K r_0^{-2} \geq (\td{\rho} r_0)^{-2}$, the point $(\td{x}_1, t_1)$ has a canonical neighborhood in $\td\MM$.
Note that $|{\Rm}|(\td{x}_1,t_1) \geq K r_0^{-2} \geq E^2 K_1^* r_0^{-2}$ and by (\ref{eq:curvboundon2Afortau}) we have $|{\Rm}|(\td{x}_0, t_1) < K^*_1 r_0^{-2}$ for any lift of $\td{x}_0$ in $\td\MM (t_1)$.
So the very last case in the Definition \ref{Def:CNA} of the canonical neighborhood assumptions cannot  occur.
Therefore, $(\td{x}_1, t_1)$ is the center of an $\varepsilon$-neck or an $(\varepsilon, E)$-cap.
In the first case set $\td{x}_2 = \td{x}_1$ and in the second case let $\td{x}_2 \in \td\MM (t_1)$ be the center of an $\varepsilon$-neck that bounds this cap.
So in either case $(\td{x}_2, t_1)$ is the center of an $\varepsilon$-neck in $\td\MM (t_1)$ and
\begin{multline*}
 \dist_{t_1} (\td{x}_1, \td{x}_2) < E |{\Rm}|^{-1/2} (\td{x}_1, t_1) \leq E K^{-1/2} r_0 \leq r_0 \qquad \\ \text{and} \qquad E^{-2} K r_0^{-2} \leq |{\Rm}| (\td{x}_2, t_1) < 2 K r_0^{-2}.
\end{multline*}
(For the last inequality, we have used (\ref{eq:Rmless2K}) and the fact that $\td{x}_2 \in \td{U} (t_1)$, which follows from (\ref{eq:x1partialUdistance}), assuming that without loss of generality $A \geq 1$.)
Let $x_2 \in \MM(t_1)$ be the projection of $\td{x}_2$.
We can then apply Lemma \ref{Lem:neckhasfewquotients} and conclude that
\begin{equation} \label{eq:volisalwaysbounded}
 \vol_{t_1} B( x_2, \tfrac12 K^{-1/2} r_0 ) > \td{w}_0 \cdot \tfrac18 K^{-3/2} r_0^3.
\end{equation}

Next, note that by (\ref{eq:distx0x1byexp}) and the conclusion in the previous paragraph we have
\begin{equation} \label{eq:distx0x2byexp}
 \dist_{t_1} (x_0, x_2) < \big( 2A e^{4K\tau_0} + A + 1 \big) r_0.
\end{equation}
Also, using (\ref{eq:x1partialUdistance}) and assuming without loss of generality that $A > 2$, we find that $B( x_2, t_1, \tfrac12 K^{-1/2} r_0 ) \subset U (t_1)$.
Observe that by assumption (iv) any minimizing geodesic in $\MM (t_1)$ connecting $x_0$ with a point in $U(t_1)$ is contained in $B(x_0, t_1, 2Ar_0) \cup U(t_1)$ and recall that by (\ref{eq:curvboundon2Afortau}) and (\ref{eq:Rmless2K}) we have
\begin{equation} \label{eq:curvboundforvc}
 |{\Rm}| < 2K r_0^{-2} \qquad \text{on} \qquad B(x_0, t_1, 2A r_0) \cup U(t_1).
\end{equation}
Lastly, by (\ref{eq:curvboundon2Afortau}) and distance and volume distortion estimates, there are uniform constants $0 < \beta = \beta (w, A) < 1$ and $C_1 = C_1 (w, A) < \infty$, which only depend on $w$ and $A$, such that
\begin{equation} \label{eq:volatt1issmall} 
\vol_{t_1} B(x_0, t_1, \beta r_0) < C_1 \vol_{t_0} B(x_0, t_0, r_0). 
\end{equation}
We now apply volume comparison to deduce a lower bound on $\vol_{t_1} B(x_0, t_1, \beta r_0)$.
In order to do this, observe that, due to assumption (iv), every minimizing geodesic that connects $x_0$ with a point in $U(t_1)$ has to lie within $B(x_0, t_1, 2 A r_0) \cup U(t_1)$.
Moreover, by (\ref{eq:distx0x2byexp})
\[  B(x_2, t_1, \tfrac12  K^{-1/2} r_0) \subset B \big( x_0, t_1, (2A e^{4K\tau_0} + A + 2) r_0 \big) . \]
So, using the curvature bound (\ref{eq:curvboundforvc}), volume comparison along minimizing geodesics between $x_0$ and points in $B(x_2, t_1, \tfrac12  K^{-1/2} r_0)$ and using (\ref{eq:volisalwaysbounded}), we find
\begin{multline*} \vol_{t_1} B(x_0, t_1, \beta r_0) 
>  \frac{c_2 (w,A) \beta^3}{(2A e^{4 K \tau_0} + A + 2)^3} \vol_{t_1} B(x_2, t_1, \tfrac12  K^{-1/2} r_0) \\
 > c_2 (w,A) c_3(w,A) \td{w}_0 \cdot \tfrac18  K^{-3/2} r_0^3.
\end{multline*}
Here $c_2 = c_2(w,A), c_3 = c_3 (w,A) > 0$ are uniform constants, which only depend on $w$ and $A$.
Together with (\ref{eq:volatt1issmall}) and assumption (iii), this yields
\[ c_2 (w,A) c_3(w,A) \td{w}_0 \cdot \tfrac18 K^{-3/2}(w,A)  r_0^3 < C_1(w,A) \ov{w} r_0^3. \]
So for small enough $\ov{w}$, depending only on $w$ and $A$, we obtain a contradiction.

Thus with this choice of $\ov{w}$, the bound (\ref{eq:Rmless2K}) does indeed imply the bound (\ref{eq:RmlessK}).
As mentioned before, this implication proves the desired result.

\textit{Proof of assertion (d). \quad} Assertion (d) follows from assertion (c) by a distance distortion estimate or from (\ref{eq:distx0x1byexp}) in the previous proof.
\end{proof}

\subsection{Curvature control in large regions that are locally good everywhere} \label{subsec:curvcontrollocgoodeverywhere}
We will now show that if we only have \emph{local} goodness control within some distance to some geometrically controlled region and if we can guarantee this control on a time-interval of uniform size, then we can deduce a curvature bound, which is independent of this distance.

In this section, we will use the following notation:
Let $U \subset \MM$ be a sub-Ricci flow with surgery of $\MM$, $t$ be a time for which $U(t)$ is defined and $d \geq 0$.
Then we denote the time-$t$ $d$-tubular neighborhood around $\partial U (t)$ in $U(t)$ by $B^U (\partial U, t, d) = B( \partial U(t), t, d) \cap U(t)$.
The parabolic neighborhood $P^U (\partial U, t, d, \Delta t)$ is defined similarly.

\begin{Proposition} \label{Prop:curvboundnotnullinarea}
There is a continuous positive function $\delta : [0, \infty) \to (0, \infty)$ such that for every $w, \theta > 0$ and $A < \infty$ there are constants $K = K(w, A), \linebreak[1] T (w, \linebreak[1] A, \linebreak[1] \theta) < \infty$ such that the following holds: \\
Let $\MM$ be a Ricci flow with surgery on the time-interval $[0, \infty)$ that is performed by $\delta(t)$-precise cutoff and whose time-slices are closed and let $t_0 > T$.
Consider a sub-Ricci flow with surgery $U \subset \MM$ on the time-interval $[t_0 - r_1^2, t_0]$.
Let $r_1, r_0, b > 0$ be constants such that
\begin{enumerate}[label=(\roman*)]
\item $\theta \sqrt{t_0} \leq r_1 \leq r_0 \leq \frac12 \sqrt{t_0}$,
\item for all $x \in \partial U(t_0)$ we have $| {\Rm} | \leq A r_1^{-2}$ on $P (x, t_0, r_1, - r_1^2)$,
\item for every $t \in [t_0 - r_1^2, t_0]$ and $x \in B^U (\partial U, t, b)$, the point $x$ is locally $w$-good at scale $r_0$ and time $t$ or $|{\Rm}| (x,t) \leq A r_1^{-2}$.
\end{enumerate}
Then for every $t \in (t_0 - r_1^2, t_0]$ and $x \in B^U (\partial U, t, b)$ we have
\[ | {\Rm} |(x,t) < K \big( (b - \dist_t(\partial U(t), x))^{-2} + (t - t_0 + r_1^2)^{-1} \big). \]
\end{Proposition}
\begin{proof}
Let $\delta(t)$ be an arbitrary function that goes to zero as $t \to \infty$.
Then for sufficiently large $t$ (depending on $w$, $A$ and $\theta$), we can use Definition \ref{Def:precisecutoff}(3) and volume comparison to conclude that no surgery point of $\MM(t)$ is locally $w$-good at scale $r_0$ and the curvature at every surgery point satisfies $|{\Rm}| > A r_1^{-2}$.
So we can assume in the following that there are no surgery points in the space-time neighborhood
\[ B = \bigcup_{t \in (t_0 - r_1^2, t_0]} B^U (\partial U, t, b) . \]
Consider the function
\[ f \quad : 
\quad (x,t) \quad \longmapsto \quad | {\Rm} |(x,t) \big( (b - \dist_t(\partial U(t), x))^{-2} + (t - t_0 + r_1^2)^{-1} \big)^{-1}  \]
on $B$.
Since $B$ is free of surgery points, we find that $|{\Rm}|$ and hence $f$ is bounded on $B$ (by a non-universal constant).

In the following, we will bound the supremum $H$ of $f$.
Choose some $(x_1, t_1) \in B$ where this supremum is attained up to a factor of $2$, i.e. $f(x_1, t_1) > \frac12 H$ and set $Q = r_1^2 | {\Rm} |(x_1, t_1)$.
Observe that
\begin{equation} \label{eq:Qf12H}
Q > f (x_1, t_1) > \tfrac12 H .
\end{equation}
Now if $H \leq \max \{ 2, 2 A \}$, then we are done, assuming $K > \max \{ 2, 2A \}$.
So assume in the following that $H > \max \{ 2, 2 A \}$.
This implies, in particular, that $Q > f(x_1, t_1) > \frac12 H > \max \{ 1, A \}$ and hence by assumption (iii) that the point $x_1$ is locally $w$-good at scale $r_0$ and time $t_1$.
Moreover, by assumption (ii),
\begin{equation} \label{eq:farfromboundaryofU}
(x_1, t_1) \not\in P (x, t_0, r_1, -r_1^2) \qquad \text{for all} \qquad x \in \partial U(t_0).
\end{equation}

Set $d_1 = \dist_{t_1}(\partial U(t_1), x_1)$, $a = r_1^{-1} (b -d_1)$ and observe that
\begin{equation}  \label{eq:4QaH2}
 Q a^2 >  f(x_1, t_1) > \tfrac12 H > 1  \qquad \text{and} \qquad   Q (t_1 - t_0 + r_1^2) r_1^{-2} > f(x_1, t_1) > \tfrac12 H . 
\end{equation}
So for all $t \geq \frac14 (t_1 - t_0 + r_1^2) + t_0 - r_1^2$ and $x \in B^U (\partial U, t, d_1 + \frac12 a r_1)$ we have
\begin{multline} 
 | {\Rm} |(x,t) \leq H \big( (b - \dist_t (\partial U (t), x) )^{-2} + (t - t_0 + r_1^2)^{-1} \big) \\
\leq H \big( 4 a^{-2} r_1^{-2} + 4 (t_1 - t_0 + r_1^2)^{-1} \big) < 16 Q r_1^{-2}. \label{eq:8Qr1}
\end{multline}

For a moment fix some arbitrary $x \in B^U (\partial U, t_1 , d_1 + \frac14 a r_1)$ and choose $\Delta t > 0$ maximal with the property that $t_1 - \Delta t \geq \frac14 (t_1 - t_0 + r_1^2) + t_0 - r_1^2$ and $\dist_t(\partial U (t), x) < d_1 + \frac38 a r_1$ for all $t \in (t_1 - \Delta t, t_1]$.
We will now estimate the distance distortion between $x$ and any point $x_0 \in \partial U$ using Lemma \ref{Lem:distdistortion}(b).
Using (\ref{eq:8Qr1}) we find that for all $t \in [t_1 - \Delta t, t_1]$ we have $| {\Rm} | < 16 Q r_1^{-2}$ on $B(x, t, \frac18 a r_1) \cap U(t) \subset B^U ( \partial U, t, d_1 + \frac12 a r_1)$.
Moreover, by (\ref{eq:4QaH2}), we have $\frac18 Q^{-1/2} r_1 < \frac18 a r_1$.
By assumption (ii) and distance distortion estimates, we can also find a uniform $0 < \beta = \beta(A) < \frac12$ such that for all $t \in [t_1 - \Delta t, t_1]$ and all $y \in \partial U(t)$ we have $B(y, t, \beta r_1) \subset B(y, t_0, r_1)$ and $|{\Rm}| < \beta^{-2} r_1^{-2}$ on $B(y, t, \beta r_1)$.
So for all $t \in [t_1 - \Delta t, t_1]$ we have
\begin{multline*}
 |{\Rm}| < \big( {\min \{ \tfrac18 Q^{-1/2} , \beta \}} \big)^{-2} r_1^{-2} \qquad \\ \text{on} \qquad B(x_0, t, \beta r_1) \cup B \big( x, t, \min \{ \tfrac18 Q^{-1/2}, \beta \} r_1 \big)
\end{multline*}
and thus Lemma \ref{Lem:distdistortion}(b) yields
\[ \frac{d}{dt} \dist_t (x_0, x) > - C \big( { \min \{ \tfrac18 Q^{-1/2}, \beta \} } \big)^{-1} r_1^{-1} \]
for some universal constant $C < \infty$.
This implies that for all $t \in [t_1 - \Delta t, t_1]$ we have
\[ \dist_t (x_0, x) \leq d_{t_1} (x_0, x) + C \big( { \min \{ \tfrac18 Q^{-1/2} , \beta \} } \big)^{-1} r_1^{-1} (t_1 - t) \]
for all $t \in [t_1 - \Delta t, t_1]$.
Letting $x_0$ vary over $\partial U$ yields
\[ \dist_t (\partial U (t), x) \leq d_1 + \tfrac14 a r_1 + C \big({  \min \{ \tfrac18 Q^{-1/2} , \beta \} } \big)^{-1} r_1^{-1} (t_1 - t) \]
So, by the definition of $\Delta t$ and using (\ref{eq:4QaH2}) and (\ref{eq:Qf12H}), we obtain (recall that $Q > \frac12 H > 1$)
\begin{multline*}
 \Delta t \geq \min \Big\{ \frac{\tfrac18  a r_1}{C \big({  \min \{ \tfrac18 Q^{-1/2} , \beta \} } \big)^{-1} r_1^{-1}}, \tfrac34 (t_1 - t_0 + r_1^2) \Big\} \\
  > c' \min \big\{ a Q^{-1/2}, a \beta, H Q^{-1} \big\} r_1^2 \\
> c \min \big\{  H^{1/2}  Q^{-1}, H^{1/2} Q^{-1/2}, H Q^{-1} \big\} r_1^2 = c H^{1/2} Q^{-1} r_1^2
\end{multline*}
for some universal $c' > 0$ and some $c = c( A) > 0$.
Note that $x \in B^U (\partial U, t_1 , d_1 + \frac14 a r_1)$ was chosen arbitrarily.
So by the choice of $\Delta t$, we find that for any such $x$ and any $t \in [t_1 - c H^{1/2} Q^{-1} r_1^2, t_1]$ we have 
\begin{equation} \label{eq:distanceboundUt}
 \dist_t (\partial U(t), x) < d_1 + \tfrac38 a r_1 < d_1 + \tfrac12 a r_1.
\end{equation}
Moreover, by (\ref{eq:8Qr1}), we conclude that $| {\Rm} | < 16 Q r_1^{-2}$ on $P ' = P^U (\partial U, t_1, d_1 + \frac14 a r_1, - c H^{1/2} Q^{-1} r_1^2)$.
So, in particular, $P'$ does not contain surgery points.
Moreover, by (\ref{eq:distanceboundUt}), we have
\[ P' \subset \bigcup_{t \in [t_1 - c H^{1/2} Q^{-1} r_1^2, t_1]} B^U \big( \partial U, t, d_1 + \tfrac12 a r_1 \big). \]

By the $t^{-1}$-positivity of the curvature on $\MM$, we have $\sec \geq - F (Q r_1^{-2} t_0 ) Q r_1^{-2}$ on $P'$, where $F : [0, \infty) \to [0, \infty)$ is a decreasing function that goes to zero on the open end.
Since $F(Q r_1^{-2} t_0) \leq F(4Q) \leq F(H)$, we have the bound $\sec \geq - F(H) Q r_1^{-2}$ on $P'$.
Next, using (\ref{eq:farfromboundaryofU}) and using the constant $0 < \beta = \beta(A) < 1$ from before, we get that $\dist_{t_1} (\partial U (t_1), x_1) > \beta r_1$.
Then
\[ P \big( x_1, t_1, \min \{ \beta, \tfrac14 a \} r_1, - c H^{1/2} Q^{-1} r_1^2 \big) \subset P'. \]
Define $S : (0, \infty) \to (0, \infty)$ by $S(x) = \min \{ F^{-1/2} (x), \frac18 x^{1/2}, \frac12 \beta  x^{1/2}, c^{1/2} x^{1/4} \}$.
Then $S(x) \to \infty$ as $x \to \infty$ and we find, using (\ref{eq:4QaH2}) and (\ref{eq:Qf12H}), that
\begin{alignat*}{1}
 \tfrac14 a &> \tfrac18 H^{1/2} Q^{-1/2} \geq S(H) Q^{-1/2}, \\
 \beta &\geq \tfrac12 \beta H^{1/2} Q^{-1/2} \geq S (H) Q^{-1/2}, \\
 c H^{1/2} Q^{-1} &\geq S^2(H) Q^{-1}. 
\end{alignat*}
This yields the bound
\[ \sec \geq - S^{-2} (H) Q r_1^{-2} \qquad \text{on} \qquad P(x_1, t_1, S(H) Q^{-1/2} r_1, -S^2(H) Q^{-1} r_1^2). \]
In particular $\rho_{r_0} (x_1, t_1) \geq  S(H) Q^{-1/2} r_1$ (observe that $S(H) Q^{-1/2} r_1 \leq \beta r_1 \leq r_0$).

So by property (iii), we conclude that for $r = S(H) Q^{-1/2} r_1$ we have $\vol_{t_1} \widetilde{B}(\td{x}_1, \linebreak[1] t_1, \linebreak[1] r) > \td{c} w r^3$, where $\td{B} (\td{x}_1, t_1, r) $ denotes the universal cover of the ball $B(x_1, t_1, r)$.
We can now lift the flow on $P(x_1, t_1, r, -r^2)$ to this universal cover, rescale it parabolically by $r^{-1}$ and use Lemma \ref{Lem:6.5} to obtain
\[ Q r_1^{-2} = |{\Rm}| (x_1, t_1) < K_{0, \ref{Lem:6.5}} (\td{c} w) \tau_{0, \ref{Lem:6.5}}^{-1}( \td{c} w) r^{-2} = K_{0, \ref{Lem:6.5}}  \tau_{0, \ref{Lem:6.5}}^{-1} S^{-2} (H) Q r_1^{-2}. \]
Here $K_{0, \ref{Lem:6.5}},  \tau_{0, \ref{Lem:6.5}}$ are the constants from Lemma \ref{Lem:6.5}.
The last inequality implies $S^2(H) < K_{0, \ref{Lem:6.5}}  \tau_{0, \ref{Lem:6.5}}^{-1}$, which in turn implies that $H$ is bounded by some universal constant $H_0 = H_0(w, A) < \infty$.
This finishes the proof.
\end{proof}

\end{document}